\documentclass{article}
  \PassOptionsToPackage{numbers, sort&compress}{natbib}
  \usepackage[preprint]{neurips_2026}

\usepackage[utf8]{inputenc} %
\usepackage[T1]{fontenc}    %
\usepackage{hyperref}       %
\usepackage{url}            %
\usepackage{booktabs}       %
\usepackage{amsfonts}       %
\usepackage{nicefrac}       %
\usepackage{microtype}      %
\usepackage{xcolor}         %
\usepackage{titlesec}
\titlespacing{\section}{0pt}{5pt}{2pt}
\titlespacing{\subsection}{0pt}{4pt}{2pt}

\usepackage{amsmath,amssymb,amsthm,mathtools}
\usepackage{algorithm}
\usepackage{algorithmic}
\usepackage{wrapfig}
\usepackage{bm}
\usepackage{comment}
\usepackage{xcolor}
\usepackage{placeins}
\usepackage{xspace}
\usepackage{pgfplots}
\usepgfplotslibrary{groupplots}
\usepackage{caption}
\usepackage{subcaption}
\pgfplotsset{compat=1.17}
\usepackage{enumitem}
\usepackage{hyperref}
\usepackage{cleveref}

\setcitestyle{numbers,square}
\newtheorem{theorem}{Theorem}
\newtheorem{lemma}{Lemma}
\newtheorem{assumption}{Assumption}
\newtheorem{proposition}{Proposition}

\newtheorem{remark}{Remark}
\newtheorem{definition}{Definition}
\newtheorem{example}{Example}

\crefname{theorem}{Thm.}{Thms.}
\Crefname{theorem}{Theorem}{Theorems}
\crefname{lemma}{Lem.}{Lems.}
\Crefname{lemma}{Lemma}{Lemmas}
\crefname{proposition}{Prop.}{Props.}
\Crefname{proposition}{Proposition}{Propositions}
\crefname{corollary}{Cor.}{Cors.}
\Crefname{corollary}{Corollary}{Corollaries}
\crefname{remark}{Rem.}{Rems.}
\Crefname{remark}{Remark}{Remarks}
\crefname{definition}{Def.}{Defs.}
\Crefname{definition}{Definition}{Definitions}
\crefname{algorithm}{Alg.}{Algs.}
\Crefname{algorithm}{Algorithm}{Algorithms}
\crefname{assumption}{Assumption}{Assumptions}
\Crefname{assumption}{Assumption}{Assumptions}
\crefname{example}{Expl.}{Expls.}
\Crefname{example}{Example}{Examples}
\crefname{figure}{Fig.}{Figs.}
\Crefname{figure}{Figure}{Figures}
\crefname{section}{Sec.}{Secs.}
\Crefname{section}{Section}{Sections}

\newcommand{\mc}[1]{\mathcal{#1}}
\newcommand{\mb}[1]{\mathbb{#1}}
\renewcommand{\paragraph}[1]{\par\noindent\textbf{#1}\hspace{0.7em}\ignorespaces}

\DeclareMathOperator{\E}{\mathbb{E}}
\DeclareMathOperator{\R}{\mathbb{R}}
\DeclareMathOperator{\X}{\mathcal{X}}
\DeclareMathOperator{\Y}{\mathcal{Y}}
\DeclareMathOperator{\U}{\mathcal{U}}
\DeclareMathOperator{\co}{conv}
\DeclareMathOperator{\Tr}{Tr}

\DeclareMathOperator{\vol}{Vol}

\DeclareMathOperator{\ball}{\mathbb{B}}
\DeclareMathOperator{\sphere}{\mathbb{S}}
\DeclareMathOperator{\dx}{\mathrm{d_x}}
\DeclareMathOperator{\dy}{\mathrm{d_y}}
\DeclareMathOperator{\goldstein}{\partial_\delta}

\DeclareMathOperator{\goldsteinfull}{\partial_\delta^{\text{G}}} %
\DeclareMathOperator{\goldsteinpartial}{\partial_\delta^{\text{P}}}
\DeclareMathOperator{\goldsteinymu}{\partial_\mu^{\text{G}}}
\DeclareMathOperator{\goldsteinfullmu}{\partial_\mu^{\text{G}}}
\DeclareMathOperator{\goldsteinpartialmu}{\partial_\mu^{\text{P}}}

\newcommand{\norm}[1]{\|#1\|}
\newcommand{\inner}[2]{\langle #1, #2 \rangle}
\newcommand{\grad}{\nabla}
\newcommand{\jac}{J}%
\newcommand{\algpzos}{\hyperref[alg:partial_zero_order]{\texttt{PZOS}}\xspace}
\newcommand{\algzos}{\hyperref[alg:zero_order]{\texttt{ZOS}}\xspace}

\title{Black-Box Followers, White-Box Leaders:\\ Partial Zeroth-Order Methods for MPECs}

\author{%
  Miriam Fischer \\
  Department of Computing\\
  Imperial College London\\
  London, United Kingdom \\
  \texttt{m.fischer21@imperial.ac.uk} \\
  \And
  Dario Paccagnan \\
  Department of Computing\\
  Imperial College London\\
  London, United Kingdom \\
  \texttt{d.paccagnan@imperial.ac.uk} 
}

\begin{document}

\maketitle
\renewcommand{\thefootnote}{}
\footnotetext{This work was partially supported by: the EPSRC grant EP/Y001001/1, funded by the International Science Partnerships Fund (ISPF) and UKRI; an Imperial-MIT seed fund; a Google DeepMind Scholarship awarded to Miriam Fischer.}
\setcounter{footnote}{0}
\renewcommand{\thefootnote}{\arabic{footnote}}

\begin{abstract}
We study mathematical programs with equilibrium constraints, in which a leader knows their own cost function, but lacks a model of the followers' response. Instead, the leader can only query this response at specific points. While this setting precludes the use of gradient-based methods, existing zeroth-order approaches treat the composed objective \emph{entirely} as a black box, deploying zeroth-order tools across both the leader and follower. %
Such approaches are inefficient, as they discard information the leader already possesses about their own cost function.
In this work we instead propose to deploy zeroth-order tools \emph{only} where they are truly needed: to handle the unknown, non-smooth followers' response. Specifically, we first propose \algpzos, an algorithm that combines exact partial gradients of the leader's cost with zeroth-order Jacobian estimates of the followers' response in a chain-rule-inspired manner, and establish that it achieves a strictly lower variance bound than the black-box baseline. Second, we introduce the partial Goldstein subdifferential, a stationarity notion tailored to this composite structure, and prove convergence of our algorithm to both standard and partial Goldstein stationary points.  Finally, we validate our method on two application domains -- toll optimization in routing games and defense-attack investment in security games -- demonstrating consistent improvements over black-box baselines in convergence speed, objective value, and estimator variance, with robust performance even under few queries per iteration.
\end{abstract}
\section{Introduction}
Mathematical programs with equilibrium constraints (MPECs) arise naturally in settings where a leader optimizes their own objective while anticipating the equilibrium response of one or more followers. In many practical instances -- toll optimization in transportation networks, pricing in competitive markets, resource allocation in multi-agent systems -- the leader has a precise understanding of their own cost, but possesses no analytical model of how followers will respond. Instead, the followers' response can only be observed by querying it at specific points. For instance, in toll optimization, the toll-setter knows their revenue objective precisely but has no analytical model of how drivers respond to toll changes.

Bilevel optimization problems and MPECs are commonly tackled via gradient-based methods. These methods exploit knowledge of the followers' model to obtain gradient information on the composed objective -- for instance by differentiating through the lower-level optimality conditions~\cite{amos2017optnet, grontas2024bighype}. While highly effective, such methods are not applicable in the setting considered, where the leader has no access to the followers' model and can only observe their response through queries.

Zeroth-order methods are particularly well-suited to this setting~\cite{lin2022goldstein,cui2023mpec,kornowski2024goldsteinoptimal,liu2024goldsteinconstrained,chen2023goldsteinimproved}. Specifically, they sidestep the need for derivatives by constructing gradient estimates from sole function evaluations, and provide convergence certificates to approximate stationary points. However, existing approaches tailoring these tools to MPECs apply zeroth-order techniques to the entire composite objective -- the leader's cost evaluated at the followers' response -- treating it as a black box, even though the leader's own cost is smooth and its partial gradients are available. Such approaches discard structural information the leader already possesses, resulting in possibly high variance and slow convergence.

Against this backdrop, in this work we propose an algorithm for MPECs that employs zeroth-order tools \emph{only} where they are truly needed, namely to the followers' unknown, non-smooth response. In particular, our approach employs a chain-rule-inspired algorithm that keeps the leader's partial gradients exact, and uses zeroth-order methods solely to estimate the Jacobian of the followers' response.
Two technical challenges arise. First, since the followers' response is not available in closed form, partial gradients of the leader's cost and the Jacobian estimate of the followers' response must be evaluated at different points, rendering the resulting estimator \emph{biased}. Second, owing to the non-differentiability of the followers' response, convergence to standard notions of stationarity is not achievable, and one must resort to the more general notion of $(\delta,\varepsilon)$-Goldstein stationarity. We resolve both challenges and make the following three contributions.

\begin{enumerate}[label=C\arabic*,ref=C\arabic*,leftmargin=*]
\item \label{contributionconvergence}
\textbf{Algorithm exploiting upper-level structure.} We propose \algpzos, an algorithm that combines exact partial gradients of the leader's cost with zeroth-order Jacobian estimates of the followers' response in a chain-rule-inspired manner. Unlike \algzos~-- the black-box method that applies zeroth-order tools to the entire composed objective --, \algpzos applies zeroth-order tools only to the followers' response, fully exploiting knowledge of the leader's cost. We prove that this yields a strictly lower variance bound than the black-box baseline \algzos.
\item \label{contributiongoldstein}
\textbf{Convergence guarantees.} We prove convergence of \algpzos to both $(\delta,\varepsilon)$-Goldstein stationary points and $(\delta,\varepsilon)$-partial Goldstein stationary points of the composed objective. The latter notion, which we introduce here, is tailored to composite objectives where the leader's cost is smooth. Unlike the standard Goldstein subdifferential, which allows variation in all components of the composed objective over a neighborhood, the partial variant fixes the partial gradients of the leader's cost at the current point and allows variation only in the Jacobian of the followers' response. We provide a discussion of both notions in Appendix~\ref{sec:goldstein}. %

\item \label{contributionexperiments}
\textbf{Empirical validation.} 
Across 225 routing game instances and 427 security game instances (\Cref{sec:experiments}), \algpzos consistently outperforms the baseline \algzos in convergence speed, final objective value, and variance. Notably, \algpzos achieves robust performance from a single two-point evaluation per iteration, whereas \algzos requires averaging over multiple two-point evaluations to attain comparable results. This translates to substantially fewer oracle calls for equivalent performance.
\end{enumerate}

{The remainder of the paper is organized as follows. In \cref{sec:relatedwork}, we discuss related work. In \cref{sec:problemsetup}, we formalize the problem setting. In \cref{sec:PZOS} we present \algpzos. In \cref{sec:mainresult} we introduce the notion of (partial) Goldstein stationarity and state our main convergence result. We then validate our approach on routing and security games in \cref{sec:experiments}. Proofs and additional details are deferred to the appendix.}

\subsection{Related Work}\label{sec:relatedwork}
\noindent\textbf{Zeroth-order methods and Goldstein stationarity.}
Zeroth-order optimization addresses settings where only function evaluations are available~\cite{nesterov2017spokoiny}. Interestingly, in the non-smooth non-convex setting of interest to this work, convergence to Clarke stationary points is not achievable in finite time, and even convergence to near Clarke stationary points requires exponentially many queries~\cite{zhang2020clarkestationary, kornowski2022nearclarkestationary}. This has motivated the study of the  relaxed notion of $(\delta,\varepsilon)$-Goldstein stationarity. 
\citet{lin2022goldstein} first establish convergence to $(\delta,\varepsilon)$-Goldstein stationary points via zeroth-order smoothing%
; \citet{chen2023goldsteinimproved} improve the sample complexity %
via variance reduction, while \citet{liu2024goldsteinconstrained} extend to constrained settings. Most recently, \citet{kornowski2024goldsteinoptimal} provide an algorithm with optimal convergence rate. %
 In this context, our work aims to provide better practical performance while retaining provable convergence guarantees, as opposed to improving on these best-possible convergence rates.

\noindent\textbf{Zeroth-order methods for MPECs and bilevel optimization.} {Classical approaches to solving MPECs, e.g., {gradient-based methods~\cite{grontas2024bighype}}, interior-point methods \cite{raghunathan2005mpccinterior,benson2006interior}, sequential quadratic programming \cite{qi2000sqp,fletcher2006sqp}, and penalty/relaxation methods \cite{demiguel2005mpecrelaxation,benson2006interior,hoheisel2013mpecrelaxation} require explicit knowledge of the followers' model and are inapplicable in our setting, where the leader can only query the followers' response. Existing zeroth-order methods for MPECs~\citep{cui2023mpec,cui2023hierarchicalgames,tao2025zobilevel} reformulate the MPEC as a single-level problem and apply zero order methods to the entire composite objective, without distinguishing between the leader's known, smooth cost and the unknown, non-smooth followers' response $y^*$ -- precisely the inefficiency our work addresses}. Perhaps the closest work to ours is \citet{sow2022pzobo}, who smooth only the inner solution mapping in bilevel optimization. However, their setting differs in two key aspects. First, the lower-level problem they consider is unconstrained which allows Gaussian smoothing. Second, they assume the lower-level response to be differentiable which results in convergence guarantees to standard stationary points. Neither carries over to our setting, where the constrained lower level problem requires feasibility-preserving smoothing, and its non-differentiability necessitates the $(\delta,\varepsilon)$-Goldstein stationarity framework.
{\section{Problem Formulation}\label{sec:problemsetup}

We consider mathematical programs with equilibrium constraints, where a leader aims at minimizing a cost function $f(x,y)$, with $f: \R^{\dx} \times \R^{\dy} \to \R$, over their decision variable $x$, while anticipating the response $y$ of one or more followers. We model the followers' response as the  solution to a parametric variational inequality: given the leader's decision $x$, the followers seek $y \in \Y$ such that $\langle G(x, y), y' - y \rangle \geq 0$ for all $y' \in \Y,$ where $G: \R^{\dx} \times \R^{\dy} \to \R^{\dy}$ and $\Y$ is closed and convex. 

This formulation captures a broad range of equilibrium notions,  including Nash equilibria, making it well-suited to the applications. Throughout this manuscript, we assume that the parametric variational inequality admits a unique solution for each given $x$, which we denote $y^*(x)$. This is a minimal assumption, satisfied in a broad range of problems of interest, including Cournot competition~\cite{cui2023mpec}, traffic routing~\cite{yang2004formulation}, and security games~\cite{elkind2024csfunique,iliaev2023csf}. Crucially, since the lower level problem is constrained by $y\in\mc{Y}$, the response $y^*(x)$ is in general non-differentiable. The problem of interest thus reads as
\begin{equation}\label{eq:setting}
    \min_{x \in \R^{\dx},\, y \in \Y} \; f(x, y) \quad \text{subject to} \quad 
    \langle G(x, y), y' - y \rangle \geq 0 \quad \forall y' \in \Y.
\end{equation}
 
We work under two standard assumptions. The first is a bounded suboptimality condition, common in zeroth-order optimization~\cite{lin2022goldstein, kornowski2024goldsteinoptimal}. The second collects regularity conditions: Lipschitz continuity and differentiability of $f$ are standard in this literature~\cite{cui2023mpec, lin2022goldstein}, with the latter enabling exploitation of the leader's partial gradients. For $y^*$, we assume only Lipschitz continuity. Differentiability is not required and does not hold in general due to the presence of the constraint $y \in \Y$.
 
\begin{assumption}[Bounded suboptimality]\label{ass:bounded}
{Given an initial point $x_0\in\R^{\dx}$, there exists a constant $\Delta > 0$ such that $f(x_0, y^*(x_0)) - \inf_{x} f(x, y^*(x)) \leq \Delta$.}
\end{assumption}
\vspace*{1mm}

{\begin{assumption}[Regularity]\label{ass:lipschitz}~
\begin{enumerate}[leftmargin=*, itemsep=-2pt, topsep=0pt]
\item $f$ is $L_f$-Lipschitz with bounded partial derivatives 
$\norm{\grad_x f}, \norm{\grad_y f}\leq L_f$.
\item $f$ is differentiable with $L_g$-Lipschitz gradient.
\item $y^*(x)$ is $L_y$-Lipschitz continuous in $x$.
\end{enumerate}
\end{assumption}}

\section{Our Partial Gradient Zeroth-order Algorithm}
\label{sec:PZOS}

Existing approaches to solve~\eqref{eq:setting} treat the composed objective $F(x) = f(x, y^*(x))$ as a black box and apply zeroth-order methods to it directly~\cite{cui2023mpec}. The key idea is to construct the smooth surrogate \[\tilde F_\mu(x) \doteq \E_{u \sim \U(\ball^{\dx})}[F(x + \mu u)],\] where $\U(\ball^{\dx})$ denotes the uniform distribution over the unit ball $\{u \in \R^{\dx} : \|u\| \leq 1\}$, obtained by averaging $F$ over random perturbations of magnitude at most $\mu > 0$. This surrogate can be shown to have two crucial properties: it is smooth, enabling gradient descent, and its gradient admits an unbiased estimator via finite differences along randomly sampled directions. Specifically, the black-box approach estimates the gradient of $\tilde F_\mu$ via the two-point finite-difference estimator~\cite{shamir2017boundestimator}
\begin{equation}\label{eq:twopointzoestimate}
    \tilde g_t = \frac{\dx}{2\mu}\bigl[f(x_t + \mu v_t, y^*(x_t + \mu v_t)) - 
    f(x_t - \mu v_t, y^*(x_t - \mu v_t))\bigr]v_t,
\end{equation}
where $v_t \sim \U(\mathbb{S}^{\dx})$ is sampled uniformly from the unit sphere $\{v \in \R^{\dx} : \|v\| = 1\}$, which requires only evaluations of $f$ and $y^*$. Since the approximation error $|\tilde F_\mu(x) - F(x)|$ can be controlled by $\mu$, convergence on $\tilde F_\mu$ translates into guarantees on the original objective $F$ -- specifically to $(\delta,\varepsilon)$-Goldstein stationary points, as the gradient of $\tilde F_\mu$ lies in the Goldstein subdifferential of $F$ for $\mu \leq \delta$~\cite{lin2022goldstein}. However, this black-box treatment discards the available partial gradient information on the leader's cost $f$ -- precisely the inefficiency our work addresses.

We instead propose to replace \emph{only} the followers' response $y^*(x)$ with its smooth surrogate, as this is the only component for which gradients are unavailable, while keeping the leader's cost $f$ intact. Specifically, for a fixed parameter $\mu > 0$, we define the smoothed response
\begin{equation}\label{eq:ysmooth}
y_\mu(x) \doteq \E_{u \sim \U(\ball^{\dx})}[y^*(x + \mu u)],\quad\text{ and correspondingly }\quad F_\mu(x) \doteq f(x, y_\mu(x)).
\end{equation}
Unlike the black-box smooth surrogate, $F_\mu$ does not replace the entire objective, but evaluates the function $f$ at the smoothed response $y_\mu$, allowing to exploit the partial gradients $\grad_x f$ and $\grad_y f$.

Building on these ideas, our algorithm (\cref{alg:partial_zero_order}) proceeds as follows. At each iteration $t$, we sample a random direction $v_t$ from the unit sphere  $\U(\sphere^{\dx})$ and compute an estimate $H_t$ of the Jacobian $\jac y_\mu(x_t)$ via finite differences. Note that $\jac y_\mu(x)$ is a matrix, since $y^*$ maps into $\R^{\dy}$ rather than $\R$, and is the only quantity requiring a zeroth-order estimate:
\[
H_t = \frac{\dx}{2\mu}\bigl(y^*(x_t + \mu v_t) - y^*(x_t - \mu v_t)\bigr) 
v_t^\top.
\]
Combining this with exact partial gradients of $f$ at $(x_t, y^*(x_t))$ gives the direction
\begin{equation}\label{eq:estimatedirection}
    g_t = \grad_x f(x_t, y^*(x_t)) + H_t^\top \grad_y f(x_t, y^*(x_t)),
\end{equation}
and we update $x_{t+1} = x_t - \alpha g_t$.

\begin{algorithm}[H]
\caption{Partial Zero-Order Smoothing (PZOS)}
\begin{algorithmic}\label{alg:partial_zero_order}
\STATE \textbf{Input:} Initial $x_0 \in \R^{\dx}$, step size $\alpha$, 
smoothing parameter $\mu>0$, iterations $T$
\FOR{$t = 0, 1, \ldots, T-1$}
\STATE Sample $v_t \sim \U({\sphere^{\dx}})$ 
  \hfill\textbackslash\textbackslash  {\small~Random direction uniformly from unit sphere}
\STATE Evaluate $y^*(x_t)$, $y^*(x_t + \mu v_t)$, $y^*(x_t - \mu v_t)$ 
\STATE $H_t = \frac{\dx}{2\mu}\bigl[(y^*(x_t + \mu v_t) - y^*(x_t - \mu 
    v_t))\,v_t^\top \bigr]$ 
\hfill\textbackslash\textbackslash  {\small~Compute Jacobian estimate}
\STATE $g_t = \grad_x f(x_t, y^*(x_t)) + H_t^\top \grad_y f(x_t, y^*(x_t))$ 
\hfill\textbackslash\textbackslash  {\small~Compute gradient estimate}
\STATE $x_{t+1} = x_t - \alpha g_t$ 
\hfill\textbackslash\textbackslash  {\small~Update step (minimization); for maximization, use $+$}
\ENDFOR
\STATE \textbf{Output:} $x^R$ where $R\in\{0,1,\dots,T-1\}$ is uniformly sampled
\end{algorithmic}
\end{algorithm}
\vspace*{-3mm}
We conclude by presenting the first upshot of our approach, before moving to the main convergence result. Specifically, we show that $g_t$ achieves a strictly lower second-moment bound than its black-box counterpart $\tilde g_t$, with the gap growing with the problem dimension \mbox{$\dx$. This is formalized next.}%
\begin{wrapfigure}[12]{r}{0.34\linewidth}
  \centering
  \vspace*{-9mm} 
  \definecolor{mycolor1}{rgb}{0.00000,0.35000,0.75000}%

\begin{tikzpicture}

\begin{axis}[%
height=2.8cm,
width=0.9\linewidth,
scale only axis,
xmin=0,
xmax=300,
xlabel style={font=\color{white!15!black}},
xlabel={$\text{Problem dimension d}_\text{x}$},
ymin=7.70056329317326,
ymax=4007.70056329317,
ylabel style={font=\color{white!15!black}},
ylabel={Mean squared gradient norm},
axis background/.style={fill=white},
title style={font=\bfseries},
axis x line*=bottom,
axis y line*=right,
xmajorgrids,
ymajorgrids,
xtick = {0,50,100,150,200,250,300},
legend style={at={(0.03,0.97)}, anchor=north west, legend cell align=left, align=left, draw=white!15!black}
]

\addplot [color=mycolor1, line width=2.0pt]
  table[row sep=crcr]{%
2	0.037149373302615\\
5	0.192849048858476\\
10	0.960668636274019\\
15	1.47821633895556\\
20	3.8887408410889\\
25	5.90526489461628\\
30	6.64714088325725\\
35	9.29117319620009\\
40	12.1874155529811\\
45	16.5834001122649\\
50	24.5967426682977\\
55	26.4644685003181\\
60	25.5865944254713\\
65	35.5369775683257\\
70	28.8645491147193\\
75	47.3915301375284\\
80	56.7046298428362\\
85	55.1218350297763\\
90	69.2989311309253\\
95	78.1010014752309\\
100	90.501762174164\\
105	85.2679675210889\\
110	120.630199941407\\
115	110.410845859054\\
120	120.358777011224\\
125	138.325797758359\\
130	152.077914654021\\
135	155.314690823711\\
140	173.704404779238\\
145	182.561791986124\\
150	195.576912360159\\
155	226.051065829533\\
160	236.096473463315\\
165	234.922891865202\\
170	276.614971199217\\
175	255.457890658331\\
180	289.918981543731\\
185	341.226778557949\\
190	333.561068886118\\
195	304.052598367158\\
200	320.264864636773\\
205	335.427391700634\\
210	428.494126107553\\
215	482.859921857595\\
220	400.241022256309\\
225	453.153068879463\\
230	444.462508638929\\
235	460.796351500776\\
240	529.54156131653\\
245	552.950562631243\\
250	546.404002722256\\
255	514.874987842098\\
260	572.607852616931\\
265	588.578259880998\\
270	652.396793158414\\
275	682.175693172544\\
280	689.765437954505\\
285	658.541211019633\\
290	697.919877296392\\
295	744.367291406111\\
300	795.603027530778\\
};
\addlegendentry{$||g_t||^2$ (ours)}

\addplot [color=black, line width=2.0pt]
  table[row sep=crcr]{%
2	0.0720382310531737\\
5	0.277681155473236\\
10	1.67599750772179\\
15	3.81627222079274\\
20	5.84422543541167\\
25	11.8753042231406\\
30	14.7694389269309\\
35	20.7627347846184\\
40	26.307237348122\\
45	37.8328234740546\\
50	46.720234324135\\
55	55.9141606160325\\
60	70.9587063346862\\
65	83.0901214785249\\
70	96.468237129471\\
75	120.848556213558\\
80	145.599655309363\\
85	160.972599632997\\
90	178.118437538444\\
95	196.328267104212\\
100	242.567180452068\\
105	335.781434842264\\
110	375.222232006804\\
115	421.710653109661\\
120	484.948032654067\\
125	452.497267853968\\
130	550.629092763733\\
135	572.010204371875\\
140	673.702026230518\\
145	694.44375893695\\
150	701.601906542918\\
155	795.567801193935\\
160	832.991901851835\\
165	905.719276515107\\
170	947.068099957176\\
175	1100.13016575422\\
180	1153.64667677444\\
185	1200.53891402118\\
190	1312.65940166872\\
195	1449.58046091096\\
200	1394.94281739629\\
205	1579.84647755348\\
210	1771.40768711687\\
215	1669.32151752607\\
220	1732.03669292953\\
225	1838.72244780918\\
230	2016.1959232951\\
235	2121.2722204134\\
240	2277.6758261773\\
245	2349.41522580614\\
250	2654.36573694404\\
255	2769.03366317015\\
260	2694.75596688909\\
265	2898.7554901487\\
270	2815.20911062928\\
275	3387.58040789055\\
280	3223.45725918719\\
285	3411.91517558441\\
290	3499.01405991708\\
295	3597.33996614948\\
300	3734.08068616977\\
};
\addlegendentry{$||\tilde g_t||^2$ (baseline)}

\end{axis}
\end{tikzpicture}%
%
  \vspace*{-2mm}
  \caption{\small Empirical second moments as a function of problem dimension.} %
  \label{fig:variance_comparison}
\end{wrapfigure}
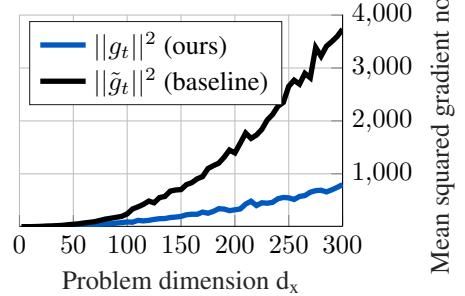
\begin{lemma}[Second-moment bound]\label{lem:variancecomparison}
Under \Cref{ass:bounded,ass:lipschitz}, let $g_t$ be the estimator 
from~\eqref{eq:estimatedirection} and let $\widetilde{g}_t$ be the black-box estimator 
from~\eqref{eq:twopointzoestimate}. Then,
\begin{align*}
    \E[\norm{g_t}^2 \mid x_t] &\leq k_2\dx L_f^2 L_y^2 + L_f^2(1+2L_y), \\
    \E[\norm{\widetilde{g}_t}^2 \mid x_t] &\leq k_2\dx L_f^2 L_y^2 + 
    k_2\dx L_f^2(1+2L_y),
\end{align*}
where $k_2\geq 1$ is a universal constant from \cite[Lem.~10]{shamir2017boundestimator}, and the first bound is no larger than the second for $\dx \geq 1$.
\end{lemma}
\Cref{fig:variance_comparison} validates \cref{lem:variancecomparison} empirically: it plots the empirical second moments $\E[\|g_t\|^2]$ and $\E[\|\tilde g_t\|^2]$ as a function of the problem dimension $\dx$, on Stackelberg instances (\cref{sec:experiments_stackelberg}) over 3500 evaluations %
Since the bounds in \cref{lem:variancecomparison} hold uniformly in $x_t$, they extend directly to these unconditional averages, confirming that the gap between $\E[\|g_t\|^2]$ and $\E[\|\tilde g_t\|^2]$ grows with $\dx$.

{
\section{Main Result}\label{sec:mainresult}
{Since $F(x) = f(x, y^*(x))$ is non-smooth and accessible only via queries, convergence to near Clarke stationary points requires exponentially many queries~\cite{kornowski2022nearclarkestationary},} motivating the community to adopt $(\delta,\varepsilon)$-Goldstein stationarity~\cite{lin2022goldstein} as the standard target. Below, we introduce this notion along with a novel partial counterpart tailored to our composite structure, and prove convergence of \algpzos to both.%

\paragraph{Clarke gradient and Jacobian.}
Due to the non-smoothness of $F$, standard gradients are unavailable and we work instead with the Clarke generalized gradient~\cite{clarke1990generalizedgradients}, defined as the convex hull of all limiting gradients at nearby differentiable points.\footnote{Formally, the Clarke generalized gradient of a Lipschitz continuous function $h:\R^n\to\R$ at $x$ is $\partial h(x) = \co\{\lim_{x_k \to x} \grad h(x_k) : h \text{ is differentiable at } x_k\}$, and the Clarke generalized Jacobian of a Lipschitz continuous vector-valued function $h:\R^n\to \R^m$ is defined analogously as $\partial h(x) = \co\{\lim_{x_k \to x} \jac h(x_k) : h \text{ is differentiable at } x_k\}$.} For vector-valued functions such as $y^*$, the analogous object is the Clarke generalized Jacobian $\partial y^*(x)$, defined likewise. Since $f$ is smooth, Clarke's chain rule applied to $F(x) = f(x, y^*(x))$ yields
\begin{equation}\label{eq:clarkegradientF_main}
\partial F(x) = \bigl\{\grad_x f(x,y^*(x)) + M^\top \grad_y f(x,y^*(x)) : M \in \partial y^*(x)\bigr\}.
\end{equation}
This expression serves as the foundation for the stationarity notions introduced next.

\paragraph{Goldstein stationarity.}
A natural target for optimization is a point where the Clarke gradient contains an element of small norm -- a Clarke stationary point. However, as noted above, this is not achievable in finite time with only query access to $F$. The $(\delta,\varepsilon)$-Goldstein stationarity notion~\cite{lin2022goldstein} offers a tractable relaxation: rather than requiring a small element in the Clarke gradient at a single point, it requires this of the convex hull of all Clarke gradients within a $\delta$-ball around $x$, which we denote $\ball(x,\delta) = \{x' \in \R^{\dx} : \|x' - x\| \leq \delta\}$. The $\delta$-Goldstein subdifferential is defined as:
\begin{equation}\label{eq:goldstein_full_main}\goldsteinfull F(x) = \co\Bigl\{\grad_x f(z,y^*(z)) + M^\top\grad_y f(z,y^*(z)) : z\in\ball(x,\delta),\, M\in\partial y^*(z)\Bigr\}.
\end{equation}
In this expression, both the partial gradients of $f$ and the Clarke Jacobian of $y^*$ vary as $z$ ranges over the $\delta$-ball. Since in our setting the partial gradients of $f$ are known and smooth, it is natural to consider a variant that fixes them at the center point and allows variation only in the Clarke Jacobian of $y^*$. This leads to the novel $\delta$-\emph{partial Goldstein subdifferential}:
\begin{equation}\label{eq:goldstein_partial_main}
\goldsteinpartial F(x) = \Bigl\{\grad_x f(x,y^*(x)) + M^\top\grad_y f(x,y^*(x)) : M\in\co\bigl\{\cup_{z\in\ball(x,\delta)}\partial y^*(z)\bigr\}\Bigr\}.
\end{equation}

We are now ready to define the corresponding stationarity notions, which require a small-norm element not in the Clarke gradient itself, but in the corresponding $\delta$-Goldstein subdifferential.

\begin{definition}[Goldstein stationarity]\label{def:goldstein}
For $\delta,\varepsilon > 0$:
\begin{itemize}[leftmargin=*]
    \item $x$ is a $(\delta,\varepsilon)$-Goldstein stationary point of $F$ if 
    $\min\{\norm{g}:g\in\goldsteinfull F(x)\}\leq\varepsilon$.
    \item $x$ is a $(\delta,\varepsilon)$-partial Goldstein stationary point of 
    $F$ if $\min\{\norm{g}:g\in\goldsteinpartial F(x)\}\leq\varepsilon$.
\end{itemize}
\end{definition}

The two notions are conceptually distinct: $\goldsteinfull F(x)$ varies all components of the Clarke gradient of $F$ over the $\delta$-ball, whereas $\goldsteinpartial F(x)$ anchors the smooth part and varies only the uncertain part. Crucially, neither notion is in general stronger than the other -- a point may be $(\delta,\varepsilon)$-Goldstein stationary without being $(\delta,\varepsilon)$-partial Goldstein stationary, and vice versa. {In particular, whenever $x$ lies within a $\delta$-neighborhood of a Clarke stationary point, the Goldstein subdifferential at $x$ contains zero, making $x$ trivially $(\delta,\varepsilon)$-Goldstein stationary for every $\varepsilon\ge 0$ regardless of the objective gap. We provide Example~\ref{ex:full_not_partial} where the partial Goldstein subdifferential does not suffer from this issue.} Which stationary notion provides more meaningful convergence guarantees depends on the problem instance and interaction between $x$, $y^*$, and the outer objective $f$; we provide a discussion of both notions in Appendix~\ref{sec:goldstein}. Interestingly, our main result shows convergence to both notions simultaneously -- a stronger guarantee than convergence to either one alone.

\begin{theorem}[Convergence of \algpzos]\label{thm:convergencegoldsteinpoint}
Let $x^R$ be chosen uniformly at random from the iterates $\{x_0, \ldots, x_{T-1}\}$ of \cref{alg:partial_zero_order} under \cref{ass:bounded,ass:lipschitz}, with smoothing parameter $\mu < 1$, step size $\alpha = \Theta(\mu^{1/2}\,/\,(T^{1/2}\,\dx^{3/4}))$. Let $C_p = L_g(1+L_y)L_y$ and $C_f = (1+L_y)L_g(1+2L_y)$. Then:

\begin{itemize}[leftmargin=*]
    \item \textbf{Partial Goldstein stationarity:} Choosing $\mu = \min \left(\delta, \tfrac{\varepsilon}{\sqrt{2}C_p}\right)$, after $T = \mc{O}\!\left(\dx^{3/2}/(\mu\,\varepsilon^4)\right)$ iterations, $x^R$ is, in expectation, a $(\delta,\varepsilon)$-partial Goldstein 
    stationary point of $F$, that is, 
    \[\E\left[ \min\left\{ \|g\| : g \in \goldsteinpartial F(x^R) \right\} \right] \leq \varepsilon\]
    
    \item \textbf{Goldstein stationarity:} Choosing $\mu = \min \left(\delta, \tfrac{\varepsilon}{2C_f}\right)$, after $T = \mc{O}\!\left(\dx^{3/2}/(\mu\,\varepsilon^4)\right)$ iterations, $x^R$ is, in expectation, a $(\delta,\varepsilon)$-Goldstein stationary point of $F$, that is,
    \[\E\left[ \min\left\{ \|g\| : g \in \goldsteinfull F(x^R) \right\} 
    \right] \leq \varepsilon.\]
\end{itemize}
\end{theorem}

\paragraph{Proof sketch of \cref{thm:convergencegoldsteinpoint}.}
The proof proceeds in four main steps. We highlight here the main challenges and our approach to resolve them. A detailed proof can be found in \cref{sec:smoothing-properties,sec:surrogate-smoothness,sec:convergence_goldstein}.

\noindent \textbf{Step~1: Smoothness of $y_\mu$} 
(\Cref{sec:smoothing-properties}). We extend classical smoothing results for scalar functions to the vector-valued setting $\R^{\dx}\to\R^{\dy}$, establishing that $y_\mu(x)$ is continuously differentiable with an explicit Jacobian formula. Crucially, since $y^*$ is only Lipschitz continuous, differentiability of $y_\mu$ cannot be obtained by standard arguments; instead it requires measure-theoretic reasoning via Rademacher's theorem and the dominated convergence theorem.

\noindent \textbf{Step~2: Smoothness of $F_\mu$ and biased estimator} 
(\Cref{sec:surrogate-smoothness}). Combining smoothness of $f$ with differentiability of $y_\mu$, we show that $F_\mu$ is 
$L_F$-smooth, enabling a descent lemma analysis. However, $g_t$ is a \emph{biased} estimator of $\grad F_\mu(x_t)$: since $y_\mu(x_t)$ is not available in closed form, the partial gradients of $f$ can not be evaluated at $(x_t, y_\mu(x_t))$, introducing a bias of order $\mu$. Importantly, this bias is fundamental and cannot be eliminated. Indeed, even employing an unbiased estimator $\hat{y}$ of $y_\mu(x_t)$ would yield a biased estimator of $\grad F_\mu(x_t)$, since $\E[\grad_y f(x_t, \hat{y})] \neq \grad_y f(x_t, y_\mu(x_t))$ in general. This bias does not vanish as $T\to\infty$, and would prevent convergence to $(\delta,\varepsilon)$-Goldstein stationary points if $\mu$ were set equal to $\delta$ as in the standard black-box approach -- since the bias term $\mc{O}(\mu) = \mc{O}(\delta)$ would then exceed the target tolerance $\varepsilon$ whenever $\delta \gg \varepsilon$. We circumvent this by decoupling $\mu$ from $\delta$, choosing $\mu$ sufficiently small relative to $\varepsilon$ so that the bias term $\mc{O}(\mu)$ remains below the target tolerance, while still ensuring $\mu \leq \delta$ to transfer convergence guarantees from $F_\mu$ to $F$.

\noindent \textbf{Step~3.A: Convergence to partial Goldstein stationarity} 
(\Cref{sec:convergence_goldstein}). Since $g_t$ is a biased estimator, the standard approach of establishing convergence for the surrogate function is unavailable. Instead, we show via a hyperplane separation argument that $\jac y_\mu(x) \in \goldsteinfull y^*(x)$ for $\delta \geq \mu$. Since $\goldsteinpartial F(x)$ fixes the partial gradients of $f$ at the center point $(x, y^*(x))$ and varies only the Clarke Jacobian of $y^*$, this implies $\E[g_t \mid x_t] \in \goldsteinpartial F(x_t)$ directly. Combining this containment with the descent lemma for $F_\mu$ then yields convergence to $(\delta, \varepsilon)$-partial Goldstein stationarity.

\noindent \textbf{Step~3.B: Convergence to standard Goldstein 
stationarity} (\Cref{sec:convergence_goldstein}). For standard Goldstein stationarity, an additional step is required: since $\goldsteinfull F(x)$ varies both the partial gradients of $f$ and the Clarke Jacobian of $y^*$ over the $\delta$-ball, $\E[g_t \mid x_t]$ does not lie directly in $\goldsteinfull F(x_t)$. Instead, we apply Carathéodory's theorem to show that $\grad F_\mu(x)$ lies within distance $C_f\mu$ of $\goldsteinfull F(x)$, and carry out the descent lemma analysis on $\grad F_\mu(x_t)$ directly. Choosing $\mu$ sufficiently small then ensures the additional approximation error remains below the target tolerance $\varepsilon$, yields the desired result.
{\section{Experiments}\label{sec:experiments}
We validate our method on two problem classes chosen to 
reflect the practical settings described in the introduction: toll optimization in routing games and defense-attack investment in security games.

\subsection{Routing games}\label{sec:experiments_routing}
We consider a toll optimization problem on a road network, where a central authority sets tolls on the edges of a directed graph $G=(V,E)$, representing the transportation network, with the goal of maximizing toll revenue. The network is used by $k$ user groups, each with a fixed origin-destination pair, a travel demand $d_i > 0$, and a sensitivity $s_i > 0$ to tolls. Given a toll vector $\tau$, each user group distributes its flow across the network to minimize its perceived travel cost, which combines the latency $\ell_e(f_e)$ on each edge $e$ — representing the travel time on edge $e$ as a function of the aggregated flow $f_e$ on that edge — with the perceived toll cost $\tau_e / s_i$. Since users act selfishly, they allocate themselves on the network following a Wardrop equilibrium~\cite{yang2004formulation}: each user group routes only on minimum perceived-cost paths. Under standard conditions, the aggregate equilibrium flow over the edges of the graph is unique for each given toll vector $\tau$, which we denote $f^*(\tau)\in\mb{R}^{|E|}$.
The central authority controls the toll vector and aims at maximizing the corresponding revenue:
\begin{equation}\label{eq:routing_obj}
  \max_{\tau\in\mb{R}^{|E|}} \sum_{e \in E} \tau_e \cdot f^*_e(\tau) - \lambda 
  \norm{\tau}^2,
\end{equation}
where $\lambda\norm{\tau}^2$ is a regularization term to discourage excessive tolling. The outer objective is continuously differentiable in $(\tau, f)$; however, the equilibrium response $f^*(\tau)$ is in general non-differentiable in $\tau$~\cite{lindberg2015trafficnondifferentiable}. Indeed, non-differentiability arises when small perturbations of $\tau$ change the set of minimum-cost paths used by some user group. This fits our framework directly as the leader's objective is smooth and known, while the equilibrium flow $f^*(\tau)$ is non-smooth and is accessible only through queries in practice. Indeed, the central authority sets tolls and observes the resulting equilibrium flow, but does not have access to an analytical model of how drivers respond.

\paragraph{Baseline and Setup.} 
We compare our method against \algzos, which applies the black-box smoothing idea discussed in \cref{sec:problemsetup} to the entire composed objective, treating the equilibrium flow as a black box. Concretely, \algzos estimates the gradient of the smooth surrogate $\tilde F_\mu$ via the two-point finite-difference estimator~\eqref{eq:twopointzoestimate}, and runs gradient ascent on this estimate. This is a natural and principled baseline: it has provable convergence guarantees to $(\delta,\varepsilon)$-Goldstein stationary points, and differs from \algpzos solely in that it discards the available partial gradient information on the leader's objective.

Both algorithms are run using 225 randomly generated network instances with 6-25 vertices, 12-119 edges, and 2-5 user groups, each with a fixed origin-destination pair, demands in $\U([3,\ldots,12])$, and toll sensitivities in $\U([0.1,2])$. Edge latency functions are affine, $\ell_e(f_e)=a_e f_e+b_e$, with $a_e\in\U([1,5])$ and $b_e\in\U([1,10])$. We initialize each instance with the initial toll vector $\tau_0 = 0$. %
We set $\lambda=1$, $\mu=0.5$, $\alpha=0.7/|E|$, and $T=150$ iterations.\footnote{{The step size was selected to give both algorithms stable and competitive performance and was not tuned specifically to \algpzos. If anything, $\alpha$ was chosen conservatively relative to what \algpzos alone could achieve, in order to stabilize \algzos. A sensitivity analysis over $\mu$ is provided in \Cref{fig:routing_mucomparison}, confirming that \algpzos outperforms \algzos across all tested values of $\mu$, with neither algorithm showing strong sensitivity to this parameter.}}
\begin{wrapfigure}[17]{r}{0.37\linewidth}
  \centering
  \vspace*{-6.5mm}
  \input{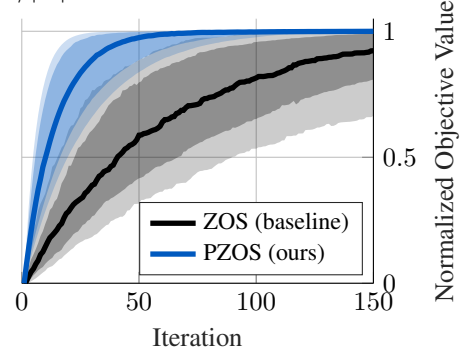}
  \caption{\small Algorithm performance on heterogeneous routing games. Median normalized objective trajectories with interquartile range (dark shading) and 10th--90th percentile range (light shading).}
  \label{fig:q1_iqr}
\end{wrapfigure}
For a given instance, both algorithms are run with the same starting point and noise directions $v_t \sim \U(\sphere^{\dx})$, so that performance differences are not attributable to sampling variation. Since objective values depend on problem scale, objectives are normalized to $[0,1]$ via $(F(x_t)-F_{\text{lowest}})/(F_{\text{highest}}-F_{\text{lowest}})$, where $F_{\text{highest}}$ ($F_{\text{lowest}}$) are the highest (lowest) objective values achieved by either algorithm on that run.\footnote{The experiments are run on a 13-inch MacBook Pro with 2.3 GHz Dual-Core Intel Core i5 processor and 8 GB 2133 MHz RAM. The code is implemented in \texttt{Matlab}.}

\paragraph{PZOS outperforms ZOS across all runs.} 
{\algpzos consistently outperforms \algzos across all 225 runs in convergence speed and median objective value (\cref{fig:q1_iqr}). It also exhibits lower objective variability across runs, which is consistent with the reduced estimator variance shown in \cref{lem:variancecomparison}.}
The performance gap is consistent across runs: the interquartile range of \algpzos lies entirely above that of \algzos, with the best runs of \algzos barely matching the worst runs of \algpzos. \Cref{fig:instance-by-instance} further confirms this on a per-instance basis: at every snapshot $t\in\{10,25,50\}$, every instance lies above the diagonal, and the gap is most pronounced for higher-dimensional instances (red dots). %
This suggests that the variance reduction predicted by \cref{lem:variancecomparison} translates directly into more reliable and consistent performance across problem instances. {Notably, \algpzos achieves this with only one additional oracle call per iteration, and still outperforms \algzos when \algzos is allowed the same or even more oracle calls (see \cref{fig:routing_Q124_oraclecalls}), confirming the gains stem from exploiting the known structure of the leader's objective.}

\begin{figure}[t]
    \centering
    \hspace*{-1cm}
    \input{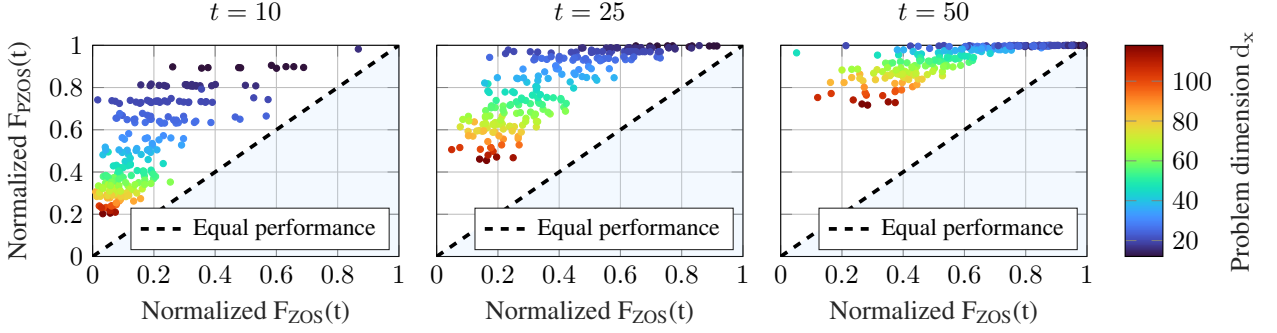}
    \caption{Normalized objective value (higher is better) achieved by the baseline \algzos ($x$-axis) vs our method \algpzos ($y$-axis) on a per-instance basis, at iteration 10 (left panel), 25 (central panel), and 50 (right panel). Every point, each corresponding to an instance, lies above the diagonal, showing that \algpzos returns higher (better) objective values on all routing instances. Problems in higher dimensions (red dots) are those for which \algzos struggles the most, consistent with \cref{lem:variancecomparison}}
    \label{fig:instance-by-instance}
\end{figure}
\paragraph{Effect of batch size.} A standard approach to reduce the variance of \algzos is to average $Q>1$ independent two-point estimates per iteration when evaluating $\tilde{g}_t$ and $g_t$, at the cost of additional oracle evaluations of $y^*$ and $f$. {
We give \algzos a significant advantage by allowing it to use $Q\in\{1,2,4\}$ independent samples per iteration, while keeping \algpzos at $Q=1$. Even so, \algzos at $Q=4$ (using 8 oracle calls per iteration compared to \algpzos's 3) fails to close the gap, see \cref{fig:q1q124_iqr}. 
Moreover, \algpzos outperforms \algzos also on an oracle-call basis: \Cref{fig:routing_Q124_oraclecalls} plots the same comparison with oracle calls to $y^*$ on the $x$-axis instead of iterations, showing that \algpzos delivers superior performance per oracle call across all batch sizes -- a significant advantage when evaluating $y^*(x)$ is expensive.} %

\begin{figure}[htb]
  \centering
  \vspace*{-0.2cm}
  \hspace*{-1cm}
  \input{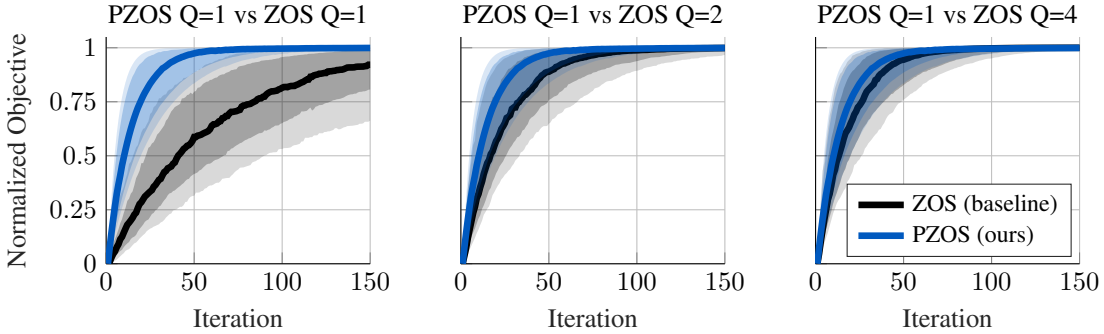}
  \caption{Effect of batch size $Q$ on algorithm performance for heterogeneous routing games: \algpzos at $Q=1$ compared against \algzos at $Q \in \{1,2,4\}$, showing median normalized objective with IQR (dark shading) and 10th--90th percentile bands (light shading), on the same 225 runs as \cref{sec:experiments_routing}.}
  \label{fig:q1q124_iqr}
\end{figure}
\vspace*{-7pt}
\subsection{Security games}\label{sec:experiments_stackelberg}
\vspace*{-5pt}
We consider a Stackelberg security game in which a defender allocates investments across $n$ targets to minimize the damage from an anticipated attack, while an attacker responds by choosing where to direct their effort. The defender moves first, committing to a defense investment $x\in\R^n_{\ge 0}$ across the $n$ targets. The attacker then observes $x$ and responds by allocating attack efforts $y\in\R^n_{\ge 0}$, subject to a budget constraint $\sum_i y_i\le B_a$, to maximize their expected gain. The success probability of an attack on target $i$ is captured by the contest success function~\cite{hausken2011continuouscsf,tullock1980contestsuccessfunction}
\(
p_i(x_i,y_i)=\frac{y_i}{x_i+y_i+b_i},
\)
where $b_i>0$ is a baseline security level. Given the defender's investment $x$, the attacker solves \vspace{-3pt}
\begin{equation}\label{eq:attacker_obj}
  \max_{y \geq 0} \;
  \sum_{i=1}^n w_i p_i(x_i,y_i) - \tfrac{c_i^{\mathrm{A}}y_i^2}{2}
  \quad \text{s.t.} \quad \sum_{i=1}^n y_i \leq B_a,
\end{equation}
where $w_i > 0$ is the value of target $i$ to the attacker and $c_i^{\mathrm{A}} > 0$ is a quadratic effort cost. The attacker's objective is strictly concave, guaranteeing a unique response $y^*(x)$ for each $x$, which is continuous but possibly non-differentiable -- non-differentiability arises when the set of targets receiving positive attack effort changes under a binding budget constraint.\footnote{Uniqueness and continuity of $y^*(x)$ follow from strict concavity of the attacker's objective and Berge's Maximum Theorem~\cite[Thm.~9.17]{sundaram1996optimizationtheory}.} The defender anticipates $y^*(x)$ and minimizes expected damage and investment cost:
\begin{equation}\label{eq:defender_obj}
  \min_{x \geq 0} \; \sum\nolimits_{i=1}^n \left(v_i\, p_i(x_i,y^*_i(x)) + 
  \tfrac{c_i^{\mathrm{D}}x_i^2}{2}\right),
\end{equation}
where $v_i > 0$ is the value of target $i$ to the defender. This fits our framework directly: the defender's objective is smooth and known, while the attacker's response $y^*(x)$ is non-smooth and, in practice, accessible only through queries -- the defender commits to an investment and observes the resulting attack, but has no analytical model of the attacker's strategy.
\begin{wrapfigure}{r}{0.38\linewidth}
\vspace*{-6mm}
  \centering
  \input{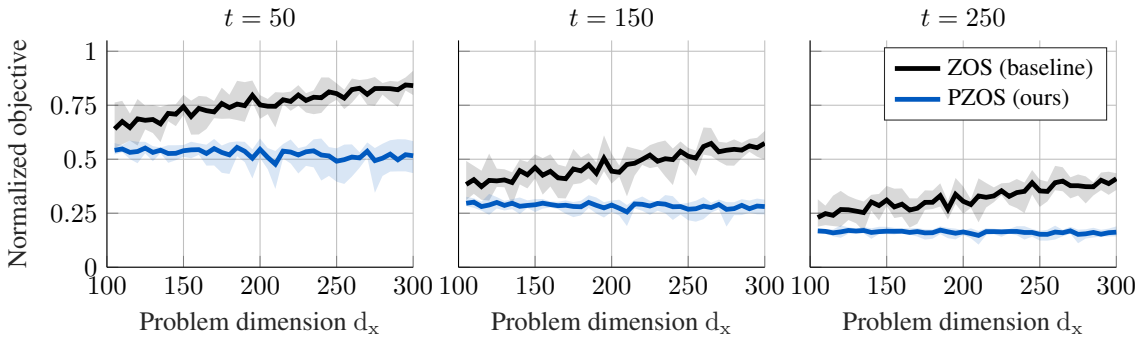}
  \vspace*{-5mm}
  \caption{\small Algorithm performance on security games. Median normalized objective with interquartile range (dark shading) and 10th--90th percentile (light shading).}
  \label{fig:ssg_csf_iqr}
  \vspace*{-6mm}
\end{wrapfigure}
\paragraph{Baseline and Setup.} We compare \algpzos against \algzos, the same baseline as in \cref{sec:experiments_routing}, and normalize trajectories as described therein. We use 427 randomly generated instances with $n\in[2,300]$ targets, $v_i,w_i\sim\U([1,5])$, $b_i\sim\U([0.1,0.5])$, $c_i^{\mathrm{D}},c_i^{\mathrm{A}}\sim\U([0.1,0.3])$. The attacker's budget is set to ensure the budget constraint is binding throughout.\footnote{Specifically, $B_a=\tfrac{1}{2}\sum_i y_i^*(x_{\mathrm{sym}})$, where $y_i^*(x_{\mathrm{sym}})$ is the attacker's optimal unconstrained effort assuming the defender mirrors the attacker's effort, $x_{\mathrm{sym}} = y_{\mathrm{sym}}$, scaled by $0.5$ to ensure the constraint binds.} We set $\mu=0.1$, $T=500$, $Q=1$, and diminishing step size $\alpha_t=0.05/\sqrt{t}$ ($n\le100$) or $0.03/\sqrt{t}$ ($n>100$).\footnote{{The smaller step size for $n>100$ was chosen to ensure stability of~\algzos at higher dimensions, where its larger estimator variance would otherwise lead to instability with the step size used for $n\le100$; \algpzos remains stable under both step sizes.}} The initial point $x_0$ is set to the defender's optimal investment assuming the attacker exerts effort $y_i=b_i$ on each target, a natural starting point that requires no knowledge of the attacker's budget or cost structure. {The algorithms we implement enforce $x \geq 0$ via projection after each gradient step. In practice, however, the projection is not activated, as the gradient points into the positive orthant.}
\paragraph{PZOS outperforms ZOS across all runs.}
The results mirror those of the routing games: \algpzos consistently outperforms \algzos across all 427 runs in convergence speed, median objective value and variance (\cref{fig:ssg_csf_iqr}). Notably, this holds across a wide range of problem dimensions $n\in[2,300]$, and confirms a systematic benefit of exploiting the known leader's objective.
\paragraph{Performance gap increases with dimension.}
\Cref{fig:ssg_csf_dimension} examines how the two algorithms behave as the problem dimension grows. While \algpzos maintains a consistently low and stable objective value across all dimensions, \algzos degrades substantially as $n$ increases, with the gap between the two methods widening monotonically at every snapshot $t\in\{50,150,250\}$. This suggests that the advantage of confining zeroth-order estimation to the Jacobian of $y^*$ alone -- rather than the entire composed objective -- becomes increasingly significant as the problem dimension grows, and is consistent with the dimension-dependent variance bound of \cref{lem:variancecomparison}, under which the second-moment gap between $g_t$ and $\tilde{g}_t$ grows linearly in the problem dimension.
\begin{figure}[H]
\vspace*{-0.2cm}
    \centering
    \hspace*{-1cm}
    \input{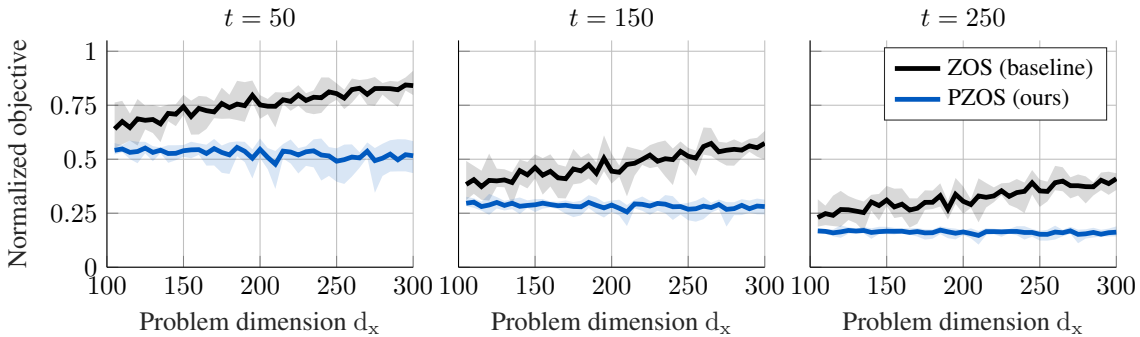}
    \caption{Normalized objective value (lower is better) achieved by the baseline \algzos (black) and our method \algpzos (blue) as a function of problem dimension (number of targets $n\in[100,300]$) at iterations 50 (left panel), 150 (center), 250 (right), over 287 instances. Shading shows min/max range. As dimension grows, \algzos degrades substantially while \algpzos remains stable.}%
    \label{fig:ssg_csf_dimension}
\end{figure}

\paragraph{Limitations.}Our algorithm is developed for the setting in which $y^*$ is constrained, but $x$ is unconstrained. In Appendix~\ref{sec:appendixalgorithms}, we test both algorithms in the constrained setting by requiring $\tau \geq 0$ in the routing objective~\eqref{eq:routing_obj}, enforcing feasibility via projection onto $\mb{R}^{|E|}_{\geq 0}$ after each update step. Both \algzos and \algpzos remain stable under this modification, with \algpzos still achieving visibly better performance than \algzos (\Cref{fig:routing_const_iqr}). Extending the theory to the constrained setting $x\in\X$ would require $y^*(x)$ to be well defined on $\X_\mu = \X+\ball(0,\mu)$ rather than only on $\X$. One could then extend the partial Goldstein stationarity notion analogously to the constrained Goldstein stationarity notion of \citet{cui2023mpec,liu2024goldsteinconstrained}.
\FloatBarrier

\bibliography{bibref}

\begin{thebibliography}{32}
\providecommand{\natexlab}[1]{#1}
\providecommand{\url}[1]{\texttt{#1}}
\expandafter\ifx\csname urlstyle\endcsname\relax
  \providecommand{\doi}[1]{doi: #1}\else
  \providecommand{\doi}{doi: \begingroup \urlstyle{rm}\Url}\fi

\bibitem[Amos and Kolter(2017)]{amos2017optnet}
Brandon Amos and J~Zico Kolter.
\newblock Optnet: Differentiable optimization as a layer in neural networks.
\newblock In \emph{International conference on machine learning}, pages
  136--145. PMLR, 2017.

\bibitem[Aubin and Cellina(1984)]{Aubin1984usccompact}
Jean-Pierre Aubin and Arrigo Cellina.
\newblock \emph{Differential inclusions. Set-valued maps and viability theory}.
\newblock Springer, 1984.

\bibitem[Benson et~al.(2006)Benson, Sen, Shanno, and
  Vanderbei]{benson2006interior}
Hande~Y Benson, Arun Sen, David~F Shanno, and Robert~J Vanderbei.
\newblock Interior-point algorithms, penalty methods and equilibrium problems.
\newblock \emph{Computational Optimization and Applications}, 34\penalty0
  (2):\penalty0 155--182, 2006.

\bibitem[Chen et~al.(2023)Chen, Xu, and Luo]{chen2023goldsteinimproved}
Lesi Chen, Jing Xu, and Luo Luo.
\newblock Faster gradient-free algorithms for nonsmooth nonconvex stochastic
  optimization.
\newblock In \emph{International Conference on Machine Learning}, pages
  5219--5233. PMLR, 2023.

\bibitem[Clarke(1990)]{clarke1990generalizedgradients}
Frank~H Clarke.
\newblock \emph{Optimization and nonsmooth analysis}.
\newblock SIAM, 1990.

\bibitem[Cui and Shanbhag(2023)]{cui2023hierarchicalgames}
Shisheng Cui and Uday~V Shanbhag.
\newblock On the computation of equilibria in monotone and potential stochastic
  hierarchical games.
\newblock \emph{Mathematical Programming}, 198\penalty0 (2):\penalty0
  1227--1285, 2023.

\bibitem[Cui et~al.(2023)Cui, Shanbhag, and Yousefian]{cui2023mpec}
Shisheng Cui, Uday~V Shanbhag, and Farzad Yousefian.
\newblock Complexity guarantees for an implicit smoothing-enabled method for
  stochastic {MPECs}.
\newblock \emph{Mathematical Programming}, 198\penalty0 (2):\penalty0
  1153--1225, 2023.

\bibitem[DeMiguel et~al.(2005)DeMiguel, Friedlander, Nogales, and
  Scholtes]{demiguel2005mpecrelaxation}
Victor DeMiguel, Michael~P Friedlander, Francisco~J Nogales, and Stefan
  Scholtes.
\newblock A two-sided relaxation scheme for mathematical programs with
  equilibrium constraints.
\newblock \emph{SIAM Journal on Optimization}, 16\penalty0 (2):\penalty0
  587--609, 2005.

\bibitem[Elkind et~al.(2024)Elkind, Ghosh, and Goldberg]{elkind2024csfunique}
Edith Elkind, Abheek Ghosh, and Paul~W Goldberg.
\newblock Continuous-time best-response and related dynamics in tullock
  contests with convex costs.
\newblock \emph{arXiv preprint arXiv:2402.08541}, 2024.

\bibitem[Flaxman et~al.(2005)Flaxman, Kalai, and McMahan]{flaxman2005}
Abraham~D. Flaxman, Adam~Tauman Kalai, and H.~Brendan McMahan.
\newblock Online convex optimization in the bandit setting: gradient descent
  without a gradient.
\newblock In \emph{Proceedings of the Sixteenth Annual ACM-SIAM Symposium on
  Discrete Algorithms}, SODA '05, page 385–394, USA, 2005. Society for
  Industrial and Applied Mathematics.
\newblock ISBN 0898715857.

\bibitem[Fletcher et~al.(2006)Fletcher, Leyffer, Ralph, and
  Scholtes]{fletcher2006sqp}
Roger Fletcher, Sven Leyffer, Danny Ralph, and Stefan Scholtes.
\newblock Local convergence of {SQP} methods for mathematical programs with
  equilibrium constraints.
\newblock \emph{SIAM Journal on Optimization}, 17\penalty0 (1):\penalty0
  259--286, 2006.

\bibitem[Grontas et~al.(2024)Grontas, Belgioioso, Cenedese, Fochesato, Lygeros,
  and D{\"o}rfler]{grontas2024bighype}
Panagiotis~D Grontas, Giuseppe Belgioioso, Carlo Cenedese, Marta Fochesato,
  John Lygeros, and Florian D{\"o}rfler.
\newblock Big hype: Best intervention in games via distributed hypergradient
  descent.
\newblock \emph{IEEE Transactions on Automatic Control}, 69\penalty0
  (12):\penalty0 8338--8353, 2024.

\bibitem[Hausken and Bier(2011)]{hausken2011continuouscsf}
Kjell Hausken and Vicki~M Bier.
\newblock Defending against multiple different attackers.
\newblock \emph{European Journal of Operational Research}, 211\penalty0
  (2):\penalty0 370--384, 2011.

\bibitem[Hoheisel et~al.(2013)Hoheisel, Kanzow, and
  Schwartz]{hoheisel2013mpecrelaxation}
Tim Hoheisel, Christian Kanzow, and Alexandra Schwartz.
\newblock Theoretical and numerical comparison of relaxation methods for
  mathematical programs with complementarity constraints.
\newblock \emph{Mathematical Programming}, 137\penalty0 (1):\penalty0 257--288,
  2013.

\bibitem[Iliaev et~al.(2023)Iliaev, Oren, and Segev]{iliaev2023csf}
David Iliaev, Sigal Oren, and Ella Segev.
\newblock A tullock-contest-based approach for cyber security investments.
\newblock \emph{Annals of Operations Research}, 320\penalty0 (1):\penalty0
  61--84, 2023.

\bibitem[Kornowski and Shamir(2022)]{kornowski2022nearclarkestationary}
Guy Kornowski and Ohad Shamir.
\newblock Oracle complexity in nonsmooth nonconvex optimization.
\newblock \emph{Journal of Machine Learning Research}, 23\penalty0
  (314):\penalty0 1--44, 2022.

\bibitem[Kornowski and Shamir(2024)]{kornowski2024goldsteinoptimal}
Guy Kornowski and Ohad Shamir.
\newblock An algorithm with optimal dimension-dependence for zero-order
  nonsmooth nonconvex stochastic optimization.
\newblock \emph{Journal of Machine Learning Research}, 25\penalty0
  (122):\penalty0 1--14, 2024.

\bibitem[Lin et~al.(2022)Lin, Zheng, and Jordan]{lin2022goldstein}
Tianyi Lin, Zeyu Zheng, and Michael Jordan.
\newblock Gradient-free methods for deterministic and stochastic nonsmooth
  nonconvex optimization.
\newblock \emph{Advances in Neural Information Processing Systems},
  35:\penalty0 26160--26175, 2022.

\bibitem[Lindberg and Engelson(2015)]{lindberg2015trafficnondifferentiable}
P.O. Lindberg and Leonid Engelson.
\newblock Tolled multi-class traffic equilibria and toll sensitivities.
\newblock \emph{EURO Journal on Transportation and Logistics}, 4\penalty0
  (2):\penalty0 197--222, 2015.
\newblock ISSN 2192-4376.

\bibitem[Liu et~al.(2024)Liu, Chen, Luo, and Low]{liu2024goldsteinconstrained}
Zhuanghua Liu, Cheng Chen, Luo Luo, and Bryan Kian~Hsiang Low.
\newblock Zeroth-order methods for constrained nonconvex nonsmooth stochastic
  optimization.
\newblock In Ruslan Salakhutdinov, Zico Kolter, Katherine Heller, Adrian
  Weller, Nuria Oliver, Jonathan Scarlett, and Felix Berkenkamp, editors,
  \emph{Proceedings of the 41st International Conference on Machine Learning},
  volume 235 of \emph{Proceedings of Machine Learning Research}, pages
  30842--30872. PMLR, 2024.

\bibitem[Nesterov and Spokoiny(2017)]{nesterov2017spokoiny}
Yurii Nesterov and Vladimir Spokoiny.
\newblock Random gradient-free minimization of convex functions.
\newblock \emph{Foundations of Computational Mathematics}, 17\penalty0
  (2):\penalty0 527--566, 2017.

\bibitem[Qi and Wei(2000)]{qi2000sqp}
Liqun Qi and Zengxin Wei.
\newblock On the constant positive linear dependence condition and its
  application to {SQP} methods.
\newblock \emph{SIAM Journal on Optimization}, 10\penalty0 (4):\penalty0
  963--981, 2000.

\bibitem[Raghunathan and Biegler(2005)]{raghunathan2005mpccinterior}
Arvind~U Raghunathan and Lorenz~T Biegler.
\newblock An interior point method for mathematical programs with
  complementarity constraints (mpccs).
\newblock \emph{SIAM Journal on Optimization}, 15\penalty0 (3):\penalty0
  720--750, 2005.

\bibitem[Shamir(2017)]{shamir2017boundestimator}
Ohad Shamir.
\newblock An optimal algorithm for bandit and zero-order convex optimization
  with two-point feedback.
\newblock \emph{Journal of Machine Learning Research}, 18\penalty0
  (52):\penalty0 1--11, 2017.

\bibitem[Sow et~al.(2022)Sow, Ji, and Liang]{sow2022pzobo}
Daouda Sow, Kaiyi Ji, and Yingbin Liang.
\newblock On the convergence theory for {Hessian-free} bilevel algorithms.
\newblock \emph{Advances in Neural Information Processing Systems},
  35:\penalty0 4136--4149, 2022.

\bibitem[Sundaram(1996)]{sundaram1996optimizationtheory}
Rangarajan~K Sundaram.
\newblock \emph{A first course in optimization theory}.
\newblock Cambridge University Press, 1996.

\bibitem[Tao et~al.(2025)Tao, Cui, Li, and Sun]{tao2025zobilevel}
Haochen Tao, Shisheng Cui, Zhuo Li, and Jian Sun.
\newblock A zeroth-order stochastic implicit method for bilevel-structured
  actor-critic schemes.
\newblock \emph{Science China Information Sciences}, 68\penalty0 (5):\penalty0
  150204, 2025.

\bibitem[Tullock(1980)]{tullock1980contestsuccessfunction}
Gordon Tullock.
\newblock Efficient rent seeking.
\newblock In James~M. Buchanan, Robert~D. Tollison, and Gordon Tullock,
  editors, \emph{Toward a Theory of the Rent-Seeking Society}, pages 97--112.
  Texas A\&M University Press, College Station, TX, 1980.

\bibitem[Vershynin(2018)]{vershynin2018highdimensionalprobability}
Roman Vershynin.
\newblock \emph{High-dimensional probability: An introduction with applications
  in data science}, volume~47.
\newblock Cambridge University Press, 2nd edition, 2018.

\bibitem[Yang and Huang(2004)]{yang2004formulation}
Hai Yang and Hai-Jun Huang.
\newblock The multi-class, multi-criteria traffic network equilibrium and
  systems optimum problem.
\newblock \emph{Transportation Research Part B: Methodological}, 38\penalty0
  (1):\penalty0 1--15, 2004.

\bibitem[Yousefian et~al.(2012)Yousefian, Nedi\'{c}, and
  Shanbhag]{yousefian2012randomizedsmoothingproperties}
Farzad Yousefian, Angelia Nedi\'{c}, and Uday~V Shanbhag.
\newblock On stochastic gradient and subgradient methods with adaptive
  steplength sequences.
\newblock \emph{Automatica}, 48\penalty0 (1):\penalty0 56--67, 2012.

\bibitem[Zhang et~al.(2020)Zhang, Lin, Jegelka, Sra, and
  Jadbabaie]{zhang2020clarkestationary}
Jingzhao Zhang, Hongzhou Lin, Stefanie Jegelka, Suvrit Sra, and Ali Jadbabaie.
\newblock Complexity of finding stationary points of nonconvex nonsmooth
  functions.
\newblock In \emph{International Conference on Machine Learning}, pages
  11173--11182. PMLR, 2020.

\end{thebibliography}

\clearpage
\appendix
\crefname{section}{Appendix}{Appendices}
\Crefname{section}{Appendix}{Appendices}
 
\section*{Technical appendices and supplementary material}
\section{Supplementary algorithms and graphs}\label{sec:appendixalgorithms}
\begin{algorithm}[hbt]
\caption{Zero-Order Smoothing (ZOS)}
\begin{algorithmic}\label{alg:zero_order}
\STATE \textbf{Input:} Initial $x_0 \in \R^{\dx}$, step size $\alpha$, smoothing parameter $\mu>0$, iterations $T$
\FOR{$t = 0, 1, \ldots, T-1$}
\STATE Sample $v_t \sim \U({\sphere^{\dx}})$
    \hfill\textbackslash\textbackslash  {\small~Random direction uniformly from unit sphere}
\STATE Evaluate $y^*(x_t + \mu v_{t})$, and $y^*(x_t - \mu v_{t})$
\STATE $\tilde g_{t} = \frac{\dx}{2\mu}[f(x_t + \mu v_{t},y^*(x_t + \mu v_{t})) - f(x_t - \mu v_{t},y^*(x_t - \mu v_{t}))]\,v_{t}$
    \hfill\textbackslash\textbackslash {\small~Gradient estimate}
\STATE $x_{t+1} = x_t - \alpha \tilde g_t$
    \hfill\textbackslash\textbackslash {\small~Update step (minimization); for maximization, use $+$}
\ENDFOR
\STATE \textbf{Output:} $x^R$ where $R\in\{0,1,\dots,T-1\}$ is uniformly sampled
\end{algorithmic}
\end{algorithm}

\begin{algorithm}[!hbt]
\caption{Zero-Order with batch size Q (ZOS-B)}
\begin{algorithmic}\label{alg:zero_order_batch}
\STATE \textbf{Input:} Initial $x_0 \in \R^{\dx}$, step size $\alpha$, smoothing parameter $\mu>0$, iterations $T$, batch size $Q$
\FOR{$t = 0, 1, \ldots, T-1$}
\FOR{$q = 1, \ldots, Q$}
\STATE Sample $v_{t,q} \sim \U(\sphere^{\dx})$
    \hfill\textbackslash\textbackslash  {\small~Random direction uniformly from unit sphere}
\STATE Evaluate $y^*(x_t + \mu v_{t,q})$, and $y^*(x_t - \mu v_{t,q})$
\STATE $g_{t,q} = \frac{\dx}{2\mu}[f(x_t + \mu v_{t,q},y^*(x_t + \mu v_{t,q})) - f(x_t - \mu v_{t,q},y^*(x_t - \mu v_{t,q}))]\,v_{t,q}$
    \hfill\textbackslash\textbackslash {\small~Gradient est.}
\ENDFOR
\STATE Compute gradient average: $g_t = \frac{1}{Q} \sum_{q=1}^Q g_{t,q}$
    \hfill\textbackslash\textbackslash {\small~Average of gradient estimates}
\STATE Update: $x_{t+1} = x_t - \alpha g_t$  \hfill\textbackslash\textbackslash {\small~Update step (minimization); for maximization, use $+$}
\ENDFOR
\STATE \textbf{Output:} $x^R$ where $R\in\{0,1,\dots,T-1\}$ is uniformly sampled
\end{algorithmic}
\end{algorithm}
\begin{algorithm}[!hbt]
\caption{Partial Zero-Order with batch size Q (PZOS-B)}
\begin{algorithmic}\label{alg:partial_zero_order_batch}
\STATE \textbf{Input:} Initial $x_0 \in \R^{\dx}$, step size $\alpha$, smoothing parameter $\mu>0$, iterations $T$, batch size $Q$
\FOR{$t = 0, 1, \ldots, T-1$}
\STATE Evaluate $y^*(x_t)$
\FOR{$q = 1, \ldots, Q$}
\STATE Sample $v_{t,q} \sim \U(\sphere^{\dx})$ \hfill\textbackslash\textbackslash  {\small~Random direction uniformly from unit sphere}
\STATE Evaluate $y^*(x_t + \mu v_{t,q})$, and $y^*(x_t - \mu v_{t,q})$
\STATE Compute Jacobian estimate: $H_{t,q} = \frac{\dx}{2\mu}((y^*(x_t + \mu v_{t,q}) - y^*(x_t - \mu v_{t,q})) \,v_{t,q}^\top)$
\ENDFOR
\STATE Compute Jacobian average: $\bar H_t = \frac{1}{Q}\sum_{q=1}^Q H_{t,q}$
\STATE Compute gradient estimate: $g_t = \grad_x f(x_t, y^*(x_t)) + \bar H_t^\top \grad_y f(x_t, y^*(x_t))$ \hfill\textbackslash\textbackslash{\small~Gradient estimate}
\STATE Update: $x_{t+1} = x_t - \alpha g_t$ \hfill\textbackslash\textbackslash {\small~Update step (minimization); for maximization, use $+$}
\ENDFOR
\STATE \textbf{Output:} $x^R$ where $R\in\{0,1,\dots,T-1\}$ is uniformly sampled
\end{algorithmic}
\end{algorithm}

\begin{figure}[t!]
    \centering
    \hspace*{-1cm}
    \input{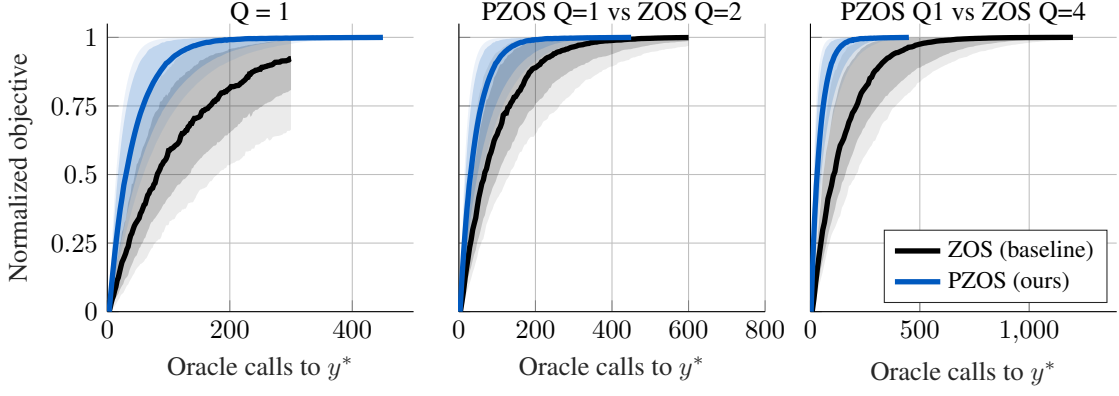}
    \caption{Normalized objective value (higher is better) for the heterogeneous routing game of \cref{sec:experiments_routing} by batch size $Q=1,2,4$, with number of oracle calls of $y^*$ on the $x$-axis. Median with interquartile range (dark shading) and 10-90th percentile bands (light shading). \algpzos outperforms \algzos not only by iteration, but also by number of oracle calls.}
    \label{fig:routing_Q124_oraclecalls}
\end{figure}

\begin{figure}
    \centering
    \hspace*{-1cm}
    \input{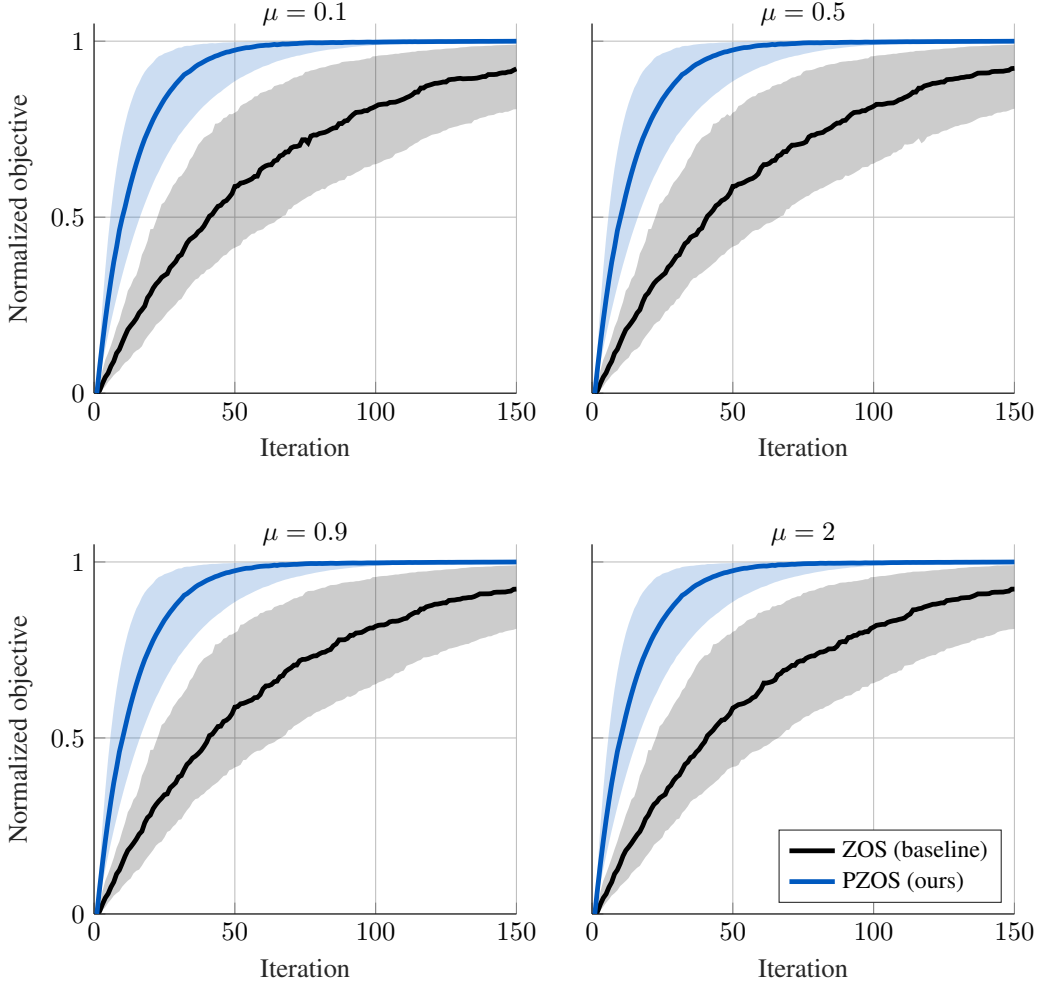}
    \caption{Normalized objective value (higher is better) for smoothing parameter $\mu=\{0.1,0.5,0.9,2\}$ on the same 225 runs as \cref{sec:experiments_routing}. Median with interquartile range. \algpzos performs significantly better than \algzos for all considered values of $\mu$.}
    \label{fig:routing_mucomparison}
\end{figure}

\begin{figure}
    \centering
    \input{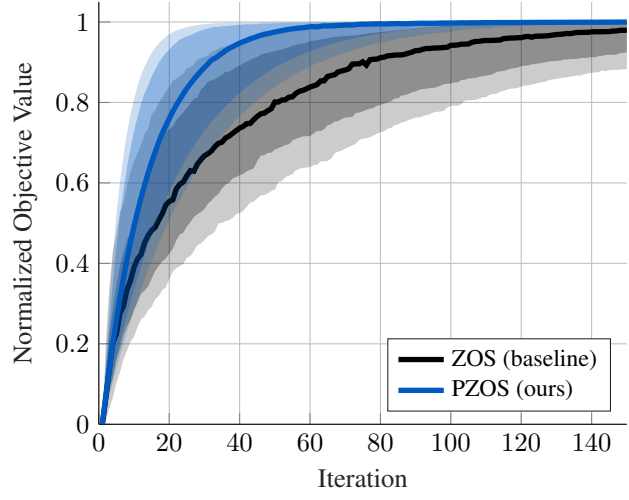}
    \caption{Algorithm performance over heterogeneous routing games from Section~\ref{sec:experiments_routing} with non-negativity constraint $\tau \ge 0$ enforced via projection. Median normalized objective trajectories with interquartile range (dark shading) and 10th–90th percentile range (light shading).}
    \label{fig:routing_const_iqr}
\end{figure}

\begin{figure}
  \centering
  \definecolor{mycolor1}{rgb}{0.00000,0.35000,0.75000}%

\begin{tikzpicture}

\begin{axis}[%
width=0.6\linewidth,
scale only axis,
xmin=0,
xmax=300,
xlabel style={font=\color{white!15!black}},
xlabel={$\text{Problem dimension d}_\text{x}$},
ymin=-500,
ymax=4500,
ylabel style={font=\color{white!15!black}},
ylabel={Mean squared gradient norm},
axis background/.style={fill=white},
title style={font=\bfseries},
axis x line*=bottom,
axis y line*=right,
xmajorgrids,
ymajorgrids,
xtick = {0,50,100,150,200,250,300},
legend style={at={(0.03,0.97)}, anchor=north west, legend cell align=left, align=left, draw=white!15!black}
]
\addplot [color=black, line width=2.0pt]
 plot [error bars/.cd, y dir=both, y explicit, error bar style={line width=1.0pt}, error mark options={line width=1.0pt, mark size=3.0pt, rotate=90}]
 table[row sep=crcr, y error plus index=2, y error minus index=3]{%
2	0.0720382310531737	0.0717059749795495	0.0717059749795495\\
5	0.277681155473236	0.167500946894982	0.167500946894982\\
10	1.67599750772179	1.37793125221054	1.37793125221054\\
15	3.81627222079275	1.83085066148375	1.83085066148375\\
20	5.84422543541167	2.26915553026669	2.26915553026669\\
25	11.8753042231406	6.24458872746895	6.24458872746895\\
30	14.7694389269309	6.30620393132587	6.30620393132587\\
35	20.7627347846184	8.88701256738455	8.88701256738455\\
40	26.307237348122	6.53015701043902	6.53015701043902\\
45	37.8328234740546	14.8925940775723	14.8925940775723\\
50	46.720234324135	15.9354180658447	15.9354180658447\\
55	55.9141606160325	9.79751230426537	9.79751230426537\\
60	70.9587063346862	15.5072080347386	15.5072080347386\\
65	83.0901214785249	21.5986198695955	21.5986198695955\\
70	96.468237129471	13.817566156744	13.817566156744\\
75	120.848556213558	26.3341831940064	26.3341831940064\\
80	145.599655309363	38.9157950452658	38.9157950452658\\
85	160.972599632997	39.6619902391498	39.6619902391498\\
90	178.118437538444	43.2462978225858	43.2462978225858\\
95	196.328267104212	48.4659983697767	48.4659983697767\\
100	242.567180452068	67.0501138886031	67.0501138886031\\
105	335.781434842264	64.9626583014864	64.9626583014864\\
110	375.222232006804	64.1904899783738	64.1904899783738\\
115	421.710653109661	52.9347432376994	52.9347432376994\\
120	484.948032654067	57.1776107742218	57.1776107742218\\
125	452.497267853968	68.3054893131537	68.3054893131537\\
130	550.629092763734	122.03522711466	122.03522711466\\
135	572.010204371874	66.4660164986395	66.4660164986395\\
140	673.702026230518	193.589961106558	193.589961106558\\
145	694.443758936951	164.224359168612	164.224359168612\\
150	701.601906542917	96.8789928471359	96.8789928471359\\
155	795.567801193934	52.4535619305897	52.4535619305897\\
160	832.991901851835	159.500147485095	159.500147485095\\
165	905.719276515106	147.411101995038	147.411101995038\\
170	947.068099957176	273.756255971962	273.756255971962\\
175	1100.13016575422	492.974210406119	492.974210406119\\
180	1153.64667677444	146.540408293411	146.540408293411\\
185	1200.53891402118	145.351508868829	145.351508868829\\
190	1312.65940166872	254.684567730405	254.684567730405\\
195	1449.58046091096	456.310087396173	456.310087396173\\
200	1394.94281739629	209.45874306644	209.45874306644\\
205	1579.84647755348	178.462149154963	178.462149154963\\
210	1771.40768711687	629.055413507477	629.055413507477\\
215	1669.32151752607	350.87529452338	350.87529452338\\
220	1732.03669292953	361.487177524635	361.487177524635\\
225	1838.72244780918	457.839945948974	457.839945948974\\
230	2016.1959232951	140.851630140391	140.851630140391\\
235	2121.2722204134	425.334979444206	425.334979444206\\
240	2277.6758261773	890.525435826054	890.525435826054\\
245	2349.41522580614	473.474434423091	473.474434423091\\
250	2654.36573694404	486.95342300764	486.95342300764\\
255	2769.03366317015	795.292137659746	795.292137659746\\
260	2694.75596688909	637.010103894159	637.010103894159\\
265	2898.7554901487	275.037332565357	275.037332565357\\
270	2815.20911062928	397.636804694583	397.636804694583\\
275	3387.58040789056	929.975741257687	929.975741257687\\
280	3223.45725918719	742.234142931112	742.234142931112\\
285	3411.91517558441	707.174489393674	707.174489393674\\
290	3499.01405991708	607.941668318282	607.941668318282\\
295	3597.33996614949	808.812570184341	808.812570184341\\
300	3734.08068616976	543.776774385006	543.776774385006\\
};
\addlegendentry{$||\tilde g_t||^2$ (baseline)}

\addplot [color=mycolor1, line width=2.0pt]
 plot [error bars/.cd, y dir=both, y explicit, error bar style={line width=1.0pt}, error mark options={line width=1.0pt, mark size=3.0pt, rotate=90}]
 table[row sep=crcr, y error plus index=2, y error minus index=3]{%
2	0.037149373302615	0.0355545303998001	0.0355545303998001\\
5	0.192849048858476	0.289889394458312	0.289889394458312\\
10	0.960668636274019	1.18497482718324	1.18497482718324\\
15	1.47821633895556	1.25364542917667	1.25364542917667\\
20	3.8887408410889	3.70032021606844	3.70032021606844\\
25	5.90526489461628	4.87638996516373	4.87638996516373\\
30	6.64714088325725	4.05610429765497	4.05610429765497\\
35	9.29117319620009	9.24645232734573	9.24645232734573\\
40	12.1874155529811	11.8232920498023	11.8232920498023\\
45	16.5834001122649	12.8369847062298	12.8369847062298\\
50	24.5967426682977	22.3222373674261	22.3222373674261\\
55	26.4644685003181	12.5958659703986	12.5958659703986\\
60	25.5865944254714	9.39325130445303	9.39325130445303\\
65	35.5369775683257	23.5306010041256	23.5306010041256\\
70	28.8645491147193	10.352567188402	10.352567188402\\
75	47.3915301375284	19.3548243063839	19.3548243063839\\
80	56.7046298428362	25.9673456881661	25.9673456881661\\
85	55.1218350297763	22.1214610603513	22.1214610603513\\
90	69.2989311309253	28.4130343731991	28.4130343731991\\
95	78.1010014752309	45.9665344937364	45.9665344937364\\
100	90.501762174164	21.1837867562733	21.1837867562733\\
105	85.2679675210889	42.3264540637423	42.3264540637423\\
110	120.630199941407	57.7269978413479	57.7269978413479\\
115	110.410845859054	58.6140644485731	58.6140644485731\\
120	120.358777011224	29.906469795273	29.906469795273\\
125	138.325797758359	46.9417804775426	46.9417804775426\\
130	152.077914654021	80.6548599717426	80.6548599717426\\
135	155.314690823711	68.0861682534598	68.0861682534598\\
140	173.704404779238	74.5303782339923	74.5303782339923\\
145	182.561791986124	103.397167671714	103.397167671714\\
150	195.576912360159	29.8464500278802	29.8464500278802\\
155	226.051065829533	70.2171198956616	70.2171198956616\\
160	236.096473463315	95.921071448788	95.921071448788\\
165	234.922891865202	130.081491182816	130.081491182816\\
170	276.614971199217	142.385238425239	142.385238425239\\
175	255.457890658331	116.907867579776	116.907867579776\\
180	289.918981543731	93.8743650256624	93.8743650256624\\
185	341.226778557949	124.814336242801	124.814336242801\\
190	333.561068886118	114.10612542846	114.10612542846\\
195	304.052598367158	114.980153722098	114.980153722098\\
200	320.264864636774	155.035541250983	155.035541250983\\
205	335.427391700634	151.202759371618	151.202759371618\\
210	428.494126107552	177.831584670771	177.831584670771\\
215	482.859921857595	120.380014107361	120.380014107361\\
220	400.241022256309	179.707287996483	179.707287996483\\
225	453.153068879463	108.692819791252	108.692819791252\\
230	444.462508638929	153.312494066919	153.312494066919\\
235	460.796351500776	129.251895571254	129.251895571254\\
240	529.54156131653	135.119648450192	135.119648450192\\
245	552.950562631243	157.736028225916	157.736028225916\\
250	546.404002722256	95.4660635407283	95.4660635407283\\
255	514.874987842098	82.4642587528272	82.4642587528272\\
260	572.607852616931	112.712232503566	112.712232503566\\
265	588.578259880998	148.760503324496	148.760503324496\\
270	652.396793158414	241.972154672728	241.972154672728\\
275	682.175693172543	184.037290002897	184.037290002897\\
280	689.765437954505	255.205578929031	255.205578929031\\
285	658.541211019633	181.590860787125	181.590860787125\\
290	697.919877296392	228.280786469083	228.280786469083\\
295	744.367291406111	265.479817090966	265.479817090966\\
300	795.603027530778	304.810727611887	304.810727611887\\
};
\addlegendentry{$||g_t||^2$ (ours)}

\end{axis}
\end{tikzpicture}%
%
  \caption{Plot of \Cref{fig:variance_comparison} with error bars. Mean $\pm$ 2 standard deviations of $\|g_t\|^2$ and $\|\tilde g_t\|^2$ on Stackelberg instances (Section~\ref{sec:experiments_stackelberg}) with 3500 samples per problem dimension.} %
\end{figure}
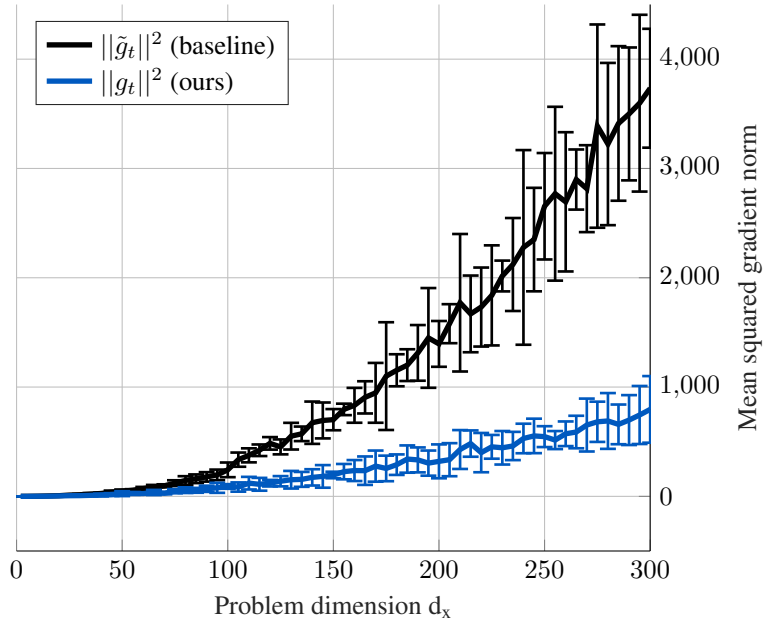

\FloatBarrier
\section{Goldstein Stationarity: Detailed Framework and Examples}\label{sec:goldstein}

This section provides technical details underlying the notion of Goldstein and partial Goldstein stationary point in contribution~\ref{contributiongoldstein}. Specifically, we first recall the Clarke generalized Jacobian and chain rule, then introduce the corresponding Goldstein and partial Goldstein subdifferential with their corresponding stationarity notion. We conclude with two illustrative examples showing that neither notion is in general stronger than the other.

\subsection{Clarke Generalized Jacobian}
We begin by recalling the Clarke generalized Jacobian for vector-valued functions, which extends the Clarke generalized gradient to mappings $G: \R^n \to \R^m$.

\begin{definition}[{Clarke generalized Jacobian, \cite[Def.~2.6.1]{clarke1990generalizedgradients}}]\label{def:clarkejac}
Let $G:\mb{R}^n \to \mb{R}^m$ be a vector-valued function with components $G=(g_1(x), g_2(x), \ldots, g_m(x))$, where each $g_i:\R^n \to \R$ is locally Lipschitz. The Clarke generalized Jacobian $\partial G(x)$ of $G$ at a point $x \in \R^n$ is defined as the convex hull of all limits of Jacobians along sequences of differentiable points converging to $x$:
\begin{equation}
\partial G(x) = \co\left\{\lim \jac G(x_k) : x_k \to x, \; G \text{ is differentiable at each } x_k\right\},
\label{app:eq:clarkejac}
\end{equation}
where $\jac G(x) \in \R^{m \times n}$ denotes the Jacobian matrix of partial derivatives of $G$ at a point $x$ where those partial derivatives exist.
\end{definition}

The Clarke generalized Jacobian is a nonempty, convex, compact subset of $\R^{m \times n}$, and the set-valued mapping $x \mapsto \partial G(x)$ is upper semicontinuous \cite[Prop.~2.6.2]{clarke1990generalizedgradients}.

\begin{theorem}[{Clarke's Generalized Chain Rule, \cite[Thm.~2.6.6]{clarke1990generalizedgradients}}]\label{thm:clarkechainrule}
Let $G: \R^n \to \R^m$ be locally Lipschitz near $x$ and let $f: \R^m \to \R$ be locally Lipschitz near $G(x)$. Define the function $h: \R^n \to \R$ by $h(x) = f(G(x))$. Then:
\begin{equation}\label{eq:chainrule_general}
\partial h(x) \subseteq \co\left\{ v M : v \in \partial f(G(x)), \; M \in \partial G(x) \right\},
\end{equation}
where $\partial f(G(x)) \subseteq \R^{1 \times m}$ denotes the Clarke generalized gradient of $f$ at $G(x)$, treated as a set of row vectors. If, in addition, $f$ is strictly differentiable at $G(x)$, $\partial f(G(x)) = \{\grad f(G(x))^\top\}$ and equality holds:
\begin{equation}\label{eq:chainrule_strict}
\partial h(x) = \left\{ v M : v \in \partial f(G(x)), \; M \in \partial G(x) \right\}.
\end{equation}
\end{theorem}
For a definition of strict differentiability, see \cite[Section 2.2]{clarke1990generalizedgradients}. For our purposes, it suffices to note that continuous differentiability implies strict differentiability.
\begin{proposition}[{\cite[Corollary to Prop.~2.2.1]{clarke1990generalizedgradients}}]\label{prop:strictdiff}
If $f: \R^m \to \R$ is continuously differentiable at $y$, then $f$ is strictly differentiable at $y$.
\end{proposition}

\begin{remark}[Gradient convention]\label{rem:gradient_convention}
Following the convention in \cite[Remark~2.6.3, Prop.~2.6.4]{clarke1990generalizedgradients}, the Clarke generalized gradient $\partial f(y)$ of a scalar-valued function $f: \R^m \to \R$ consists of \emph{row vectors} in $\R^{1 \times m}$. This ensures dimensional consistency in the chain rule \eqref{eq:chainrule_general}, where elements of $\partial f(G(x))$ multiply elements of $\partial G(x) \subseteq \R^{m \times n}$ from the left to yield elements in $\R^{1 \times n}$.

However, throughout this manuscript, we adopt the convention that gradients are \emph{column vectors}. Accordingly, we reformulate the chain rule \eqref{eq:chainrule_strict} as
\begin{equation}\label{eq:chainrule_column}
\partial h(x) = \left\{ M^\top \grad f(G(x)) : M \in \partial G(x) \right\},
\end{equation}
where elements of $\partial h(x)$ are now column vectors in $\R^n$. We use this column-vector convention for all Clarke generalized gradients hereafter.
\end{remark}

\subsection{(Partial) Goldstein stationary points}
We first introduce the $\delta$-Goldstein subdifferential in the general setting, which enlarges the Clarke generalized gradient by aggregating over a $\delta$-neighborhood. We then specialize to composite functions of the form $h(x) = f(G(x))$, where we introduce both the standard Goldstein subdifferential and our novel partial counterpart.

\begin{definition}[$\delta$-Goldstein subdifferential]
Let $h: \R^n \to \R$ be a locally Lipschitz function. For $\delta > 0$, the $\delta$-Goldstein subdifferential of $h$ at $x \in \R^n$ is defined as
\begin{equation*}
\goldstein h(x) = \co\left\{ \cup_{z \in \ball(x, \delta)} \partial h(z) \right\},
\end{equation*}
where $\ball(x, \delta) = \{z \in \R^n : \|z - x\| \leq \delta\}$ denotes the closed ball of radius $\delta$ centered at $x$ in the $\ell_2$-norm.

More generally, for a locally Lipschitz vector-valued function $G: \R^n \to \R^m$, we define the $\delta$-Goldstein subdifferential as
\begin{equation}\label{eq:goldstein_multivariate}
\goldstein G(x) = \co\left\{ \bigcup_{z \in \ball(x, \delta)} \partial G(z) \right\} \subseteq \R^{m \times n}.
\end{equation}
\end{definition}

For composite functions of the form $h(x) = f(G(x))$, we distinguish between two notions of Goldstein subdifferentials.
\begin{definition}[(Partial) Goldstein subdifferential for composite functions]\label{def:goldstein_composite}
Let $G: \R^n \to \R^m$ be locally Lipschitz near $x$ and let $f: \R^m \to \R$ be continuously differentiable near $G(x)$. Define $h(x) = f(G(x))$. For $\delta > 0$:
\begin{enumerate}[label=(\roman*)]
\item The \emph{$\delta$-Goldstein subdifferential} of $h$ at $x$ is
\begin{equation*}
\goldsteinfull h(x) = \co\left\{ \bigcup_{z \in \ball(x, \delta)} \partial h(z) \right\} = \co\left\{ \bigcup_{z \in \ball(x, \delta)} \left\{ M^\top \grad f(G(z)) : M \in \partial G(z) \right\} \right\},
\end{equation*}
where the second equality follows from \cref{thm:clarkechainrule,prop:strictdiff,rem:gradient_convention} and \cref{eq:chainrule_column}. 

\item The \emph{$\delta$-partial Goldstein subdifferential} of $h$ at $x$ is
\begin{equation*}
\goldsteinpartial h(x) = \left\{ M^\top \grad f(G(x)) : M \in \goldsteinfull G(x) \right\},
\end{equation*}
where $\goldsteinfull G(x) = \co\left\{ \bigcup_{z \in \ball(x, \delta)} \partial G(z) \right\}$.
\end{enumerate}
\end{definition}

We are now ready to present the notions of Goldstein and partial Goldstein stationarity.
\begin{definition}[(Partial) Goldstein stationary points for composite functions]\label{def:goldstein_stationary_composite}
Let $h(x) = f(G(x))$ be as in \cref{def:goldstein_composite}. A point $x \in \R^n$ is called:
\begin{enumerate}[label=(\roman*)]
\item A \emph{$(\delta, \varepsilon)$-Goldstein stationary point} if $\min\left\{ \|g\| : g \in \goldsteinfull h(x) \right\} \leq \varepsilon$.
\item A \emph{$(\delta, \varepsilon)$-partial Goldstein stationary point} if $\min\left\{ \|g\| : g \in \goldsteinpartial h(x) \right\} \leq \varepsilon$.
\end{enumerate}
\end{definition}

The distinction between Goldstein and partial Goldstein subdifferential lies in where the gradient of $f$ is evaluated: the subdifferential evaluates $\grad f$ at all points $G(z)$ for $z \in \ball(x, \delta)$, while the partial subdifferential evaluates $\grad f$ only at the center point $G(x)$. There is no general inclusion relationship between $\goldsteinfull h(x)$ and $\goldsteinpartial h(x)$, and neither stationarity notion in \cref{def:goldstein_stationary_composite} is a stronger or weaker stationarity condition than the other. We illustrate this with examples in \cref{sec:goldstein_examples}.

\subsection{Application to our problem setting}\label{sec:goldstein_setting}
We now specialize the general framework to our setting. Recall from \eqref{eq:setting} and \cref{ass:lipschitz} that our objective function is $F(x) = f(x, y^*(x))$, where $f: \R^{\dx} \times \R^{\dy} \to \R$ is smooth (i.e., differentiable with Lipschitz-continuous gradient) and $y^*: \R^{\dx} \to \R^{\dy}$ is Lipschitz continuous but potentially non-smooth.

To apply the chain rule framework, we write $F(x) = f(G(x))$ where $G: \R^{\dx} \to \R^{\dx + \dy}$ is defined by
\begin{equation*}
G(x) = \begin{pmatrix} x \\ y^*(x) \end{pmatrix}.
\end{equation*}
The function $G$ is locally Lipschitz since both the identity mapping $x \mapsto x$ and $y^*$ are Lipschitz. At points where $y^*$ is differentiable, the classical Jacobian of $G$ is
\begin{equation*}
\jac G(x) = \begin{pmatrix} I_{\dx} \\ \jac y^*(x) \end{pmatrix} \in \R^{(\dx + \dy) \times \dx},
\end{equation*}
where $I_{\dx}$ denotes the $\dx \times \dx$ identity matrix. The Clarke generalized Jacobian of $G$ is therefore
\begin{equation*}
\partial G(x) = \left\{ \begin{pmatrix} I_{\dx} \\ M \end{pmatrix} : M \in \partial y^*(x) \right\},
\end{equation*}
where $\partial y^*(x) = \co\left\{ \lim_{k \to \infty} \jac y^*(x_k) : x_k \to x, \; y^* \text{ differentiable at each } x_k \right\}$ according to \eqref{app:eq:clarkejac}.

Since $f$ is continuously differentiable by \cref{ass:lipschitz} (and hence strictly differentiable by \cref{prop:strictdiff}), applying \cref{thm:clarkechainrule} with the column-vector convention \eqref{eq:chainrule_column} yields:
\begin{align}
\partial F(x) &= \left\{ \begin{pmatrix} I_{\dx} & M^\top \end{pmatrix} \begin{pmatrix} \grad_x f(x, y^*(x)) \\ \grad_y f(x, y^*(x)) \end{pmatrix} : M \in \partial y^*(x) \right\} \nonumber \\
&= \left\{ \grad_x f(x, y^*(x)) + M^\top \grad_y f(x, y^*(x)) : M \in \partial y^*(x) \right\}. \label{eq:clarke_gradient_F}
\end{align}

Using this characterization, we can now express the Goldstein subdifferentials for our objective function $F(x)=f(x,y^*(x))$ for problem setting \eqref{eq:setting}. Under \cref{ass:lipschitz}, for $\delta > 0$:
\begin{enumerate}[label=(\roman*)]
\item The \emph{$\delta$-Goldstein subdifferential} of $F$ at $x \in \R^{\dx}$ is
\begin{equation}\label{eq:goldstein_full_F}
\goldsteinfull F(x) = \co\left\{ \bigcup_{z \in \ball(x, \delta)} \left\{ \grad_x f(z, y^*(z)) + M^\top \grad_y f(z, y^*(z)) : M \in \partial y^*(z) \right\} \right\}.
\end{equation}

\item The \emph{$\delta$-partial Goldstein subdifferential} of $F$ at $x \in \R^{\dx}$ is
\begin{equation}\label{eq:goldstein_partial_F}
\goldsteinpartial F(x) = \left\{ \grad_x f(x, y^*(x)) + M^\top \grad_y f(x, y^*(x)) : M \in \goldsteinfull y^*(x) \right\},
\end{equation}
where $\goldsteinfull y^*(x) = \co\left\{ \bigcup_{z \in \ball(x, \delta)} \partial y^*(z) \right\}$.
\end{enumerate}

These are the notions of stationarity used throughout our convergence analysis in \cref{sec:convergence_goldstein}. As illustrated by the examples in \cref{sec:goldstein_examples}, neither notion is in general stronger than the other -- a point may be a $(\delta,\varepsilon)$-Goldstein stationary point without being a $(\delta,\varepsilon)$-partial Goldstein stationary point, and vice versa.

\subsection{Comparison between Goldstein and partial Goldstein stationarity}\label{sec:goldstein_examples}
\Cref{ex:full_not_partial} exhibits a point that is $(\delta,\varepsilon)$-Goldstein stationary but not $(\delta,\varepsilon)$-partial Goldstein stationary. %

\begin{example}[$(\delta,\varepsilon)$-Goldstein stationary, but not $(\delta,\varepsilon)$-partial Goldstein stationary]\label{ex:full_not_partial}
Define $f(x,y) = (x-1)^2 + y$ and $y^*(x) = |x|$, so that $F(x) = (x-1)^2 + |x|$. The follower response $y^*$ is Lipschitz and non-differentiable at $x = 0$; $f$ is smooth. The composite $F$ is non-differentiable at $x = 0$ with one-sided derivatives $F'(0^+) = -1$ and $F'(0^-) = -3$, giving Clarke gradient 
\begin{equation*}
    \partial F(0) = \co\!\left\{\lim_{z \to 0^+} F'(z),\; \lim_{z \to 0^-} F'(z)\right\} = \co\{-1, -3\} = [-3, -1],
\end{equation*}
whereas $F$ is continuously differentiable on $(-\infty,0)$ and $(0,\infty)$, so $\partial F(x) = \{F'(x)\}$ for $x\neq 0$ by \cref{prop:strictdiff}, where
\begin{equation*}
F'(x) = \begin{cases}
2x - 1, & x > 0, \\
2x - 3, & x < 0.
\end{cases}
\end{equation*}
The global minimizer of $F$ is at $x = \tfrac{1}{2}$, where $F(\tfrac{1}{2}) = \tfrac{3}{4}$.

\smallskip
\textbf{Goldstein subdifferential.} For $\delta > 0$, $F$ is differentiable on $\ball(0,\delta)\setminus\{0\}$ with $F'(z) = 2z-1$ for $z > 0$ and $F'(z) = 2z-3$ for $z < 0$. Collecting Clarke gradients over $\ball(0,\delta)$, the Goldstein subdifferential at $x_0 = 0$ equals
\begin{equation*}
\goldsteinfull F(0) =  \co\left\{\cup_{z \in \ball(x, \delta)} \partial F(z) \right\}  = [-2\delta - 3,\; 2\delta - 1].
\end{equation*}
This interval contains zero if and only if $\delta \geq \tfrac{1}{2}$:
\begin{itemize}
\item[-] $\delta < \tfrac{1}{2}$: all elements are negative and $\min\{|g|:g\in\goldsteinfull F(0)\} = 1 - 2\delta > 0$.
\item[-] $\delta \geq \tfrac{1}{2}$: the ball reaches the minimiser $z = \tfrac{1}{2}$, so $0 \in \goldsteinfull F(0)$.
\end{itemize}
Hence $x = 0$ is a $(\delta,\varepsilon)$-Goldstein stationary point for every $\delta \geq \tfrac{1}{2}$ and $\varepsilon \geq 0$.

\smallskip
\textbf{Partial Goldstein subdifferential.} The Clarke Jacobian of $y^*(x) = |x|$ is $\jac y^*(z) = \mathrm{sgn}(z)$ for $z \neq 0$, so $\partial y^*(0) = \co\{-1,\,1\} = [-1,1]$. Since the Clarke Jacobian at any $z \in \ball(0,\delta)$ is a singleton contained in $[-1,1]$, the Goldstein subdifferential of $y^*$ at $x_0 = 0$ is identical to $\goldsteinfull y^*(0) = [-1,1]$ for all $\delta > 0$. Combined with the partial gradients $\grad_x f = -2$ and $\grad_y f = 1$ at $(x_0, y^*(x_0)) = (0, 0)$, the partial Goldstein subdifferential of $F$ at $x_0=0$ equals 
\begin{align*}
\goldsteinpartial F(0) &= \bigl\{ \grad_x f(0, y^*(0)) + M \cdot \grad_y f(0, y^*(0)) : M \in \goldsteinfull y^*(0) \bigr\} \\
&= \{ -2 + M : M \in [-1, 1] \} = [-3,\,-1].
\end{align*}
The minimum norm of $\goldsteinpartial F(0)$ equals $1$ for all $\delta > 0$, so $x = 0$ is \emph{not} a $(\delta,\varepsilon)$-partial Goldstein stationary point for any $\varepsilon < 1$, regardless of~$\delta$.

The partial Goldstein subdifferential is constant in $\delta$: the outer gradients $\grad_x f$ and $\grad_y f$ are fixed (and non-zero) at $x_0 = 0$, and the Goldstein subdifferential $\goldsteinfull y^*(0)$ does not grow with $\delta$. The Goldstein subdifferential, on the other hand, grows with $\delta$ and, once $\delta \geq \tfrac{1}{2}$, reaches past the minimiser $z = \tfrac{1}{2}$, causing it to contain zero.
\end{example}

\begin{figure}[t]
\centering
\begin{subfigure}[t]{0.5\textwidth}
\centering
\begin{tikzpicture}[baseline=(current axis.north)]
\begin{axis}[
    width=0.95\textwidth,
    height=0.80\textwidth,
    xlabel={$x$},
    ylabel={$F(x)$},
    title={$f(x,y)=(x-1)^2+y$,\;$y^*(x)=|x|$},
    title style={font=\small,yshift=10pt},
    label style={font=\small},
    tick label style={font=\footnotesize},
    grid=both,
    grid style={gray!30},
    xmin=-1.5, xmax=2,
    ymin=0, ymax=4.5,
    domain=-1.5:2,
    samples=200,
    axis lines=middle,
    every axis x label/.style={at={(ticklabel* cs:1)}, anchor=west},
    every axis y label/.style={at={(ticklabel* cs:1)}, anchor=south},
]
\addplot[blue, thick] {(x-1)^2 + abs(x)};
\addplot[only marks, mark=*, mark size=2.5pt, red] coordinates {(0, 1)};
\node[red, anchor=south east, font=\small] at (axis cs:0, 0.3) {$x_0 = 0$};
\addplot[only marks, mark=diamond*, mark size=3pt, black!60!green] coordinates {(0.5, 0.75)};
\node[black!60!green, anchor=north, font=\small] at (axis cs:0.5, 0.60) {$x = \tfrac{1}{2}$};
\end{axis}
\end{tikzpicture}
\caption{Example~\ref{ex:full_not_partial}: $F(x) = (x-1)^2 + |x|$}\label{fig:plot_goldsteinexamples_a}
\end{subfigure}%
\begin{subfigure}[t]{0.5\textwidth}
\centering
\begin{tikzpicture}[baseline=(current axis.north)]
\begin{axis}[
    width=0.95\textwidth,
    height=0.80\textwidth,
    xlabel={$x$},
    ylabel={$F(x)$},
    title={$f(y_1,y_2)=y_1-y_2^2$,\;$y^*(x)=(|x|,|x{-}1|)^\top$},
    title style={font=\small,yshift=10pt},
    label style={font=\small},
    tick label style={font=\footnotesize},
    grid=both,
    grid style={gray!30},
    xmin=-1.5, xmax=2.5,
    ymin=-3, ymax=2.5,
    domain=-1.5:2.5,
    samples=200,
    axis lines=middle,
    every axis x label/.style={at={(ticklabel* cs:1)}, anchor=west},
    every axis y label/.style={at={(ticklabel* cs:1)}, anchor=south},
]
\addplot[blue, thick] {abs(x) - (x-1)^2};
\addplot[only marks, mark=*, mark size=2.5pt, red] coordinates {(0, -1)};
\node[red, anchor=south east, font=\small] at (axis cs:-0.05, -1.0) {$x_0 = 0$};
\addplot[only marks, mark=square*, mark size=2.5pt, orange] coordinates {(1, 1)};
\node[orange, anchor=south east, font=\small] at (axis cs:0.95, 1.15) {$z=1$};
\addplot[only marks, mark=triangle*, mark size=3pt, black!60!green] coordinates {(1.5, 1.25)};
\node[black!60!green, anchor=south west, font=\small] at (axis cs:1.55, 1.15) {$x = \tfrac{3}{2}$};
\end{axis}
\end{tikzpicture}\vspace*{5.2mm}
\caption{Example~\ref{ex:partial_not_full}: $F(x) = |x| - (x-1)^2$}\label{fig:plot_goldsteinexamples_b}
\end{subfigure}
\caption{Graphs of $F(x) = f(x, y^*(x))$ for Examples~\ref{ex:full_not_partial} and~\ref{ex:partial_not_full}. \textbf{Left:} $F(x) = (x-1)^2+|x|$ is non-differentiable at $x_0=0$ (red dot). For $\delta \geq \tfrac{1}{2}$, the Goldstein subdifferential reaches the minimizer $z = \tfrac{1}{2}$ (green diamond) and contains zero, while the partial Goldstein subdifferential equals $[-3,-1]$ for all $\delta > 0$. \textbf{Right:} $F(x) = |x|-(x-1)^2$ is non-differentiable at $x_0=0$ (red dot). The non-differentiability of $y^*$ at $z=1$ (orange square) is invisible in $F$.
For $\delta \in [1, \tfrac{3}{2})$, the partial Goldstein subdifferential spans $[-1,3]$ and contains zero due to the non-differentiability of $y^*$ at $z=1$ entering $\ball(0,\delta)$, even though this does not propagate to $F$; the Goldstein subdifferential remains strictly positive. The genuine critical point of $F$ is at $z = \tfrac{3}{2}$ (green triangle), which enters the Goldstein subdifferential for $\delta \ge \tfrac{3}{2}$.
}
\label{fig:plot_goldsteinexamples}
\end{figure}
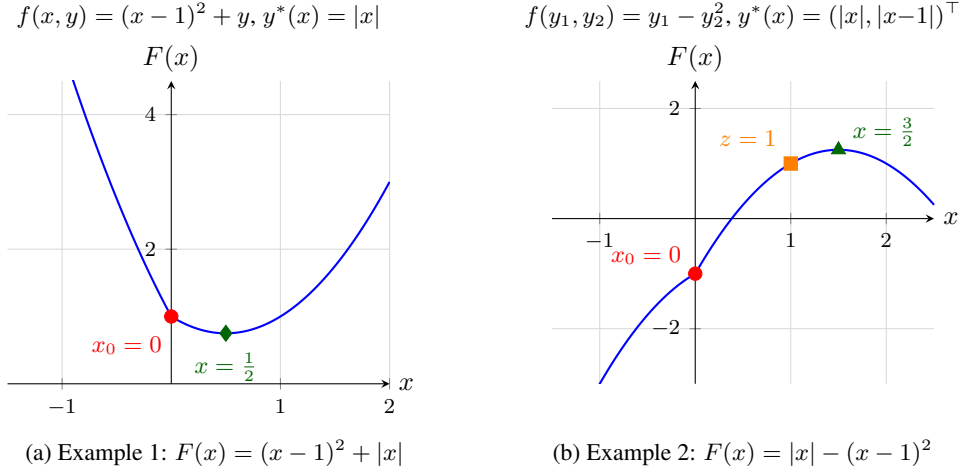

\Cref{ex:partial_not_full} exhibits a point that is $(\delta,\varepsilon)$-partial Goldstein stationary  but not $(\delta,\varepsilon)$-Goldstein stationary.

\begin{example}[$(\delta,\varepsilon)$-partial Goldstein stationary, but not $(\delta,\varepsilon)$-Goldstein stationary]\label{ex:partial_not_full}
Define $f(y_1,y_2) = y_1 - y_2^2$ and $y^*(x) = \bigl(|x|,\,|x-1|\bigr)^\top$, so that $F(x) = |x| - (x-1)^2$. The follower response $y^*$ is Lipschitz and non-differentiable at $x = 0$ and $x = 1$. However, the non-differentiability of $y^*$ at $z = 1$ does not propagate to $F$, which is non-differentiable only at $x = 0$. %
The one-sided derivatives are $F'(0^+) = 3$ and $F'(0^-) = 1$, giving Clarke gradient $\partial F(0) = [1,3]$.

\smallskip
\textbf{Goldstein subdifferential.} $F$ is continuously differentiable on $(-\infty,0)$ and $(0,\infty)$, so $\partial F(x) = \{F'(x)\}$ for $x\neq 0$ by \cref{prop:strictdiff}, with:
\[F'(x) = \begin{cases}
3 - 2x, & x > 0, \\
1 - 2x, & x < 0,
\end{cases}\]
both are strictly positive near $x = 0$. Collecting Clarke gradients over $\ball(0,\delta)$:
\begin{itemize}
\item[-] $\delta \in (0,1)$: all gradients in the ball lie in $[1,3]$, so $\goldsteinfull F(0) = [1,3]$ and $\min\{|g|:g\in\goldsteinfull F(0)\} = 1 > 0$.
\item[-] $\delta \in [1, \tfrac{3}{2})$: the range of $F'$ over the ball extends below $1$ (from the positive side) and above $3$ (from the negative side), giving $\goldsteinfull F(0) = [3-2\delta,\,1+2\delta]$ with $\min\{|g|\} = 3 - 2\delta > 0$.
\item[-] $\delta \geq \tfrac{3}{2}$: the ball reaches $z = \tfrac{3}{2}$ where $F'(\tfrac{3}{2}) = 0$, so $0 \in \goldsteinfull F(0)$.
\end{itemize}
Hence $x = 0$ is not a $(\delta,\varepsilon)$-Goldstein stationary point for $\delta \in (0,1)$ and $\varepsilon < 1$, nor for $\delta \in [1, \tfrac{3}{2})$ and $\varepsilon < 3-2\delta$.

\smallskip
\textbf{Partial Goldstein subdifferential.} The Jacobian of $y^*$ at differentiable points $z \notin \{0,1\}$ is
\[
\jac y^*(z) = \begin{cases}(-1,-1)^\top, & z < 0,\\ (1,-1)^\top, & 0 < z < 1,\\ (1,\;1)^\top, & z > 1,\end{cases}
\]
with Clarke Jacobians $\partial y^*(0) = \co\{(-1,-1)^\top,(1,-1)^\top\}$ and $\partial y^*(1) = \co\{(1,-1)^\top,(1,1)^\top\}$. Whether the point $z = 1$ of non-differentiability of $y^*$ falls within $\ball(0,\delta)$ determines the Goldstein subdifferential of $y^*$:
\begin{itemize}
\item[-] $\delta < 1$: only the non-differentiable point $z = 0$ lies in $\ball(0,\delta)$, so
    \[
    \goldsteinfull y^*(0) = \co\!\left\{\begin{pmatrix}-1\\-1\end{pmatrix},\begin{pmatrix}1\\-1\end{pmatrix}\right\} = \left\{\begin{pmatrix}a\\-1\end{pmatrix} : a \in [-1,1]\right\}.
    \]
    The partial subdifferential is $\goldsteinpartial F(0) = \{a\cdot 1 + (-1)\cdot(-2) : a \in [-1,1]\} = [1,3]$, with $\min\{|g|\} = 1 > 0$.

\item[-] $\delta \geq 1$:  the point $z = 1$ of non-differentiability of $y^*$ enters the ball. The Goldstein subdifferential of $y^*$ becomes
    \[
    \goldsteinfull y^*(0) = \co\!\left\{\begin{pmatrix}-1\\-1\end{pmatrix},\begin{pmatrix}1\\-1\end{pmatrix},\begin{pmatrix}1\\1\end{pmatrix}\right\} = \left\{\begin{pmatrix}a\\b\end{pmatrix}: -1 \leq b \leq a \leq 1\right\}.
    \]
    Combined with partial derivatives $\grad_x f = 0$ and $\grad_y f(0,1) = (1,-2)^\top$ at $x_0 = 0$, $y^*(0) = (0,1)^\top$, the partial subdifferential is $\goldsteinpartial F(0) = \{a - 2b : -1 \leq b \leq a \leq 1\}$. Evaluating at the three vertices:
    \begin{align*}
    (-1,-1) &\;\mapsto\; -1 + 2 = 1, \\
    (1,-1) &\;\mapsto\; 1 + 2 = 3, \\
    (1,\;1) &\;\mapsto\; 1 - 2 = -1.
    \end{align*}
    Since the vertex values have mixed signs, $\goldsteinpartial F(0) = [-1,3]$, which contains zero.
\end{itemize}
Hence $x = 0$ is a $(\delta,\varepsilon)$-partial Goldstein stationary point for every $\delta \geq 1$ and $\varepsilon \geq 0$.

For $\delta \in [1, \tfrac{3}{2})$: $x = 0$ is $(\delta,\varepsilon)$-partial Goldstein stationary for all $\varepsilon \geq 0$, but not $(\delta,\varepsilon)$-Goldstein stationary for $\varepsilon < 3-2\delta$. For $\delta \geq \tfrac{3}{2}$, $x = 0$ is both $(\delta,\varepsilon)$-partial Goldstein stationary and $(\delta,\varepsilon)$-Goldstein stationary, for all $\varepsilon \geq 0$.

The mechanism is the reverse of Example~\ref{ex:full_not_partial}: the non-differentiability of $y^*$ at $z = 1$ is invisible in $F$, but once $\delta \geq 1$ it contributes the Clarke Jacobian $\partial y^*(1)$ to $\goldsteinfull y^*(0)$, enriching the partial subdifferential enough to include zero. The Goldstein subdifferential, collecting only genuine Clarke gradients of $F$ in $\ball(0,\delta)$, remains strictly positive until the ball reaches the true critical point at $z = \tfrac{3}{2}$.
\end{example}

Taken together, the two examples illustrate structurally different phenomena. In Example~\ref{ex:full_not_partial}, the Goldstein subdifferential certifies approximate stationarity by reaching a genuine minimizer of $F$ within the $\delta$-ball; the partial subdifferential correctly identifies non-stationarity because the outer gradients $\grad_x f$ and $\grad_y f$ at $x_0$ point consistently away from zero. In \Cref{ex:partial_not_full}, the roles reverse: the partial subdifferential is affected by the non-differentiability of $y^*$ at $z = 1$, which does not propagate to $F$ but contributes Clarke Jacobians to $\goldsteinfull y^*(0)$ once $\delta \geq 1$, shifting the partial subdifferential to include zero. The Goldstein subdifferential correctly reports non-stationarity until it reaches the true critical point of $F$ at $z = \tfrac{3}{2}$.
\section{Proof of \texorpdfstring{Theorem~\ref{thm:convergencegoldsteinpoint}}{Theorem~\ref{thm:convergencegoldsteinpoint}}}

This section contains the proofs of \cref{thm:convergencegoldsteinpoint}. \Cref{sec:smoothing-properties} establishes smoothing properties of the equilibrium response $y_\mu$. \Cref{sec:surrogate-smoothness} establishes smoothness of the surrogate $F_\mu$ and analyzes the bias of the gradient estimator $g_t$. \Cref{sec:convergence_goldstein} establishes convergence to $(\delta,\varepsilon)$-Goldstein and $(\delta,\varepsilon)$-partial Goldstein stationary points. We begin with a proof overview.

\paragraph{Step 1: Smoothing properties of $y_\mu$ (\Cref{sec:smoothing-properties}).} We extend classical smoothing results for scalar functions to the vector-valued setting $\R^{\dx} \to \R^{\dy}$, showing that the smoothed equilibrium response $y_\mu(x)$ is continuously differentiable with an explicit expression for its Jacobian $\jac y_\mu(x)$ (\Cref{lem:ysmooth-differentiable-lipschitz}). We establish that $y_\mu(x)$ inherits the Lipschitz continuity of $y^*$ (\Cref{lem:properties-smoothing}). Further, we show that in the vector-valued and constrained setting $\R^{\dx} \to \Y$, $y_\mu(x)$ remains in $\Y$ if $\Y$ is closed and convex (\Cref{rem:ymuiny}), ensuring the smoothed followers' response $y_\mu$ remains feasible in our setting.

\paragraph{Step 2: Smoothness of $F_\mu$ and estimator properties (\Cref{sec:surrogate-smoothness}).} Combining smoothness of $f$ and smoothness of $y_\mu$, we show that $F_\mu(x) = f(x, y_\mu(x))$ is differentiable and $L_F$-smooth (\cref{lem:descent-Fsmooth}) and satisfies $|F_\mu(x) - F(x)| \leq L_f L_y \mu$ (\cref{pro:Fsmoothbound}). 
We then establish properties of the estimator $g_t$ used in \algpzos. The Jacobian estimator $H_t$ is unbiased for $\jac y_\mu(x_t)$ (\cref{lem:jacobianunbiasedestimator}), and we bound the second moment $\E[\norm{g_t}^2 \mid x_t]$ (\cref{lem:estimatorsquared}). %
Since $y_\mu(x_t)$ is not accessible in closed form, $g_t$ evaluates the partial derivatives of $f$ at $(x_t, y^*(x_t))$ rather than $(x_t, y_\mu(x_t))$, making it a biased estimator of $\grad F_\mu(x_t)$; we show this bias is bounded by $\mc{O}(\mu)$ (\cref{lem:difference-estimate-gradfsmooth}).
\paragraph{Step 3: Convergence to (partial) Goldstein stationarity (\Cref{sec:convergence_goldstein}).} 
Existing zero-order methods set the smoothing parameter $\mu$ equal to the desired radius $\delta$ of the $\delta$-Goldstein subdifferential, then smooth the composed objective $F$ directly via $\Tilde{F}_\mu = \E_{u \sim \U(\ball^{\dx})}[F(x + \mu u)] = \E_{u\sim\U(\ball^{\dx})}[f(x+\mu u,y^*(x+\mu u)]$. $\tilde{g}_t$ of \algzos is an unbiased estimator of $\grad \tilde{F}_\mu$, and used to show convergence of $\Tilde{F}$ via the descent lemma. Containment of $\Tilde{F}_\mu$ (with $\mu=\delta$) in $\goldsteinfull F(x)$ immediately yields convergence of \algzos to $(\delta,\varepsilon)$-Goldstein stationary points.

This route is unavailable to us. Our surrogate function $F_\mu = f(x,y_\mu)$ does not smooth $F$ directly, as our objective is to utilize knowledge of the partial gradients of $f$. Since $y_\mu(x)$ is not accessible in closed form, our estimator $g_t$ evaluates the partial derivatives of $f$ at $(x_t, y^*(x_t))$ rather than $(x_t,y_\mu(x_t))$, so our estimator is biased for $\grad F_\mu$. This bias is fundamental and cannot be eliminated with a different estimator of $\grad F_\mu$: even an unbiased estimator $\hat{y}$ of $y_\mu(x_t)$ would yield a biased estimator of $\grad F_\mu(x_t)$, since $\E[\grad_x f(x, \hat{y})] \neq \grad_x f(x, y_\mu(x))$ and $\E[\grad_y f(x_t, \hat{y})] \neq \grad_y f(x_t, y_\mu(x_t))$ even if $\E[\hat{y}] = y_\mu(x)$ \footnote{This follows since, for a general nonlinear mapping $h$ and a random variable $z$, $\E[h(z)] \neq h(\E[z])$.} and introducing bias in the estimation of $\grad_x f$ and $\grad_y f$ (which are the parts where we want to exploit knowledge) defeats the purpose of exploiting information about the known gradients. We circumvent this by decoupling $\mu$ from $\delta$: choosing $\mu$ small relative to $\varepsilon$ ensures the $\mc{O}(\mu)$ bias bound from \cref{lem:difference-estimate-gradfsmooth} stays below the target tolerance $\varepsilon$, while $\mu \leq \delta$ ensures that $\goldsteinfullmu \subseteq \goldsteinfull F$ and $\goldsteinpartialmu \subseteq \goldsteinpartial F$.

For the \emph{partial} Goldstein subdifferential, we first show via a strict hyperplane separation argument that $\jac y_\mu(x) \in \goldsteinfullmu y^*(x)$ (\cref{lem:jacobiangoldstein}). Since $\goldsteinpartialmu F(x)$ fixes the partial gradients of $f$ at the center $(x, y^*(x))$ and varies only the Clarke Jacobian of $y^*$, this containment implies $\E[g_t \mid x_t] \in \goldsteinpartial F(x_t)$ directly for $\delta \geq \mu$. Combining this with the descent lemma for $F_\mu$ yields convergence to $(\delta, \varepsilon)$-partial Goldstein stationarity (\cref{thm:partial_goldstein}), where $\mu$ is chosen sufficiently small relative to $\varepsilon$ to ensure the bias term $\mc{O}(\mu)$ remains below the target tolerance $\varepsilon$.

For the Goldstein subdifferential, $\E[g_t \mid x_t]$ does not lie in $\goldsteinfullmu F(x_t)$ directly, since the latter also varies the partial gradients of $f$ over the $\mu$-ball. Instead, we apply Carathéodory's theorem to show that $\grad F_\mu(x)$ lies within distance $\mc{O}(\mu)$ of $\goldsteinfullmu F(x)$ (\cref{lem:distancefullgoldstein}), and $\mu$ is chosen to balance the distance of $\grad F_\mu(x)$ to $\goldsteinfullmu F(x)$ and the bias of $g$ for $\grad F_\mu(x)$, ensuring the sum of both terms is below the target tolerance $\varepsilon$ (\cref{thm:full_goldstein}).
\subsection{Smoothing Properties of the Equilibrium Response}\label{sec:smoothing-properties}
We establish that the smoothed response $y_\mu(x) = \E_{u \sim \U(\ball^{\dx})}[y^*(x + \mu u)]$ is continuously differentiable and derive an explicit expression for its Jacobian matrix $\jac y_\mu(x)$. Subsequently, we establish several fundamental properties of $y_\mu$ and $\jac y_\mu$. Our approach extends known smoothing results for scalar-valued functions to the vector-valued setting $\R^{\dx} \to \R^{\dy}$.

We build on the following classical result for smoothing scalar-valued functions.
\begin{theorem}[{\cite[Lemma 1(i)]{cui2023mpec}, \cite[Lemma 2.1]{flaxman2005}}]\label{thm:smoothing-gradient}
    Let $h: \R^{\dx} \to \R$ be a continuous function, and let $\mu > 0$ be a given scalar. Define $h_\mu: \R^{\dx} \to \R$ by
    \begin{equation*}
        h_\mu(x) = \E_{u \sim \U(\ball^{\dx})}[h(x + \mu u)].
    \end{equation*}
    Then $h_\mu$ is continuously differentiable. In particular, for any $x \in \R^{\dx}$, the gradient $\grad_x h_\mu(x) \in \R^{\dx}$ is given by
    \begin{equation*}
        \grad_x h_\mu(x) = \frac{\dx}{\mu} \E_{v \sim \U(\sphere^{\dx})}[h(x + \mu v) v].
    \end{equation*}
\end{theorem}

We extend uniform smoothing to the vector-valued and constrained case, and define, for a vector-valued and constrained mapping $y^*: \R^{\dx} \to \Y \subseteq \mb{R}^{\dy}$ and a given scalar $\mu > 0$, the mapping $y_\mu:\mb{R}^{\dx}\to\mb{R}^{\dy}$ by
\begin{equation}\label{eq:ysmoothvector}
    y_\mu(x) = \E_{u \sim \U(\ball^{\dx})}[y^*(x + \mu u)] = \begin{pmatrix} \E_{u \sim \U(\ball^{\dx})}[y^*_1(x + \mu u)] \\ \vdots \\ \E_{u \sim \U(\ball^{\dx})}[y^*_{\dy}(x + \mu u)] \end{pmatrix}, 
\end{equation}
where $u$ is uniformly distributed over the unit ball $\ball^{\dx} = \{u \in \R^{\dx} : \norm{u} \leq 1\}$ and $y^*_i(x)$ denotes the $i$-th component of $y^*(x) \in \Y \subseteq \R^{\dy}$. 

Since $y_\mu(x)$ is defined as an expectation, it is not immediate that 
$y_\mu(x)$ remains in $\Y$; in fact, averages of elements of an arbitrary set $\Y$ need not lie in $\Y$ themselves. Below, we show that if $\Y$ is closed and convex, the smoothed response $y_\mu(x)$ remains in the set $\Y$.%

\begin{lemma}\label{rem:ymuiny}
    Under \cref{ass:lipschitz}, let $y^*: \R^{\dx} \to \Y$ be the unique response of the followers for given $x\in\mb{R}^{\dx}$, where $y^*$ is $L_y$-Lipschitz continuous and $\Y \subseteq \R^{\dy}$ is closed and convex. Then $y_\mu(x) = \E_{u \sim \U(\ball^{\dx})}[y^*(x + \mu u)] \in \Y$.
\end{lemma}
\begin{proof}
Fix any $x \in \R^{\dx}$ and $\mu > 0$. By assumption, $y^*(z) \in \Y$ for all $z \in \R^{\dx}$. Since $y^*$ is $L_y$-Lipschitz by \cref{ass:lipschitz}, it is continuous, and therefore the integral
\begin{equation}\label{eq:ymuintegral}
    y_\mu(x) = \E_{u \sim \U(\ball^{\dx})}[y^*(x + \mu u)] = \frac{1}{\vol(\ball^{\dx})} \int_{\ball^{\dx}} y^*(x + \mu u) \, du
\end{equation}
is well-defined and finite (since it is an integral of a continuous function over a compact set).

Suppose for contradiction that $y_\mu(x) \notin \Y$. Since $\Y$ is closed and convex, and $\{y_\mu(x)\}$ (a singleton) is a closed, convex, and bounded (therefore compact and convex) set disjoint from $\Y$, the strict hyperplane separation theorem guarantees the existence of $b \in \R^{\dy}$ and $\alpha \in \R$ such that:
\begin{equation} \label{eq:hyperplaneseparationymu}
\langle b, y_\mu(x) \rangle > \alpha \quad \text{and} \quad \langle b, y \rangle < \alpha \quad \forall y \in \Y.
\end{equation}
Using \eqref{eq:ymuintegral}, we express the left-hand side of the first inequality in \eqref{eq:hyperplaneseparationymu} as
\begin{equation*}
    \langle b, y_\mu(x) \rangle = \left\langle b, \frac{1}{\vol(\ball^{\dx})} \int_{\ball^{\dx}} y^*(x + \mu u) \, du \right\rangle = \frac{1}{\vol(\ball^{\dx})} \int_{\ball^{\dx}} \langle b, y^*(x + \mu u) \rangle \, du.\\
\end{equation*}
Since $y^*(z) \in \Y$ for all $z \in \R^{\dx}$, thus $y^*(x + \mu u) \in \Y$ for all $u \in \ball^{\dx}$, the second inequality in \eqref{eq:hyperplaneseparationymu} implies $\langle b, y^*(x + \mu u) \rangle < \alpha$ for all $u \in \ball^{\dx}$. Therefore, we can bound the integral:
\begin{equation*}
    \frac{1}{\vol(\ball^{\dx})}\int_{\ball^{\dx}} \langle b, y^*(x + \mu u) \rangle \, du \leq \frac{1}{\vol(\ball^{\dx})}\int_{\ball^{\dx}} \alpha \, du = \frac{1}{\vol(\ball^{\dx})} \alpha \vol(\ball^{\dx}) = \alpha.
\end{equation*}
This implies $\langle b, y_\mu(x) \rangle \leq \alpha$, which contradicts $\langle b, y_\mu(x) \rangle > \alpha$. Therefore, we conclude that $y_\mu(x) \in \Y$.
\end{proof}
 
Equipped with \cref{rem:ymuiny}, we can generalize \cref{thm:smoothing-gradient} to vector-valued and constrained mappings, and provide an explicit expression for the Jacobian of the smoothed mapping.
\begin{lemma}[Jacobian of the smoothed response]\label{lem:ysmooth-differentiable-lipschitz}
Let $y^*: \R^{\dx} \to \Y$ be the unique response of the followers for given $x\in\mb{R}^{\dx}$ and suppose that $y^*$ is $L_y$-Lipschitz continuous, where $\Y \subseteq \R^{\dy}$ is closed and convex. For a given scalar $\mu > 0$, define $y_\mu: \R^{\dx} \to \R^{\dy}$ by $y_\mu(x) = \E_{u \sim \U(\ball^{\dx})}[y^*(x + \mu u)]$ (see \cref{eq:ysmoothvector}).
Then $y_\mu$ is continuously differentiable over $\R^{\dx}$. Moreover, for any $x \in \R^{\dx}$, the Jacobian matrix $\jac y_\mu(x) \in \R^{\dy \times \dx}$ is given by 
\begin{equation}\label{eq:jacobian-ymu}
    \jac y_\mu(x) = \frac{\dx}{\mu} \E_{v \sim \U(\sphere^{\dx})}\Bigl[y^*(x + \mu v)\,v^\top\Bigr],
\end{equation} 
where $\sphere^{\dx} = \{v \in \R^{\dx} : \norm{v} = 1\}$ denotes the unit sphere.
\end{lemma}
\begin{proof}
Let $y^*_i(x)$, for $i \in \{1, \ldots, \dy\}$, denote the $i$-th component of the vector $y^*(x) \in \Y \subseteq \R^{\dy}$, and let $[y_\mu]_i$ denote the $i$-th component of $y_\mu$. We first recall that since $\Y$ is closed and convex, $y_\mu(x) \in \Y$ by \cref{rem:ymuiny}. %

We establish continuous differentiability of $y_\mu$ in two steps. First, we show that all partial derivatives $\frac{\partial [y_\mu]_i}{\partial x_j}(x)$ for $i \in \{1, \ldots, \dy\}$ and $j \in \{1, \ldots, \dx\}$ exist and derive explicit expressions for them. Second, we verify that all partial derivatives are continuous. Together, these two steps establish that $y_\mu$ is continuously 
differentiable, with Jacobian $\jac y_\mu(x)$ given by~\eqref{eq:jacobian-ymu}.

\medskip
\noindent\textbf{Step 1: Existence of partial derivatives.}
Since $y^*(x): \R^{\dx} \to \Y$ is $L_y$-Lipschitz continuous, each component $y^*_i$ is also $L_y$-Lipschitz continuous for all $x_1, x_2 \in \R^{\dx}$,
\begin{equation*}
    |y^*_i(x_1) - y^*_i(x_2)| \leq \norm{y^*(x_1) - y^*(x_2)} \leq L_y \norm{x_1 - x_2}.
\end{equation*}
Since $[y_\mu(x)]_i = \E_{u \sim \U(\ball^{\dx})}[y^*_i(x + \mu u)]$, we can apply \cref{thm:smoothing-gradient} to each component $[y_\mu]_i$ of $y_\mu$ to conclude that $[y_\mu]_i$ is differentiable over $\R^{\dx}$, with gradient
\begin{equation*}
    \grad_x [y_\mu]_i(x) = \frac{\dx}{\mu} \E_{v \sim \U(\sphere^{\dx})}[y^*_i(x + \mu v)\,v].    
\end{equation*}
The $j$-th component of $\grad_x [y_\mu]_i(x)$ is the $j$-th partial derivative of $[y_\mu]_i$:
\begin{equation}\label{eq:smoothpartial}
    \frac{\partial [y_\mu]_i}{\partial x_j}(x) = \frac{\dx}{\mu} \E_{v \sim \U(\sphere^{\dx})}[y^*_i(x + \mu v)\,v_j],
\end{equation}
where $v_j$ denotes the $j$-th component of the vector $v \in \sphere^{\dx}$.  

We observe that 
\begin{equation*}
    \E_{v \sim \U(\sphere^{\dx})}[y^*_i(x + \mu v)v] = \int_{\sphere^{\dx}} y^*_i(x + \mu v) \, v \, p(v) \, dv,
\end{equation*}
where $p(v)$ is the probability density function of the uniform distribution over the unit sphere $\sphere^{\dx}$. Since $y^*_i$ is continuous and $\sphere^{\dx}$ is compact, the expectation in \eqref{eq:smoothpartial} is well-defined, and the partial derivative in \eqref{eq:smoothpartial} exists.

The Jacobian matrix $\jac y_\mu(x)$ is obtained by collecting all partial derivatives $\frac{\partial [y_\mu]_i}{\partial x_j}(x)$ and has the form:
\begin{equation*}
    \jac y_\mu(x) = \begin{pmatrix} \frac{\partial [y_\mu]_1}{\partial x_1}(x) & \cdots & \frac{\partial [y_\mu]_1}{\partial x_{\dx}}(x) \\ \vdots & \ddots & \vdots \\ \frac{\partial [y_\mu]_{\dy}}{\partial x_1}(x) & \cdots & \frac{\partial [y_\mu]_{\dy}}{\partial x_{\dx}}(x) \end{pmatrix},
\end{equation*} 
with each entry $[\jac y_\mu(x)]_{ij}$ given by \eqref{eq:smoothpartial}. 

Utilizing the identity $y^*_i(x + \mu v) \cdot v_j = [y^*(x + \mu v) \, v^\top]_{ij}$, we obtain the compact representation
\begin{equation*}
    \jac y_\mu(x) = \frac{\dx}{\mu} \E_{v \sim \U(\sphere^{\dx})}\Bigl[y^*(x + \mu v)\,v^\top\Bigr],
\end{equation*}
where the expectation is understood component-wise.

\medskip
\noindent\textbf{Step 2: Continuity of partial derivatives.}
The existence of partial derivatives and the Jacobian alone does not guarantee continuous differentiability. We therefore verify continuity of all partial derivatives.

Fix $i \in \{1, \ldots, \dy\}$ and $j \in \{1, \ldots, \dx\}$. For any $x, z \in \R^{\dx}$:
\begin{align*}
    \left|\frac{\partial [y_\mu]_i}{\partial x_j}(x) - \frac{\partial [y_\mu]_i}{\partial x_j}(z)\right| 
    & =
    \left|\frac{\dx}{\mu} \E_{v \sim \U(\sphere^{\dx})}\left[y^*_i(x + \mu v) v_j\right] - \frac{\dx}{\mu} \E_{v \sim \U(\sphere^{\dx})}\left[y^*_i(z + \mu v) v_j\right]\right| \\
    &= \left|\frac{\dx}{\mu} \E_{v \sim \U(\sphere^{\dx})}\left[(y^*_i(x + \mu v) - y^*_i(z + \mu v)) v_j\right]\right| \\
    &\leq \frac{\dx}{\mu} \E_{v \sim \U(\sphere^{\dx})}\left[|y^*_i(x + \mu v) - y^*_i(z + \mu v)| \cdot |v_j|\right] \\
    &\leq \frac{\dx}{\mu} \E_{v \sim \U(\sphere^{\dx})}\left[|y^*_i(x + \mu v) - y^*_i(z + \mu v)|\right],
\end{align*}
where the last inequality uses $|v_j| \leq \norm{v} = 1$ for $v \in \sphere^{\dx}$. By the Lipschitz continuity of $y^*_i$:
\begin{equation*}
    |y^*_i(x + \mu v) - y^*_i(z + \mu v)| \leq L_y \norm{(x + \mu v) - (z + \mu v)} = L_y \norm{x - z}, 
\end{equation*}
and therefore
\begin{equation*}
    \left|\frac{\partial [y_\mu]_i}{\partial x_j}(x) - \frac{\partial [y_\mu]_i}{\partial x_j}(z)\right| \leq \frac{\dx L_y}{\mu} \norm{x - z}.
\end{equation*}
For any $\varepsilon > 0$, choosing $\delta = \frac{\varepsilon \mu}{\dx L_y}$ ensures that for any $x, z$ with $\norm{x - z} < \delta$,
\begin{equation*}
    \left|\frac{\partial [y_\mu]_i}{\partial x_j}(x) - \frac{\partial [y_\mu]_i}{\partial x_j}(z)\right| < \varepsilon.
\end{equation*}
This establishes uniform continuity, and therefore continuity, of $\frac{\partial [y_\mu]_i}{\partial x_j}$ at every point $x \in \R^{\dx}$. Since all partial derivatives exist everywhere and are continuous, $y_\mu: \R^{\dx} \to \Y$ is continuously differentiable. 
\end{proof}

The following lemma summarizes key properties of the smoothed equilibrium response $y_\mu(x)$. These results extend known smoothing properties for scalar functions to the multivariate setting $\R^{\dx} \to \R^{\dy}$.

\begin{lemma}[Properties of the smoothed equilibrium response]\label{lem:properties-smoothing}
Let $y^*: \R^{\dx} \to \Y$ be the unique followers' response for given $x\in\mb{R}^{\dx}$ and suppose that $y^*$ is $L_y$-Lipschitz continuous, where $\Y \subseteq \R^{\dy}$ is closed and convex. Let $y_\mu(x)$ be defined as in \eqref{eq:ysmooth}, and let $\jac y_\mu(x) = \frac{\dx}{\mu} \E_{v \sim \U(\sphere^{\dx})} \Bigl[y^*(x + \mu v) \,v^\top\Bigr]$. Then, for any $x, z \in \R^{\dx}$:
\begin{enumerate}[label=(\roman*)]
    \item $\norm{y_\mu(x) - y^*(x)} \leq L_y \mu$.
    \item $\norm{y_\mu(x) - y_\mu(z)} \leq L_y \norm{x - z}$.
    \item $\norm{\jac y_\mu(x) - \jac y_\mu(z)} \leq \frac{k_1 L_y \sqrt{\dx}}{\mu} \norm{x - z}$, and $k_1>0$ is a constant.
\end{enumerate}
\end{lemma}

\begin{proof}
\textbf{Part (i):} For any $x \in \R^{\dx}$, by the definition of $y_\mu(x)$ and the Lipschitz continuity of $y^*$:
\begin{align*}
    \norm{y_\mu(x) - y^*(x)} &= \left\|\E_{u \sim \U(\ball^{\dx})}[y^*(x + \mu u)] - y^*(x)\right\| \\
    &= \left\|\E_{u \sim \U(\ball^{\dx})}[y^*(x + \mu u) - y^*(x)]\right\| \\
    &\leq \E_{u \sim \U(\ball^{\dx})}[\norm{y^*(x + \mu u) - y^*(x)}] \quad \text{(Jensen's inequality)} \\
    &\leq \E_{u \sim \U(\ball^{\dx})}[L_y \norm{\mu u}] \quad \text{(Lipschitz continuity of } y^*\text{)} \\
    &= L_y \mu \E_{u \sim \U(\ball^{\dx})}[\norm{u}] \\
    &= L_y \mu \cdot \frac{\dx}{\dx + 1} \quad \text{(see \cite[Exercise 0.8]{vershynin2018highdimensionalprobability})}\\
    &\leq L_y \mu.
\end{align*}

\textbf{Part (ii):} For any $x, z \in \R^{\dx}$:
\begin{align*}
    \norm{y_\mu(x) - y_\mu(z)} &= \left\|\E_{u \sim \U(\ball^{\dx})}[y^*(x + \mu u)] - \E_{u \sim \U(\ball^{\dx})}[y^*(z + \mu u)]\right\| \\
    &= \left\|\E_{u \sim \U(\ball^{\dx})}[y^*(x + \mu u) - y^*(z + \mu u)]\right\| \\
    &\leq \E_{u \sim \U(\ball^{\dx})}[\norm{y^*(x + \mu u) - y^*(z + \mu u)}] \quad \text{(Jensen's inequality)} \\
    &\leq \E_{u \sim \U(\ball^{\dx})}[L_y \norm{x - z}] \quad \text{(Lipschitz property of } y^*\text{)} \\
    &= L_y \norm{x - z}.
\end{align*}

\textbf{Part (iii):} By definition of the spectral norm, for any matrix $A\in\R^{\dy\times \dx}$,
\begin{equation*}\label{eq:spectral_var}
    \|A\|_2 = \|A^\top\|_2 = \sup_{w\in\R^{\dy},\|w\|=1} \|A^\top w\|.
\end{equation*}
Therefore, it suffices to show that for every unit vector $w\in\R^{\dy}$ with $\|w\|=1$,
\begin{equation*}\label{eq:goal}
    \bigl\|\bigl[\jac y_\mu(x)-\jac y_\mu(z)\bigr]^\top w\bigr\|
    \;\le\;k_1\,\frac{L_y\sqrt{\dx}}{\mu}\,\|x-z\|.
\end{equation*}
Fix such a $w$, and define the scalar function $h_w:\R^{\dx}\to\R$ by
\begin{equation}\label{eq:hw_def}
    h_w(\cdot) \;=\; w^\top y^*(\cdot).
\end{equation}
Since $y^*$ is $L_y$-Lipschitz and $\|w\|=1$, we have for all $x,z\in\R^{\dx}$:
\begin{equation}\label{eq:hw_lip}
    |h_w(x)-h_w(z)| \;=\; |w^\top(y^*(x)-y^*(z))| \;\le\; \|w\|\,\|y^*(x)-y^*(z)\| \;\le\; L_y\|x-z\|,
\end{equation}
where the first inequality follows by Cauchy-Schwarz, and the second inequality follows as $\|w\|=1$ and $y^*(\cdot)$ is $L_y$-Lipschitz. Hence $h_w$ is $L_y$-Lipschitz. Then, applying uniform smoothing to $h_w$, by \cref{thm:smoothing-gradient} the function
\begin{equation*}
    \begin{split}
    h_{w,\mu}(x) &= \E_{u \sim \U(\ball^{\dx})}[h_w(x + \mu u)] = \E_{u \sim \U(\ball^{\dx})}[w^\top y^*(x + \mu u)] \\
    &= w^\top \E_{u \sim \U(\ball^{\dx})}[y^*(x + \mu u)] = w^\top y_\mu(x).
    \end{split}
\end{equation*} is continuously differentiable, with gradient 
\begin{equation}\label{eq:grad_hw_sphere}
    \grad h_{w,\mu}(x) = \frac{\dx}{\mu} \E_{v \sim \U(\sphere^{\dx})}[h_w(x + \mu v)\, v].
\end{equation}
Further, by \cite[Lem. 8]{yousefian2012randomizedsmoothingproperties}, $\grad h_{w,\mu}(x)$ is $k_1 L_y \sqrt{\dx}/\mu$-Lipschitz, where $k_1 > 0$ is a constant. 

We now show that $\grad h_{w,\mu}(x) = \jac y_\mu(x)^\top w$. Substituting $h_w(\cdot) = w^\top y^*(\cdot) = \sum_{i=1}^{\dy} w_i\, y^*_i(\cdot)$ into~\eqref{eq:grad_hw_sphere} and exchanging the sum with the expectation:
\begin{align}
    \grad h_{w,\mu}(x)
    &= \frac{\dx}{\mu} \E_{v \sim \U(\sphere^{\dx})}\!\left[\sum_{i=1}^{\dy} w_i\, y^*_i(x + \mu v)\, v\right] \notag \\
    &= \sum_{i=1}^{\dy} w_i \frac{\dx}{\mu} \E_{v \sim \U(\sphere^{\dx})}[y^*_i(x + \mu v)\, v] \notag \\
    &= \sum_{i=1}^{\dy} w_i\, \grad_x [y_\mu]_i(x) \;=\; \jac y_\mu(x)^\top w, \label{eq:grad_hw_jac}
\end{align}
where $\frac{\dx}{\mu} \E_{v \sim \U(\sphere^{\dx})}[y^*_i(x + \mu v)\, v] = \grad_x [y_\mu]_i(x)$ follows by \cref{lem:ysmooth-differentiable-lipschitz} and the last equality uses $[\jac y_\mu(x)^\top w]_j = \sum_{i=1}^{\dy} w_i\,\frac{\partial [y_\mu]_i}{\partial x_j}(x) = \sum_{i=1}^{\dy} w_i\,[\grad_x [y_\mu]_i(x)]_j$.

Therefore, by~\eqref{eq:grad_hw_jac}:
\begin{equation*}
    \bigl\|\bigl[\jac y_\mu(x)-\jac y_\mu(z)\bigr]^\top w\bigr\| = \|\grad h_{w,\mu}(x) - \grad h_{w,\mu}(z)\| \leq \frac{k_1 L_y \sqrt{\dx}}{\mu} \norm{x - z}.
\end{equation*}
Since this holds for every unit vector $w \in \R^{\dy}$,
\begin{equation*}
    \bigl\|\jac y_\mu(x)-\jac y_\mu(z)\bigr\|_2 = \sup_{\|w\|=1} \bigl\|\bigl[\jac y_\mu(x)-\jac y_\mu(z)\bigr]^\top w\bigr\| \le k_1\,\frac{L_y\sqrt{d_x}}{\mu}\,\|x-z\|.
\end{equation*}
This completes the proof.
\end{proof}

\begin{lemma}\label{lem:jacobiannorm}
    Assume that $y_\mu(x)$ is differentiable and $L_y$-Lipschitz. Then $\|\jac y_\mu(x)\|_2\leq L_y$.
\end{lemma}
\begin{proof}
As $y_\mu$ is differentiable by \cref{lem:ysmooth-differentiable-lipschitz}, the directional derivative \(D_v y_\mu(x) = \lim_{t \to 0} \frac{y_\mu(x + tv) - y_\mu(x)}{t}\) of $y_\mu$ in direction $v$ at $x$ exists and equals $\jac y_\mu(x) v$. It follows that
\begin{align*}
\|\jac y_\mu(x) v\| &= \left\|\lim_{t \to 0} \frac{y_\mu(x + tv) - y_\mu(x)}{t}\right\| \\
&= \lim_{t \to 0} \left\|\frac{y_\mu(x + tv) - y_\mu(x)}{t}\right\| \\
&\leq \lim_{t \to 0} \frac{L_y\|tv\|}{|t|} \quad \text{( } y_\mu \text{ is } L_y \text{-Lipschitz by \cref{lem:properties-smoothing}\,(ii))}\\
&= L_y\|v\|,
\end{align*}
where the second equality follows from the continuity of the norm $\norm{\cdot}$ and the existence of the limit $\lim_{t \to 0} \frac{y_\mu(x + tv) - y_\mu(x)}{t}=\jac y_\mu(x)v$.

Then, for arbitrary unit vectors $v$ with $\|v\|=1$, $\|\jac y_\mu(x) v\| \leq L_y$, and by definition of the spectral norm:
    \begin{equation*}
        \|\jac y_\mu(x)\|_2 = \sup_{v \in \R^{\dx},v\neq 0} \frac{\norm{\jac y_\mu(x) v}_2}{\norm{v}_2}= \sup_{v\in\R^{\dx},\norm{v}_2 = 1} \norm{\jac y_\mu(x) v}_2 \leq L_y. \qedhere
    \end{equation*}
\end{proof}
\subsection{Smoothness of surrogate function and estimator properties}\label{sec:surrogate-smoothness}
We now analyze the surrogate function $F_\mu(x) = f(x, y_\mu(x))$, with $y_\mu(x)$ defined in \eqref{eq:ysmooth}. We first establish that the distance between $F_\mu(x)$ and $F(x)$ is bounded by a constant depending on $\mu$ (\cref{pro:Fsmoothbound}), then establish that $F_\mu$ is differentiable and smooth under \cref{ass:lipschitz} (\cref{lem:descent-Fsmooth}). %

\begin{proposition}\label{pro:Fsmoothbound}
Consider the problem in \eqref{eq:setting} under \cref{ass:lipschitz}. Let $F(x) = f(x, y^*(x))$ and $F_\mu(x) = f(x, y_\mu(x))$, where $y^*: \R^{\dx} \to \Y$ be the unique response of the followers for given $x\in\mb{R}^{\dx}$ and $y_\mu$ is defined as in \eqref{eq:ysmooth}. Then,
\begin{equation*}
    |F_\mu(x) - F(x)| \leq L_f L_y \mu.
\end{equation*}
\end{proposition}

\begin{proof}
We have $F(x) = f(x, y^*(x))$ and $F_\mu(x) = f(x, y_\mu(x))$, so \( |F_\mu(x) - F(x)| = |f(x, y_\mu(x)) - f(x, y^*(x))| \). By \cref{ass:lipschitz}, $f(x, y)$ is $L_f$-Lipschitz in $y$ for fixed $x$, so \( |f(x, y_1) - f(x, y_2)| \leq L_f \norm{y_1 - y_2} \). From \cref{lem:properties-smoothing}(i), $\norm{y_\mu(x) - y^*(x)} \leq L_y \mu$, and thus 
\begin{equation*}
    |F_\mu(x) - F(x)| \leq L_f \norm{y_\mu(x) - y^*(x)} \leq L_f L_y \mu. \qedhere
\end{equation*}
\end{proof}

The following lemma establishes that $F_\mu$ is smooth under our assumptions. Unlike the smoothness of $\tilde{F}_\mu(x, y^*(x)) = \E_{u \sim \U(\ball^{\dx})}[f(x + \mu u, y^*(x + \mu u))]$, which follows directly from the smoothing properties in \cref{thm:smoothing-gradient}, the smoothness of $F_\mu$ does not follow immediately since $F_\mu$ does not smooth $f(x, y^*(x))$ directly. Instead, the proof exploits the smoothness of $f$ in both $x$ and $y$, together with the smoothness of $y_\mu$.

\begin{lemma}\label{lem:descent-Fsmooth}
    Under \Cref{ass:lipschitz}, $F_\mu(x) = f(x, y_\mu(x))$ is differentiable and $L_F$-smooth, with $L_F = L_g(1 + L_y)^2+\frac{k_1 L_f L_y \sqrt{\dx}}{\mu}$, and $k_1>0$ is a constant.
\end{lemma}
\begin{proof}[Proof of \cref{lem:descent-Fsmooth}]\label{app:proof_descent-Fsmooth}
As $f$ is differentiable by \cref{ass:lipschitz} and $y_\mu$ is differentiable by \cref{lem:ysmooth-differentiable-lipschitz}, the composition $F_\mu(x) = f(x, y_\mu(x))$ is differentiable. By the chain rule:
\begin{equation}\label{eq:Fsmooth-chainrule}
\grad F_\mu(x) = \grad_x f(x, y_\mu(x)) + \jac y_\mu(x)^\top \grad_y f(x, y_\mu(x)).
\end{equation}

To show that $F_\mu$ is $L_F$-smooth, we need to show that for any $x_1, x_2 \in \R^{\dx}$:
\[\norm{\grad F_\mu(x_1) - \grad F_\mu(x_2)} \leq L_F \norm{x_1 - x_2}.\]

Using \eqref{eq:Fsmooth-chainrule}, we have
\begin{align*}
&\norm{\grad F_\mu(x_1) - \grad F_\mu(x_2)} \\
&= \norm{\grad_x f(x_1, y_\mu(x_1)) + \jac y_\mu(x_1)^\top \grad_y f(x_1, y_\mu(x_1))  - \grad_x f(x_2, y_\mu(x_2)) - \jac y_\mu(x_2)^\top \grad_y f(x_2, y_\mu(x_2))} \\
&\leq \norm{\grad_x f(x_1, y_\mu(x_1)) - \grad_x f(x_2, y_\mu(x_2))} + \norm{\jac y_\mu(x_1)^\top \grad_y f(x_1, y_\mu(x_1)) - \jac y_\mu(x_2)^\top \grad_y f(x_2, y_\mu(x_2))}.
\end{align*}

\noindent For the first term, since $\grad f$ is $L_g$-Lipschitz by \cref{ass:lipschitz}:
\begin{align}
\norm{\grad_x f(x_1, y_\mu(x_1)) - \grad_x f(x_2, y_\mu(x_2))} &\leq L_g \norm{(x_1, y_\mu(x_1)) - (x_2, y_\mu(x_2))} \nonumber \\
&\leq L_g (\norm{x_1 - x_2} + \norm{y_\mu(x_1) - y_\mu(x_2)}) \nonumber \\
&\leq L_g (1 + L_y) \norm{x_1 - x_2}, \label{eq:normxy}
\end{align}
where we use that $y_\mu$ is $L_y$-Lipschitz by \cref{lem:properties-smoothing}(ii). 

\smallskip
\noindent For the second term, we add and subtract $\jac y_\mu(x_1)^\top \grad_y f(x_2, y_\mu(x_2))$:
\begin{align*}
&\norm{\jac y_\mu(x_1)^\top \grad_y f(x_1, y_\mu(x_1)) - \jac y_\mu(x_2)^\top \grad_y f(x_2, y_\mu(x_2))} \\
&= \norm{\jac y_\mu(x_1)^\top \grad_y f(x_1, y_\mu(x_1)) - \jac y_\mu(x_1)^\top \grad_y f(x_2, y_\mu(x_2)) \\
&\quad\quad+ \jac y_\mu(x_1)^\top \grad_y f(x_2, y_\mu(x_2)) - \jac y_\mu(x_2)^\top \grad_y f(x_2, y_\mu(x_2))} \\
&\leq \norm{\jac y_\mu(x_1)^\top [\grad_y f(x_1, y_\mu(x_1)) - \grad_y f(x_2, y_\mu(x_2))]} + \norm{[\jac y_\mu(x_1) - \jac y_\mu(x_2)]^\top \grad_y f(x_2, y_\mu(x_2))}.
\end{align*}

Bounding the first sub-term:
\begin{align*}
    \norm{\jac y_\mu(x_1)^\top [\grad_y f(x_1, y_\mu(x_1)) - \grad_y f(x_2, y_\mu(x_2))]}
    &\leq \norm{\jac y_\mu(x_1)}_2 \norm{\grad_y f(x_1, y_\mu(x_1)) - \grad_y f(x_2, y_\mu(x_2))} \\
    &\leq L_y \cdot L_g \norm{(x_1, y_\mu(x_1)) - (x_2, y_\mu(x_2))} \quad \text{(by \cref{lem:jacobiannorm})} \\
    &\leq L_y L_g (1 + L_y) \norm{x_1 - x_2} \quad \text{(by \eqref{eq:normxy})}.
\end{align*}

Bounding the second sub-term:
\begin{align*}
    \norm{[\jac y_\mu(x_1) - \jac y_\mu(x_2)]^\top \grad_y f(x_2, y_\mu(x_2))} 
    &\leq \norm{\jac y_\mu(x_1) - \jac y_\mu(x_2)}_2 \norm{\grad_y f(x_2, y_\mu(x_2))} \\
    &\leq \frac{k_1 L_y \sqrt{\dx}}{\mu} \norm{x_1 - x_2} \cdot L_f \quad \text{(by \cref{lem:properties-smoothing}(iii) and \cref{ass:lipschitz})}.
\end{align*}

\noindent Then, combining all bounds,
\begin{align*}
    \|\grad F_\mu(x_1) - \grad F_\mu(x_2)\| &\leq L_g(1+L_y)\norm{x_1-x_2} + L_y L_g (1+L_y) \norm{x_1-x_2} + \frac{k_1 L_f L_y \sqrt{\dx}}{\mu} \norm{x_1-x_2} \\
    &= \left[L_g(1 + L_y)^2 + \frac{k_1 L_f L_y \sqrt{\dx}}{\mu}\right] \norm{x_1-x_2}
\end{align*}

This establishes that $F_\mu$ has Lipschitz gradient with constant $L_F = L_g(1 + L_y)^2+\frac{k_1 L_f L_y \sqrt{\dx}}{\mu}$.
\end{proof}

\subsubsection{Properties of the Zeroth-Order Estimators}\label{sec:properties_estimator}
We now turn to the analysis of \cref{alg:partial_zero_order}. 

\begin{lemma}\label{lem:jacobianunbiasedestimator}
Consider \cref{alg:partial_zero_order} under \cref{ass:lipschitz}. Then $H_t = \frac{\dx}{2\mu}(y^*(x_t + \mu v_t) - y^*(x_t - \mu v_t)) \, v_t^\top$, where $v_t \sim \U(\sphere^{\dx})$, is an unbiased zeroth-order estimator of $\jac y_\mu(x_t)$, i.e.
\begin{equation*}
    \E[\,H_t \mid x_t\,] = \jac y_\mu(x_t),
\end{equation*}
and $\jac y_\mu(x_t)= \frac{\dx}{\mu} \E_{v_t \sim \U(\sphere^{\dx})}\Bigl[y^*(x_t + \mu v_t) \,v_t^\top\Bigr]$.
\end{lemma}

\begin{proof}
Given $x_t$, the randomness in $H_t$ arises solely from $v_t \sim \U(\sphere^{\dx})$, as $y^*(x_t+\mu v_t)$ is uniquely determined for any given $x_t+\mu v_t$:
\begin{align*}
\E[H_t \mid x_t] &= \E_{v_t \sim \U(\sphere^{\dx})}\left[\frac{\dx}{2\mu}(y^*(x_t + \mu v_t) - y^*(x_t - \mu v_t)) \, v_t^\top\right]\\
&= \frac{\dx}{2\mu} \E_{v_t \sim \U(\sphere^{\dx})}\Bigl[(y^*(x_t + \mu v_t) - y^*(x_t - \mu v_t)) \, v_t^\top\Bigr].
\end{align*}
Observe that for $v_t \sim \U(\sphere^{\dx})$, the distribution is symmetric, i.e. $v_t$ and $-v_t$ have the same distribution. Therefore, $\E_{v_t \sim \U(\sphere^{\dx})}\Bigr[y^*(x_t + \mu (-v_t)) \, (-v_t^\top)\Bigr] = \E_{v_t \sim \U(\sphere^{\dx})}\Bigl[y^*(x_t + \mu v_t) \, v_t^\top\Bigr]$ and it follows
\begin{align*}
\E[H_t | x_t] &= \frac{\dx}{2\mu} \E_{v_t \sim \U(\sphere^{\dx})}\Bigl[\bigl(y^*(x_t + \mu v_t) - y^*(x_t - \mu v_t)\bigr) \, v_t^\top\Bigr]\\
&= \frac{\dx}{2\mu} \Bigl(\E_{v_t \sim \U(\sphere^{\dx})}[y^*(x_t + \mu v_t) \, v_t^\top] - \E_{v_t \sim \U(\sphere^{\dx})}[y^*(x_t - \mu v_t) \, v_t^\top]\Bigr)\\
&= \frac{\dx}{2\mu} \Bigl(\E_{v_t \sim \U(\sphere^{\dx})}[y^*(x_t + \mu v_t) \, v_t^\top] + \E_{v_t \sim \U(\sphere^{\dx})}[y^*(x_t + \mu (-v_t)) \, (-v_t^\top)]\Bigr)\\
&= \frac{\dx}{2\mu} \Bigl(\E_{v_t \sim \U(\sphere^{\dx})}[y^*(x_t + \mu v_t) \, v_t^\top] + \E_{v_t \sim \U(\sphere^{\dx})}[y^*(x_t + \mu v_t) \, v_t^\top]\Bigr)\\
&= \frac{\dx}{\mu} \E_{v_t \sim \U(\sphere^{\dx})}\Bigl[y^*(x_t + \mu v_t) \, v_t^\top\Bigr].
\end{align*}

By \cref{lem:ysmooth-differentiable-lipschitz},
\(
    \jac y_\mu (x_t) = \frac{\dx}{\mu} \E_{v_t \sim \U(\sphere^{\dx})}\Bigl[y^*(x_t + \mu v_t) \, v_t^\top\Bigr] 
\), hence
\(
\E[\,H_t \mid x_t\,] = \jac y_\mu(x_t)
\) follows. \qedhere
\end{proof}

Despite the unbiasedness of $H_t$, the estimator $g_t$ of \algpzos is \emph{not} an unbiased estimator of the gradient of $F_\mu$. We recall by \eqref{eq:Fsmooth-chainrule} that $\grad F_\mu$ is given by
\begin{equation*}
    \grad F_\mu(x_t) = \grad_x f(x_t, y_\mu(x_t)) + \jac y_\mu(x_t)^\top \grad_y f(x_t, y_\mu(x_t)).
\end{equation*}
In contrast, the expected value of the estimator $g_t$ of \algpzos for given $x_t$ is given by 
\begin{align}\label{eq:gradientestimate-expectedvalue}
    \E[\,g_t \mid x_t\,] &= \E[\,\grad_x f(x_t, y^*(x_t)) + H_t^\top \grad_y f(x_t, y^*(x_t)) \mid x_t\,] \nonumber \\
    &= \grad_x f(x_t, y^*(x_t)) + \E[\,H_t \mid x_t\,]^\top \grad_y f(x_t, y^*(x_t)) \nonumber \\
    &= \grad_x f(x_t, y^*(x_t)) + \jac y_\mu(x_t)^\top \grad_y f(x_t, y^*(x_t)),
\end{align}
where we use $\E[H_t \mid x_t] = \jac y_\mu(x_t)$ by \cref{lem:jacobianunbiasedestimator}. Therefore, $g_t$ is a \emph{biased} estimator of $\grad F_\mu(x_t)$, with
\begin{align}\label{eq:diffggradf}
    & \E[g_t \mid x_t] - \grad F_\mu(x_t) \notag\\
    &= \grad_x f(x_t, y^*(x_t)) - \grad_x f(x_t, y_\mu(x_t)) + \jac y_\mu(x_t)^\top [\grad_y f(x_t, y^*(x_t)) - \grad_y f(x_t, y_\mu(x_t))].
\end{align}

Below, we bound the distance of $\E[g_t | x_t]$ to $\grad F_\mu(x_t)$ by a constant that depends on $\mu$.

\begin{lemma}\label{lem:difference-estimate-gradfsmooth}
Let $\{g_t\}_{t=0}^{T-1}$ and $\{x_t\}_{t=0}^{T-1}$ be generated by \cref{alg:partial_zero_order}. Under \cref{ass:lipschitz},
    \[ \norm{\E[g_t | x_t] - \grad F_\mu(x_t)} \leq L_g(1 + L_y)L_y\mu. \]
\end{lemma}
\begin{proof}
By \eqref{eq:diffggradf}, 
\begin{align*}
    &\norm{\E[g_t \mid x_t] - \grad F_\mu(x_t)} \\
    &= \norm{\grad_x f(x_t, y^*(x_t)) - \grad_x f(x_t, y_\mu(x_t)) + \jac y_\mu(x_t)^\top [\grad_y f(x_t, y^*(x_t)) - \grad_y f(x_t, y_\mu(x_t))]} \\
    &\leq \norm{\grad_x f(x_t, y^*(x_t)) - \grad_x f(x_t, y_\mu(x_t))} + \norm{\jac y_\mu(x_t)^\top [\grad_y f(x_t, y^*(x_t)) - \grad_y f(x_t, y_\mu(x_t))]}.
\end{align*}

For the first term, since $\grad f$ is $L_g$-Lipschitz and $\norm{y^*(x_t) - y_\mu(x_t)} \leq L_y \mu$ by \cref{lem:properties-smoothing}(i):
\begin{align*}
    \norm{\grad_x f(x_t, y^*(x_t)) - \grad_x f(x_t, y_\mu(x_t))} &\leq L_g \norm{(x_t, y^*(x_t)) - (x_t, y_\mu(x_t))} \\
    &= L_g \norm{y^*(x_t) - y_\mu(x_t)} \\
    &\leq L_g L_y \mu.
\end{align*}

For the second term, using $\norm{Ax}_2\leq \norm{A}_2\norm{x}_2$ and $\norm{\jac y_\mu(x_t)}_2 \leq L_y$ by \cref{lem:jacobiannorm}:
\begin{align*}
\norm{\jac y_\mu(x_t)^\top [\grad_y f(x_t, y^*(x_t)) - \grad_y f(x_t, y_\mu(x_t))]} &\leq \norm{\jac y_\mu(x_t)} \norm{\grad_y f(x_t, y^*(x_t)) - \grad_y f(x_t, y_\mu(x_t))} \\
&\leq L_y \cdot L_g \norm{y^*(x_t) - y_\mu(x_t)} \\
&\leq L_y \cdot L_g L_y \mu = L_g L_y^2 \mu.
\end{align*}

Combining both terms,
\begin{equation*}
\norm{\E[g_t \mid x_t] - \grad F_\mu(x_t)} \leq L_g L_y \mu + L_g L_y^2 \mu = L_g (1 + L_y) L_y\mu. \qedhere
\end{equation*}
\end{proof}

We now establish a bound on the second moment of the estimator $g_t$.
\begin{lemma}\label{lem:estimatorsquared}
Let $\{g_t\}_{t=0}^{T-1}$ and $\{x_t\}_{t=0}^{T-1}$ be generated by \cref{alg:partial_zero_order}. Under \cref{ass:lipschitz}, 
\begin{equation*}
    \E[\norm{g_t}^2\mid x_t] \leq k_2 \dx L_f^2 L_y^2 + L_f^2(1+2L_y).
\end{equation*}
\end{lemma}
\begin{proof}
Recall that the estimator of \cref{alg:partial_zero_order} is
\[
    g_t = \grad_x f(x_t, y^*(x_t)) + H_t^\top \grad_y f(x_t, y^*(x_t)),
\]
where $H_t = \frac{d_x}{2\mu}(y^*(x_t+\mu v_t)-y^*(x_t-\mu v_t))\,v_t^\top$ with $v_t \sim \mathcal{U}(\mathbb{S}^{\dx})$. Conditioned on $x_t$, the terms $\grad_x f(x_t, y^*(x_t))$ and $\grad_y f(x_t, y^*(x_t))$ are deterministic; only $H_t$ (through $v_t$) is random. Expanding the squared norm:
\begin{align}
    \E[\|g_t\|^2 \mid x_t]
    &= \E\bigl[\|\grad_x f(x_t, y^*(x_t)) + H_t^\top \grad_y f(x_t, y^*(x_t))\|^2 \;\big|\; x_t\bigr] \notag \\
    &= \|\grad_x f(x_t, y^*(x_t))\|^2 \label{eq:term1} \\
    &\quad + 2\bigl\langle \grad_x f(x_t, y^*(x_t)),\; \E[H_t^\top \grad_y f(x_t, y^*(x_t)) \mid x_t]\bigr\rangle \label{eq:term2} \\
    &\quad + \E\bigl[\|H_t^\top \grad_y f(x_t, y^*(x_t))\|^2 \;\big|\; x_t\bigr]. \label{eq:term3}
\end{align}
We bound each term.

\medskip
\noindent\textbf{Term~\eqref{eq:term1}} By \Cref{ass:lipschitz}, $\|\grad_x f(x_t, y^*(x_t))\| \leq L_f$, so
\begin{equation}\label{eq:term1_bound}
    \|\grad_x f(x_t, y^*(x_t))\|^2 \leq L_f^2.
\end{equation}

\smallskip
\noindent\textbf{Term~\eqref{eq:term2}}
By the Cauchy-Schwarz inequality $\langle a, b \rangle \leq |\langle a, b \rangle| \leq \|a\| \cdot \|b\|$. Further, $\|\grad_x f(x_t, y^*(x_t))\|, \|\grad_y f(x_t, y^*(x_t))\| \leq L_f$ (by \cref{ass:lipschitz}), $\|\E[H_t \mid x_t]\|_2 = \|\jac y_\mu(x_t)\|_2$ (by \cref{lem:jacobianunbiasedestimator}), and $\|\jac y_\mu(x_t)\|_2 \leq L_y$ (by \cref{lem:jacobiannorm}). We have:
\begin{align}
    2\bigl\langle \grad_x f(x_t, y^*(x_t)),\; \E[H_t^\top \grad_y f(x_t, y^*(x_t)) \mid x_t]\bigr\rangle &\leq 2\|\grad_x f(x_t, y^*(x_t))\|\,\|\E[H_t^\top \grad_y f(x_t, y^*(x_t)) \mid x_t]\| \nonumber\\
    &\leq 2L_f \cdot \|\E[H_t \mid x_t]^\top \grad_y f(x_t, y^*(x_t)) \| \nonumber\\
    &\leq 2L_f \cdot \|\E[H_t \mid x_t]\|_2 \cdot \|\grad_y f(x_t, y^*(x_t))\| \nonumber\\
    &\leq 2L_f \cdot L_y L_f = 2L_f^2 L_y. \label{eq:term2_bound}
\end{align}

\smallskip
\noindent\textbf{Term~\eqref{eq:term3}}
Fix $\grad_y f(x_t, y^*(x_t))$ and define the scalar function $h_x(z) = \grad_y f(x_t, y^*(x_t))^\top y^*(z)$. Since $y^*$ is $L_y$-Lipschitz and $\|\grad_y f(x_t, y^*(x_t))\| \leq L_f$ by \cref{ass:lipschitz}:
\begin{align*}
    |h_x(z) - h_x(z')| &= |\grad_y f(x_t, y^*(x_t))^\top(y^*(z) - y^*(z'))| \\
        &\leq \|\grad_y f(x_t, y^*(x_t))\|\,\|y^*(z) - y^*(z')\| \\
        &\leq L_f L_y\|z - z'\|,
\end{align*}
so $h_x$ is $(L_fL_y)$-Lipschitz. Using $H_t = \frac{d_x}{2\mu}(y^*(x_t+\mu v_t) - y^*(x_t-\mu v_t))\,v_t^\top$ and factoring out the scalar:
\begin{align*}
    H_t^\top \grad_y f(x_t, y^*(x_t))
    &= \frac{\dx}{2\mu}\bigl[\grad_y f(x_t, y^*(x_t))^\top\bigl(y^*(x_t+\mu v_t) - y^*(x_t-\mu v_t)\bigr)\bigr]v_t \\
    &= \frac{\dx}{2\mu}\bigl[h_x(x_t+\mu v_t) - h_x(x_t-\mu v_t)\bigr]v_t,
\end{align*}
which is the standard two-point gradient estimator 
\begin{equation*}
    \frac{\dx}{2\mu}\bigl[h_x(x_t+\mu v_t) - h_x(x_t-\mu v_t)\bigr]v_t
\end{equation*}
for the scalar Lipschitz function $h_x$ at $x_t$ with smoothing parameter $\mu$ and random direction $v_t \sim \U(\sphere^{\dx})$. By \cite[Lem. 10]{shamir2017boundestimator} and \cite[Appendix D.1]{lin2022goldstein}, for such a two-point estimator of an $L$-Lipschitz function $h:\mb{R}^{\dx} \to \mb{R}$, we have
\begin{equation*}
    \E\!\left[\left\|\frac{\dx}{2\mu}(h(x+\mu v) - h(x-\mu v))v\right\|^2\right] \leq k_2\dx L^2,
\end{equation*}
where $k_2$ is a constant. Substituting $L = L_f L_y$ and $h=h_x$ gives
\begin{equation}\label{eq:term3_bound}
    \E\bigl[\|H_t^\top \grad_y f(x_t, y^*(x_t))\|^2 \;\big|\; x_t\bigr] \leq k_2\dx L_f^2 L_y^2.
\end{equation}

\medskip
Substituting~\eqref{eq:term1_bound}, \eqref{eq:term2_bound}, and~\eqref{eq:term3_bound} into~\eqref{eq:term1}--\eqref{eq:term3}:
\[
    \E[\|g_t\|^2 \mid x_t] \leq k_2\dx L_f^2 L_y^2 + L_f^2 + 2L_f^2 L_y = k_2\dx L_f^2 L_y^2 + L_f^2(1 + 2L_y). \qedhere
\]
\end{proof}

\medskip
Last but not least, we compare the second moment bounds for $g_t$ with those for the standard two-point zeroth-order estimator $\tilde{g}_t$ used in \cref{alg:zero_order}.
\begin{lemma}[Restatement of Lemma \ref{lem:variancecomparison}]\label{lem:variancecomparisonrestatement}
Under \cref{ass:lipschitz} and $v_t$ sampled uniformly from the unit sphere $\sphere^{\dx}$, let 
\begin{align*}
  H_t &= \frac{\dx}{2\mu}(y^*(x_t + \mu v_t) - y^*(x_t - \mu v_t))\,v_t^\top \\
  g_t &= \grad_x f(x_t, y^*(x_t)) + H_t^\top \grad_y f(x_t, y^*(x_t))
\end{align*}
be the estimator used in \cref{alg:partial_zero_order}, and let
\[
  \tilde{g}_t = \frac{\dx}{2\mu}\bigl[f(x_t + \mu v_t, y^*(x_t + \mu v_t)) - f(x_t - \mu v_t, y^*(x_t - \mu v_t))\bigr]v_t
\]
be the standard two-point zeroth-order gradient estimator of \cref{alg:zero_order}. Then the conditional second moment bounds satisfy:
\begin{align}
    \E[\norm{g_t}^2 \mid x_t] &\leq k_2\dx L_f^2 L_y^2 + L_f^2(1+2L_y), \label{eq:secondmoment_gt_stmt} \\
    \E[\norm{\tilde{g}_t}^2 \mid x_t] &\leq k_2\dx L_f^2 L_y^2 + 
    k_2\dx L_f^2(1+2L_y), \label{eq:secondmoment_at_stmt}
\end{align}
where $k_2 \geq 1$ is the universal constant from \cite[Lem.~10]{shamir2017boundestimator}. We obtain the same bound for the conditional variances:
\begin{align}
    \E[\norm{g_t - \E[g_t \mid x_t]}^2 \mid x_t] &\leq k_2\dx L_f^2 L_y^2 + L_f^2(1 + 2L_y), \label{eq:var_gp} \\
    \E[\norm{\tilde{g}_t - \E[\tilde{g}_t \mid x_t]}^2 \mid x_t] &\leq  k_2\dx L_f^2 L_y^2 + 
    k_2\dx L_f^2(1+2L_y), \label{eq:var_gzo}
\end{align}
Since $k_2 \geq 1$ and $\dx \geq 1$, the bounds for $g_t$ are at least as tight as those for $\tilde{g}_t$. 
\end{lemma}
\begin{proof}
We first recall the relevant second moment bounds, then show that $k_2 \geq 1$, and finally derive the conditional variance comparison.
 
By \cref{lem:estimatorsquared},
\begin{equation}\label{eq:secondmoment_gp}
    \E[\norm{g_t}^2 \mid x_t] \leq k_2\dx L_f^2 L_y^2 + L_f^2(1+2L_y).
\end{equation}
For the estimator $\tilde{g}_t$, we apply a general result on two-point zeroth-order estimators. Specifically, \citet[Lem.~10]{shamir2017boundestimator} establishes that, for any $L$-Lipschitz function $h:\R^d \to \R$ and given $x\in\R^d$, the two-point estimator
\[
  \tilde{g} = \frac{d}{2\delta}\bigl[h(x + \delta v) - h(x - \delta v)\bigr]v, \qquad v \sim \U(\sphere^{d})
\]
satisfies $\E[\norm{\tilde{g}}^2 \mid x] \leq k_2\, d\, L^2$, where $k_2 > 0$ is a universal constant.

$\tilde{g}_t$ is precisely this two-point estimator applied to the composite function $F(x) = f(x,y^*(x))$, using $\delta=\mu$. By the triangle inequality,
\begin{equation*}
  \begin{split}
    |F(x) - F(z)| &= |f(x,y^*(x)) - f(z,y^*(z))| \\
    &\leq |f(x,y^*(x)) - f(z,y^*(x))| + |f(z,y^*(x)) - f(z,y^*(z))| \\
    &\leq L_f\norm{x-z} + L_f L_y \norm{x-z},
  \end{split}
\end{equation*}
so $F$ is $L_f(1+L_y)$-Lipschitz. Applying the above mentioned result from \cite[Lem.~10]{shamir2017boundestimator} with $L = L_f(1+L_y)$:
\begin{equation}\label{eq:secondmoment_gzo}
    \E[\norm{\tilde{g}_t}^2 \mid x_t] \leq k_2\dx L_f^2(1+L_y)^2 =  k_2\dx L_f^2 L_y^2 + 
    k_2\dx L_f^2(1+2L_y).
\end{equation}
 
It remains to show that $k_2 \geq 1$. To do so, we observe that the bound $\E[\norm{\tilde{g}_t}^2\mid x_t] \leq k_2 \dx L^2$ from \cite[Lem.~10]{shamir2017boundestimator} must hold for every $L$-Lipschitz function, and present a specific function that attains $\E[\norm{\tilde{g}_t}^2\mid x_t] = \dx L^2$, which forces $k_2 \geq 1$. For any $L > 0$, consider the linear function $h(x) = L e_1^\top x$, with $e_1=(1,0,\ldots,0)^\top$, which is exactly $L$-Lipschitz. The two-point estimator at any point $x$ evaluates to
\[
    \tilde{g} = \frac{\dx}{2\mu}\bigl[h(x + \mu v) - h(x - \mu v)\bigr]v = \frac{\dx}{2\mu}\bigl[2L\mu v_1\bigr]v = \dx L v_1 v,
\]
where $v = (v_1, \ldots, v_d)^\top \sim \U(\sphere^{\dx})$. Since $\norm{v} = 1$, $\E[\norm{\tilde{g}}^2\mid x] = \dx^2 L^2 \E[v_1^2]$. By rotational symmetry of the uniform distribution on $\sphere^{\dx}$, all components $v_i$ are identically distributed, so $\E[v_1^2] = \E[v_2^2] = \ldots = \E[v_d^2]$. Since $\sum_{i=1}^{\dx} v_i^2 = 1$ on the unit sphere, taking expectations gives $\dx \E[v_1^2] = 1$, hence $\E[v_1^2] = 1/\dx$. Substituting:
\[
    \E[\norm{\tilde{g}}^2\mid x] = \dx^2 L^2 \cdot \frac{1}{\dx} = \dx L^2.
\]
The bound $\dx L^2 \leq k_2 \dx L^2$ must hold, which implies $k_2 \geq 1$.
 
We conclude by noting that the variance of both estimators is upper bounded by their respective second moment. To do so, we establish the identity $\E[\norm{g_t - \E[g_t \mid x_t]}^2 \mid x_t] = \E[\norm{g_t}^2 \mid x_t] - \norm{\E[g_t \mid x_t]}^2$. Writing $g_t = \E[g_t \mid x_t] + (g_t - \E[g_t \mid x_t])$ and expanding the squared norm:
\begin{align*}
    \E[\norm{g_t}^2 \mid x_t] &= \norm{\E[g_t \mid x_t]}^2 + 2\bigl\langle \E[g_t \mid x_t],\, \E[g_t - \E[g_t \mid x_t] \mid x_t]\bigr\rangle + \E[\norm{g_t - \E[g_t \mid x_t]}^2 \mid x_t].
\end{align*}
The cross-term vanishes: $\E[g_t \mid x_t]$ is a fixed vector conditioned on $x_t$, so $\E[g_t - \E[g_t \mid x_t] \mid x_t] = \E[g_t \mid x_t] - \E[g_t \mid x_t] = 0$. Rearranging:
\[
    \E[\norm{g_t - \E[g_t \mid x_t]}^2 \mid x_t] = \E[\norm{g_t}^2 \mid x_t] - \norm{\E[g_t \mid x_t]}^2 \leq \E[\norm{g_t}^2 \mid x_t],
\]
since $\norm{\E[g_t \mid x_t]}^2 \geq 0$. An equivalent argument applies to $\tilde{g}_t$.
\end{proof}
\subsection{Convergence to Goldstein Stationary Points}\label{sec:convergence_goldstein}
We turn to the final step to obtain our main result. %
To do so, we first establish that the Jacobian of the smoothed mapping $y_\mu$ lies within the $\mu$-Goldstein subdifferential of $y^*$ (\cref{lem:jacobiangoldstein}). It then follows that the estimator $g_t$ of \algpzos lies in the $\delta$-partial Goldstein subdifferential of $F$, for $\delta\ge\mu$. For the Goldstein subdifferential, we show that $\grad F_\mu$ is close to the Goldstein subdifferential of $F$, with a distance that depends on the Lipschitz constants of $f$, $\grad f$, and $y^*$, as well as the smoothing parameter $\mu$. We then prove our main result.

\begin{lemma}[Smoothed Jacobian and Generalized Jacobian]\label{lem:jacobiangoldstein}
Let $y^*(x): \R^{\dx} \to \Y$ be $L_y$-Lipschitz continuous and single-valued for given $x$, and $\Y$ be a closed and convex set. Define $y_\mu(x) = \E_{u \sim \U(\ball^{\dx})}[y^*(x + \mu u)]$. Then for all $x \in \R^{\dx}$:
\[
\jac y_\mu(x) \in \goldsteinymu y^*(x)=\co\left\{\bigcup_{z \in\ball(x,\mu)} \partial y^*(z) \right\},
\]
where $\partial y^*$ denotes the Clarke generalized Jacobian of $y^*$.
\end{lemma}
\begin{proof}
The proof proceeds in two steps. First, we show that the Jacobian of the smoothed mapping $y_\mu$ can be expressed as the expectation of the Jacobian of the original mapping $y^*$ over the uniform distribution on the unit ball $\ball^{\dx}$ at points where $y^*$ is differentiable. Second, we show that this Jacobian lies within the convex hull of Clarke generalized Jacobians $\partial y^*(z)$ of points $z\in\ball(x,\mu)$.

Before we begin, we recall that, since $y^*$ is $L_y$-Lipschitz continuous, by Rademacher's Theorem it is differentiable almost everywhere. Therefore, the set
\[
\Omega \coloneqq \{ u \in \ball^{\dx} : y^* \text{ is not differentiable at } x + \mu u \},
\]
of points at which $y^*$ is non-differentiable has Lebesgue measure zero. Let $p(u)$ denote the probability density function of the uniform distribution on $\ball^{\dx}$. Since $\Omega$ has measure zero, $\int_{\ball^{\dx}} p(u) \, du=\int_{\ball^{\dx}\setminus\Omega} p(u) \, du = 1$.

\noindent\textbf{Step 1: Expressing} \bm{$\jac y_\mu(x)$} \textbf{as an expectation of} \bm{$\jac y^*(z)$}. We show that, for every $x \in \R^{\dx}$, \[\jac y_\mu(x) = \int_{\ball^{\dx}\setminus\Omega} \jac y^*(x + \mu u) \, p(u) \, du.\]

We recall by \cref{lem:ysmooth-differentiable-lipschitz} that $y_\mu$ is continuously differentiable. Let $e_j$ denote the $j$-th standard basis vector in $\R^{\dx}$. For each $i \in \{1, \ldots, \dy\}$ and $j \in \{1, \ldots, \dx\}$, the partial derivative of the $i$-th component $y_{\mu,i}$ of $y_\mu$ with respect to $x_j$ is defined by the limit of the difference quotient:
\begin{equation*}\label{eq:ymupartialderivative}
    \frac{\partial y_{\mu,i}}{\partial x_j}(x) = \lim_{t \to 0} \frac{y_{\mu,i}(x + t e_j) - y_{\mu,i}(x)}{t}.
\end{equation*}
Substituting the definition of $y_\mu(x) = \int_{\ball^{\dx}} y^*(x + \mu u) p(u) \, du$ from \eqref{eq:ymuintegral}, we have:
\begin{equation}\label{eq:ymupartialderivativeintegral}
    \frac{\partial y_{\mu,i}}{\partial x_j}(x) = \lim_{t \to 0} \int_{\ball^{\dx}} \frac{y^*_i(x + t e_j + \mu u) - y^*_i(x + \mu u)}{t} p(u) \, du.
\end{equation}
We proceed to show that Lebesgue's dominated convergence theorem applies, allowing us to interchange the limit and the integral. To do so, we verify that the difference quotient (i) converges pointwise almost everywhere and (ii) is dominated by an integrable function. 
\begin{enumerate}[label=(\roman*)]
    \item \emph{Pointwise limit almost everywhere:} For all $u \in \ball^{\dx} \setminus \Omega$, the point $x + \mu u$ lies in the set where $y^*$ is differentiable. For such $u$, the difference quotient converges pointwise to the partial derivative:
    \[
    \lim_{t \to 0} \frac{y^*_i(x + t e_j + \mu u) - y^*_i(x + \mu u)}{t} = \frac{\partial y^*_i}{\partial x_j}(x + \mu u).
    \]
    Since $\Omega$ has measure zero, this convergence holds for almost every $u \in \ball^{\dx}$.
    \item \emph{Dominated by integrable function:} By $L_y$-Lipschitz continuity of $y^*$, for all $t \neq 0$, the difference quotient is bounded:
    \[
    \left\| \frac{y^*_i(x + t e_j + \mu u) - y^*_i(x + \mu u)}{t} \right\| \leq \frac{L_y \| (x + t e_j + \mu u) - (x + \mu u) \|}{|t|} = \frac{L_y \|t e_j\|}{|t|} = L_y.
    \]
\end{enumerate}
Since the constant function $L_y$ is integrable over $\ball^{\dx}$ (hence $\ball^{\dx}\setminus\Omega$), the conditions for the dominated convergence theorem are satisfied. Thus, we can interchange the limit and the integral in \eqref{eq:ymupartialderivativeintegral}:
\begin{equation*}
\begin{split}
    \frac{\partial y_{\mu,i}}{\partial x_j}(x) &= \lim_{t \to 0} \int_{\ball^{\dx}} \frac{y^*_i(x + t e_j + \mu u) - y^*_i(x + \mu u)}{t} p(u) \, du \\
    &= \lim_{t \to 0} \int_{\ball^{\dx} \setminus \Omega} \frac{y^*_i(x + t e_j + \mu u) - y^*_i(x + \mu u)}{t} p(u) \, du \\
    &= \int_{\ball^{\dx} \setminus \Omega} \frac{\partial y^*_i}{\partial x_j}(x + \mu u) \, p(u) \, du,
\end{split}
\end{equation*}
where the second equality holds because $\Omega$ has Lebesgue measure zero and the difference quotient is bounded by $L_y$, so $\int_\Omega (\cdot) \, p(u) \, du = 0$. The final equality applies the dominated convergence theorem on $\ball^{\dx} \setminus \Omega$: the pointwise limit exists for all $u \in \ball^{\dx} \setminus \Omega$, and the constant $L_y$ is integrable over this set since $\int_{\ball^{\dx} \setminus \Omega} p(u) \, du = 1$.

Collecting all partial derivatives into the Jacobian matrix yields:
\begin{equation}\label{eq:jacymuintegral}
    \jac y_\mu(x) =  \int_{\ball^{\dx} \setminus \Omega} \jac y^*(x + \mu u) \, p(u) \, du = \E_{u \sim \U(\ball^{\dx})} [\jac y^*(x + \mu u)],
\end{equation}
where the expectation is taken over the points at which $y^*$ is differentiable.

\noindent\textbf{Step 2:} \bm{$\jac y_\mu(x)$} \textbf{lies in the convex hull of Clarke generalized Jacobians}. Define the set
\begin{equation*}
    K_\mu \coloneqq \co\left\{ \bigcup_{z \in \ball(x,\mu)} \partial y^*(z) \right\} \subseteq \R^{\dy \times \dx}.
\end{equation*}
We show that $\jac y_\mu(x) \in K_\mu$ via a contradiction argument based on the strict hyperplane separation theorem. 

First, we establish that $K_\mu$ is closed, compact, and convex. The Clarke generalized Jacobian $\partial y^*(z)$ is a nonempty, convex, compact subset of $\R^{\dy \times \dx}$ for each $z$, and the set-valued mapping $z \mapsto \partial y^*(z)$ is upper semicontinuous \cite[Prop.~2.6.2]{clarke1990generalizedgradients}. Since $\ball(x, \mu)$ is compact and the image of a compact set under an upper semicontinuous map with compact values is compact \cite[Prop.~3]{Aubin1984usccompact}, the set $\bigcup_{z \in \ball(x, \mu)} \partial y^*(z)$ is compact. In the finite-dimensional Euclidean space of matrices $\R^{\dy \times \dx}$, the convex hull of a compact set is compact and closed, hence $K_\mu$ is closed, compact, and convex.

Suppose for contradiction that $\jac y_\mu(x) \notin K_\mu$. We apply the strict hyperplane separation theorem in the finite-dimensional Hilbert space $(\R^{\dy \times \dx}, \langle \cdot, \cdot \rangle_F)$, where the Frobenius inner product $\langle M, A \rangle_F = \Tr(M^\top A)$ corresponds to the Euclidean inner product $\langle \mathrm{vec}(M), \mathrm{vec}(A) \rangle$ under the isomorphism between $\R^{\dy \times \dx}$ and $\R^{\dy \cdot \dx}$ induced by the vectorization operator $\mathrm{vec}: \R^{\dy \times \dx} \to \R^{\dy \cdot \dx}$. Since $K_\mu$ is closed, compact, and convex, and $\{\jac y_\mu(x)\}$ (a singleton) is a closed set disjoint from $K_\mu$ by assumption, the strict hyperplane separation theorem guarantees that there exists a matrix $M \in \R^{\dy \times \dx}$ and a scalar $\alpha \in \R$ such that:
\begin{equation} \label{eq:hyperplaneseparation}
\langle M, \jac y_\mu(x) \rangle_F > \alpha \quad \text{and} \quad \langle M, A \rangle_F < \alpha \quad \forall A \in K_\mu.
\end{equation}

On the set $\ball^{\dx} \setminus \Omega$, the classical Jacobian $\jac y^*(x + \mu u)$ exists and, by definition of the Clarke generalized Jacobian, satisfies $\jac y^*(x + \mu u) \in \partial y^*(x + \mu u)$. Since $x + \mu u \in \ball(x, \mu)$, we have $\jac y^*(x + \mu u) \in K_\mu$ for all $u \in \ball^{\dx} \setminus \Omega$. By the second inequality in \eqref{eq:hyperplaneseparation},
\[
\langle M, \jac y^*(x + \mu u) \rangle_F < \alpha \quad \text{for all } u \in \ball^{\dx} \setminus \Omega.
\]
Using \eqref{eq:jacymuintegral}, we express the left-hand side of the first inequality in \eqref{eq:hyperplaneseparation} as
\begin{align*}
    \langle M, \jac y_\mu(x) \rangle_F = \left\langle M, \int_{\ball^{\dx} \setminus \Omega} \jac y^*(x + \mu u) \, p(u) \, du \right\rangle_F &= \int_{\ball^{\dx} \setminus \Omega} \langle M, \jac y^*(x + \mu u) \rangle_F \, p(u) \, du \\
    &\leq \int_{\ball^{\dx} \setminus \Omega} \alpha \, p(u) \, du = \alpha,
\end{align*}
where the final equality uses $\int_{\ball^{\dx} \setminus \Omega} p(u) \, du = 1$ as $\Omega$ has measure zero. This gives $\langle M, \jac y_\mu(x) \rangle \leq \alpha$, which contradicts $\langle M, \jac y_\mu(x) \rangle > \alpha$ from \eqref{eq:hyperplaneseparation}. Therefore, our initial assumption that $\jac y_\mu(x) \notin K_\mu$ must be false, and we conclude that $\jac y_\mu(x) \in K_\mu = \co\left\{ \bigcup_{z \in \ball(x,\mu)} \partial y^*(z) \right\}$. \qedhere
\end{proof}

\subsubsection{Distance to Goldstein Subdifferential}
The following lemma bounds the distance between the gradient of the surrogate function $F_\mu$ to elements of the Goldstein subdifferential $\goldsteinfullmu F(x)$.
\begin{lemma}[Distance to Goldstein subdifferential]\label{lem:distancefullgoldstein}
Let $F_\mu(x) = f(x,y_\mu(x))$ and $y_\mu(x)$ defined as in \cref{eq:ysmooth}. Under \cref{ass:lipschitz}, for any $x \in \R^{d_x}$ and $\mu >0$, there exists $\bar{q} \in \goldsteinfullmu F(x)$ such that \[\norm{\grad F_\mu(x) - \bar{q}} \leq (1+L_y)L_g(1+2L_y)\mu.\]
\end{lemma}
\begin{proof}
By \cref{lem:jacobiangoldstein}, $\jac y_\mu(x) \in \co\left\{\bigcup_{z \in \ball(x, \mu)} \partial y^*(z)\right\}$. By Carathéodory's theorem, there exist $k \leq d_y d_x + 1$ matrices $M_1, \ldots, M_k$, where each $M_i \in \partial y^*(z_i)$ for some $z_i \in \ball(x, \mu)$ (where $z_i$ need not be distinct), and weights $\lambda_i \geq 0$ with $\sum_{i=1}^k \lambda_i = 1$ such that
\begin{equation*}
\jac y_\mu(x) = \sum_{i=1}^k \lambda_i M_i.
\end{equation*}
Define, for each $i \in \{1, \ldots, k\}$:
\begin{equation*}
q_i = \grad_x f(z_i, y^*(z_i)) + M_i^\top \grad_y f(z_i, y^*(z_i)).
\end{equation*}
By \eqref{eq:clarke_gradient_F}, $q_i \in \partial F(z_i)$ for each $i$.
Define
\begin{equation*}
\bar{q} = \sum_{i=1}^k \lambda_i q_i.
\end{equation*}
Since each $q_i \in \partial F(z_i)$ with $z_i \in \ball(x, \mu)$, and $\bar{q}$ is a convex combination, we have by \cref{eq:goldstein_full_F} of the Goldstein subdifferential
\begin{equation*}
\bar{q} \in \co\left\{ \bigcup_{z \in \ball(x, \mu)} \partial F(z) \right\} = \goldsteinfullmu F(x).
\end{equation*}

We now bound $\|\grad F_\mu(x) - \bar{q}\|$. Using the chain rule \eqref{eq:Fsmooth-chainrule} for $F_\mu$:
\begin{align*}
\grad F_\mu(x) - \bar{q} &= \grad_x f(x, y_\mu(x)) + \jac y_\mu(x)^\top \grad_y f(x, y_\mu(x)) - \sum_{i=1}^k \lambda_i q_i \\
&= \grad_x f(x, y_\mu(x)) + \left( \sum_{i=1}^k \lambda_i M_i \right)^\top \grad_y f(x, y_\mu(x)) - \sum_{i=1}^k \lambda_i q_i \\
&= \sum_{i=1}^k \lambda_i \left[ \grad_x f(x, y_\mu(x)) + M_i^\top \grad_y f(x, y_\mu(x)) - q_i \right].
\end{align*}
For each $i \in \{1, \ldots, k\}$, we bound the term in brackets:
\begin{align*}
&\|\grad_x f(x, y_\mu(x)) + M_i^\top \grad_y f(x, y_\mu(x)) - q_i\| \\
&= \|\grad_x f(x, y_\mu(x)) - \grad_x f(z_i, y^*(z_i)) + M_i^\top [\grad_y f(x, y_\mu(x)) - \grad_y f(z_i, y^*(z_i))]\| \\
&\leq \|\grad_x f(x, y_\mu(x)) - \grad_x f(z_i, y^*(z_i))\| + \|M_i\|_2 \|\grad_y f(x, y_\mu(x)) - \grad_y f(z_i, y^*(z_i))\|.
\end{align*}
For the term $\|\grad_x f(x, y_\mu(x)) - \grad_x f(z_i, y^*(z_i))\|$:
\begin{equation*}
    \begin{split}
        \|\grad_x f(x, y_\mu(x)) - \grad_x f(z_i, y^*(z_i))\| &\leq L_g \|(x, y_\mu(x)) - (z_i, y^*(z_i))\| \\
        &\leq L_g (\|x - z_i\| + \|y_\mu(x) - y^*(z_i)\|) \\
        &\leq L_g (\|x - z_i\| + \|y_\mu(x) - y^*(x)\| + \|y^*(x) - y^*(z_i)\|) \\
        &\leq L_g (\|x - z_i\| +  L_y \mu + L_y \|x - z_i\|) \\
        &\leq L_g (\mu + L_y \mu + L_y \mu) = L_g (\mu+2L_y \mu),
    \end{split}
\end{equation*}
where the first inequality follows by $L_g$-Lipschitz continuity of $\grad_x f$ with respect to $w=(x,y)$, $\|y_\mu(x) - y^*(z_i)\| \leq \|y_\mu(x) - y^*(x)\| + \|y^*(x) - y^*(z_i)\|$ due to the triangle inequality, $\|y_\mu(x) - y^*(x)\| \leq L_y \mu$ by \cref{lem:properties-smoothing}(i), $\|y^*(x) - y^*(z_i)\| \leq L_y \|x - z_i\|$ by the Lipschitz continuity of $y^*$, and $\|x - z_i\| \leq \mu$ since $z_i \in \ball(x, \mu)$. 

Similarly, for the term $\|\grad_y f(x, y_\mu(x)) - \grad_y f(z_i, y^*(z_i))\|$, by the $L_g$-Lipschitz continuity of $\grad_y f$ and equivalent reasoning as above:
\begin{equation*}
\|\grad_y f(x, y_\mu(x)) - \grad_y f(z_i, y^*(z_i))\| \leq L_g (\|x - z_i\| + \|y_\mu(x) - y^*(z_i)\|) \leq L_g (\mu + 2L_y \mu).
\end{equation*}
For the spectral norm of $M_i$, we have $\|M_i\|_2 \leq L_y$. This follows from the fact that $M_i \in \partial y^*(z_i)$, and elements of the Clarke generalized Jacobian are convex combinations of limits of classical Jacobians. Particularly, since $\|\jac y^*(\tilde{x})\|_2 \leq L_y$ at all points $\tilde{x}$ at which $y^*(\tilde{x})$ is differentiable (as $y^*(\tilde{x})$ is $L_y$-Lipschitz), and since the spectral norm is convex and the bound is preserved under limits, $\|M_i\|_2 \leq L_y$ for all $M_i \in \partial y^*(z_i)$.

Combining these bounds:
\begin{align*}
&\|\grad_x f(x, y_\mu(x)) + M_i^\top \grad_y f(x, y_\mu(x)) - q_i\| \\
&\leq \|\grad_x f(x, y_\mu(x)) - \grad_x f(z_i, y^*(z_i))\| + \|M_i\|_2 \|\grad_y f(x, y_\mu(x)) - \grad_y f(z_i, y^*(z_i))\| \\
&\leq L_g (\mu + 2L_y \mu) + L_y \cdot L_g (\mu + 2L_y \mu) \\
&= (1 + L_y) L_g (1 + 2L_y) \mu.
\end{align*}
Taking norms:
\begin{align*}
\|\grad F_\mu(x) - \bar{q}\| &= \left\| \sum_{i=1}^k \lambda_i \left[ \grad_x f(x, y_\mu(x)) + M_i^\top \grad_y f(x, y_\mu(x)) - q_i \right] \right\| \\
&\leq \sum_{i=1}^k \lambda_i \|\grad_x f(x, y_\mu(x)) + M_i^\top \grad_y f(x, y_\mu(x)) - q_i\| \\
&\leq \sum_{i=1}^k \lambda_i (1 + L_y) L_g (1 + 2L_y) \mu \\
&= (1 + L_y) L_g (1 + 2L_y) \mu. \qedhere
\end{align*}
\end{proof}

\subsubsection{Proof of \texorpdfstring{\Cref{thm:convergencegoldsteinpoint}}{Theorem~\ref{thm:convergencegoldsteinpoint}}}
We now state and prove our main convergence result. We first state the result and prove it for $(\delta,\varepsilon)$-partial Goldstein stationary points and later for $(\delta,\varepsilon)$ Goldstein stationary points. We proceed in this way as the arguments differ. 

\begin{theorem}[Convergence to $(\delta,\varepsilon)$-partial Goldstein stationary point]\label{thm:partial_goldstein}
Let $x^R$ be chosen uniformly at random from the iterates $\{x_0, \ldots, x_{T-1}\}$ generated by \cref{alg:partial_zero_order}. Under \cref{ass:lipschitz,ass:bounded}, with step size $\alpha = \sqrt{2(\Delta + 2 L_f L_y \mu)/(T L_F \sigma^2_{\dx})}$ and smoothing parameter $\mu = \min (\delta, \varepsilon/(\sqrt{2}C_p)) < 1$, and $C_p = L_g(1 + L_y)L_y$, after $T = \mc{O}\!\left(\dx^{3/2}/(\mu\,\varepsilon^4)\right)$ iterations, $x^R$ is, in expectation, a $(\delta,\varepsilon)$-partial Goldstein 
    stationary point of $F$, that is, 
\[\E\left[ \min\left\{ \|g\| : g \in \goldsteinpartial F(x^R) \right\} \right] \leq \varepsilon\]
\end{theorem}
\begin{proof}
Define the smoothed objective $F_\mu(x) = f(x, y_\mu(x))$. Recall that, by \cref{lem:descent-Fsmooth}, $F_\mu(x)$ is $L_F$-smooth, with $L_F = L_g(1 + L_y)^2+\frac{k_1 L_f L_y \sqrt{\dx}}{\mu}$-Lipschitz gradient. By descent lemma for $F_\mu$, and since $x_{t+1} = x_t - \alpha g_t$ by \cref{alg:partial_zero_order}:
\begin{equation*}
    \begin{split}
        F_\mu(x_{t+1}) &\leq F_\mu(x_t) + \inner{\grad F_\mu(x_t)}{x_{t+1} - x_t} + \frac{L_F}{2}\norm{x_{t+1} - x_t}^2 \\
        &= F_\mu(x_t) - \alpha \inner{\grad F_\mu(x_t)}{g_t} + \frac{\alpha^2 L_F}{2}\norm{g_t}^2.
    \end{split}
\end{equation*}
Taking the conditional expectation given $x_t$:
\begin{equation}\label{p3_eq:descentcondexp_partial}
\E[F_\mu(x_{t+1}) \mid x_t] \leq F_\mu(x_t) - \alpha \langle \grad F_\mu(x_t), \E[g_t \mid x_t] \rangle + \frac{\alpha^2 L_F}{2}\E[\|g_t\|^2 \mid x_t].
\end{equation}
Define $\bar{g}_t = \E[g_t|x_t]$ and $b_t = \bar{g}_t - \grad F_\mu(x_t)$ with $\norm{b_t} \leq C_p\mu$ by \cref{lem:difference-estimate-gradfsmooth}. Rewriting the inner product using $\grad F_\mu = \bar{g}_t - b_t$:
\begin{equation*}
    \begin{split}
    \langle \grad F_\mu(x_t), \E[g_t \mid x_t] \rangle &=\langle\grad F_\mu(x_t),\bar{g}_t\rangle = \langle\bar{g}_t - b_t,\,\bar{g}_t\rangle \\
    &= \norm{\bar{g}_t}^2 - \langle b_t,\bar{g}_t\rangle \geq \frac{1}{2}\norm{\bar{g}_t}^2 - \frac{1}{2}\norm{b_t}^2 \geq \frac{1}{2}\norm{\bar{g}_t}^2 - \frac{1}{2}C_p^2\mu^2,
    \end{split}
\end{equation*}
where the first inequality uses $\langle a,b\rangle \leq \frac{1}{2}\norm{a}^2 + \frac{1}{2}\norm{b}^2$ for any real-valued vectors $a,b$, thus $\langle b_t,\bar{g}_t\rangle \leq \frac{1}{2}\norm{b_t}^2 + \frac{1}{2}\norm{\bar{g}_t}^2$, and the second inequality uses $\norm{b_t} \leq C_p\mu$ by \cref{lem:difference-estimate-gradfsmooth}. 

Further, we recall by \cref{lem:estimatorsquared} that $\E[\|g_t\|^2 \mid x_t] \leq k_2 \dx L_f^2 L_y^2 + L_f^2(1+2L_y) =: \sigma^2_{\dx}$, where $k_2$ is a constant. Substituting into \eqref{p3_eq:descentcondexp_partial}:
\[
    \E[F_\mu(x_{t+1})|x_t] \leq F_\mu(x_t) - \frac{\alpha}{2}\norm{\bar{g}_t}^2 + \frac{\alpha}{2}C_p^2\mu^2 + \frac{\alpha^2 L_F}{2}\sigma^2_{\dx}.
\]
Taking expectations (where $\E[\E[F_\mu(x_{t+1})|x_t]] = \E[F_\mu(x_{t+1})]$ by the law of total expectation), summing from $t=0$ to $T-1$, telescoping $\sum_{t=0}^{T-1}\E[F_\mu(x_t)-F_\mu(x_{t+1})]$ to $\E[F_\mu(x_0)] - \E[F_\mu(x_T)]$, dividing by $\alpha T/2$ and re-arranging for $\E[\|\bar{g}\|^2]$ yields 
\begin{equation*}
    \frac{1}{T} \sum_{t=0}^{T-1}\E[\norm{\bar{g}_t}^2] \leq \frac{2(\E[F_\mu(x_0)] - \E[F_\mu(x_{T})])}{\alpha T} + C_p^2\mu^2 + \alpha L_F\sigma^2_{\dx}.
\end{equation*}
We now bound the term $\E[F_\mu(x_0)] - \E[F_\mu(x_T)]$. Recall that $F_\mu(x) = f(x, y_\mu(x))$ and $F(x) = f(x, y^*(x))$. By \cref{pro:Fsmoothbound}, $|F_\mu(x) - F(x)| \leq L_f L_y \mu$. By \cref{ass:bounded}, $f(x_0, y^*(x_0)) - \inf_{x \in \R^{\dx}} f(x, y^*(x)) \leq \Delta$, so $F(x_0) - F^* \leq \Delta$, where $F^* = \inf_{x \in \R^{\dx}} F(x)$. Therefore:
\begin{align*}
    \E[F_\mu(x_0)] - \E[F_\mu(x_T)] &= \E[F_\mu(x_0) - F(x_0)] + \E[F(x_0) - F(x_T)] + \E[F(x_T) - F_\mu(x_T)] \\
    &\leq |\E[F_\mu(x_0) - F(x_0)]| + \E[F(x_0) - F(x_T)] + |\E[F(x_T) - F_\mu(x_T)]| \\
    &\leq L_f L_y \mu + \E[F(x_0) - F(x_T)] + L_f L_y \mu \\
    &= 2L_f L_y \mu + \E[F(x_0) - F(x_T)] \\
    &\leq 2L_f L_y \mu + F(x_0) - F^* \\
    &\leq 2L_f L_y \mu + \Delta =:\Delta_\mu,
\end{align*}
where we use that $\E[F(x_T)] \geq F^*$ since $F^*= \inf_x F(x)$ is the infimum of $F$. Substituting into the bound for $\E[\|\bar{g}\|^2]$:
\begin{equation}\label{p3_eq:avg-estimatebound-descent}
    \frac{1}{T} \sum_{t=0}^{T-1}\E[\norm{\bar{g}_t}^2] \leq \frac{2(2L_f L_y \mu + \Delta)}{\alpha T} + C_p^2\mu^2 + \alpha L_F\sigma^2_{\dx}.
\end{equation}

\noindent\emph{Choice of $\alpha$.} We choose the stepsize to balance the first and third terms of the bound in \eqref{p3_eq:avg-estimatebound-descent}. Setting the first and third term equal gives
\begin{equation*}
    \alpha = \sqrt{\frac{2(\Delta + 2 L_f L_y \mu)}{T L_F \sigma^2_{\dx}}},
\end{equation*}
and we have
\begin{equation}\label{eq:estimate_bound}
    \frac{1}{T}\sum_{t=0}^{T-1}\E[\norm{\bar{g}_t}^2] \leq 2\sqrt{\frac{2(\Delta + 2 L_f L_y \mu) L_F \sigma^2_{\dx}}{T}} + C_p^2\mu^2.
\end{equation}

\noindent\emph{Containment of $\E[g_t\mid x_t]$ in partial Goldstein subdifferential.}
Since $g_t = \grad_x f(x_t, y^*(x_t)) + H_t^\top \grad_y f(x_t, y^*(x_t))$ by \cref{alg:partial_zero_order}, and $H_t$ is an unbiased estimator of $\jac y_\mu(x_t)$ by \cref{lem:jacobianunbiasedestimator}, and $\jac y_\mu(x_t)$ lies in the Goldstein subdifferential $\goldsteinfullmu y^*(x_t)$ of $y^*(x_t)$, we have 
\begin{equation*}
    \E[H_t|x_t] = \jac y_\mu(x_t) \in \goldsteinfullmu y^*(x_t) = \co\left\{ \cup_{z \in \ball(x_t, \mu)} \partial y^*(z) \right\}
\end{equation*}
Since $\mu \leq \delta$, we have $\ball(x_t, \mu) \subseteq \ball(x_t, \delta)$, so $\goldsteinfullmu y^*(x_t) \subseteq \goldsteinfull y^*(x_t)$. Therefore $\jac y_\mu(x_t) \in \goldsteinfull y^*(x_t)$. By the definition of the partial Goldstein subdifferential,
\[
    \goldsteinpartial F(x) = \left\{ \grad_x f(x, y^*(x)) + M^\top \grad_y f(x, y^*(x)) : M \in \goldsteinfull y^*(x) \right\},
\]
and $\bar{g}_t=\E[g_t \mid x_t]$ has precisely this form with $M = \jac y_\mu(x_t) \in \goldsteinfull y^*(x_t)$. Hence
\begin{equation}\label{eq:containment}
    \bar{g}_t \in \goldsteinpartial F(x_t) \quad \text{for all } t=0,\ldots,T-1.
\end{equation}
As an immediate consequence:
\begin{equation}\label{eq:dist_bound}
    \min\!\big\{\norm{g} : g \in \goldsteinpartial F(x_t)\big\} \leq \norm{\bar{g}_t}.
\end{equation}

\noindent\emph{From the average bound to the output guarantee.}
Let $x^R$ be chosen uniformly at random from $\{x_0, \ldots, x_{T-1}\}$, and let $R\in\{0,\ldots,T-1\}$ be its corresponding time index. By~\eqref{eq:dist_bound}:
\begin{equation}\label{eq:exp_dist}
    \E\!\left[\min\!\big\{\norm{g} : g \in \goldsteinpartial F(x^R)\big\}\right] \leq \E[\norm{\bar{g}_{R}}].
\end{equation}
For the right-hand side, since $x^R$ is uniform over $\{0,\ldots,T{-}1\}$:
\begin{equation}\label{eq:uniform_avg}
    \E[\norm{\bar{g}_{R}}^2] = \frac{1}{T}\sum_{t=0}^{T-1}\E[\norm{\bar{g}_t}^2].
\end{equation}
Applying Jensen's inequality ($\sqrt{\cdot}$ is concave, so $\E[\sqrt{X}] \leq \sqrt{\E[X]}$):
\begin{equation}\label{eq:partialgoldstein_jensen}
    \E[\norm{\bar{g}_{R}}] = \E\!\left[\sqrt{\norm{\bar{g}_{R}}^2}\right] \leq \sqrt{\E[\norm{\bar{g}_{R}}^2]} = \sqrt{\frac{1}{T}\sum_{t=0}^{T-1}\E[\norm{\bar{g}_t}^2]}.
\end{equation}
Substituting \eqref{eq:estimate_bound} into~\eqref{eq:partialgoldstein_jensen} and combining with~\eqref{eq:exp_dist}:
\begin{equation}\label{eq:partial_combined}
    \E\!\left[\min\!\big\{\norm{g} : g \in \goldsteinpartial F(x^R)\big\}\right] \leq \sqrt{\frac{1}{T}\sum_{t=0}^{T-1}\E[\norm{\bar{g}_t}^2]} \leq \left(2\sqrt{\frac{2\Delta_\mu\,L_F\,\sigma^2_{\dx}}{T}} + C_p^2\mu^2\right)^{\!1/2}.
\end{equation}

\noindent\emph{Choice of $\mu$ and $T$ to guarantee $(\delta,\varepsilon)$-stationarity.} We require the right-hand side of~\eqref{eq:partial_combined} to be at most $\varepsilon$:
\begin{equation*}
    2\sqrt{\frac{2\Delta_\mu\,L_F\,\sigma^2_{\dx}}{T}} + C_p^2\mu^2 \leq \varepsilon^2.
\end{equation*}
We allocate $\varepsilon^2/2$ to each term.

\noindent\emph{Choice of $\mu$.} The constraint $C_p^2\mu^2 \leq \varepsilon^2/2$ requires $\mu \leq \varepsilon/(\sqrt{2}\,C_p)$. Further, we need to require $\mu \leq \delta$, to guarantee that $\jac y_\mu(x_t) \in \goldsteinfull y^*(x_t)$ as argued above. Hence we set
\[
    \mu = \min\!\left(\delta,\;\frac{\varepsilon}{\sqrt{2}\,C_p}\right).
\]

\noindent\emph{Choice of $T$.} It remains to ensure $2\sqrt{2\Delta_\mu\,L_F\,\sigma^2_{\dx}/T} \leq \varepsilon^2/2$, i.e.,
\begin{equation}\label{eq:T_partial}
    T \geq \frac{32\,\Delta_\mu\,L_F\,\sigma^2_{\dx}}{\varepsilon^4}.
\end{equation}
We simplify $\Delta_\mu$, $L_F$, and $\sigma^2_{\dx}$ in terms of $\mu$ and $\dx$: Since $\dx\ge 1$, $\sigma^2_{\dx} = k_2 \dx L_f^2 L_y^2 + L_f^2(1+2L_y) = \mc{O}(\dx)$, hence
\begin{equation*}
    \frac{32\,\Delta_\mu\,L_F\,\sigma^2_{\dx}}{\varepsilon^4} = \mc{O}\!\left(\frac{(1+\mu)(1+\sqrt{\dx}/\mu)\,\dx}{\varepsilon^4}\right) = O\!\left(\frac{\dx + \dx^{3/2}/\mu + \mu\dx + \dx^{3/2}}{\varepsilon^4}\right).
\end{equation*}
Under the assumption $\mu \leq 1$, all four terms are absorbed by $\dx^{3/2}/\mu$:
\begin{equation*}
    T = \mc{O}\!\left(\frac{\dx^{3/2}}{\mu\,\varepsilon^4}\right).
\end{equation*}

This yields the desired result. 
\end{proof}

\begin{theorem}[Convergence to $(\delta,\varepsilon)$-Goldstein stationary point]\label{thm:full_goldstein}
Let $x^R$ be chosen uniformly at random from the iterates $\{x_0, \ldots, x_{T-1}\}$ generated by \cref{alg:partial_zero_order}. Under \cref{ass:lipschitz,ass:bounded}, with step size $\alpha = \sqrt{2(\Delta + 2 L_f L_y \mu)/(T L_F \sigma^2_{\dx})}$ and smoothing parameter $\mu = \min (\delta, \varepsilon/(2C_f)) < 1$, and $C_p = L_g(1 + L_y)L_y$ and $C_f=(1+L_y)L_g(1+2L_y)$, after $T = \mc{O}\!\left(\dx^{3/2}/(\mu\,\varepsilon^4)\right)$ iterations, $x^R$ is, in expectation, a $(\delta,\varepsilon)$-Goldstein stationary point of $F$, that is,
    \[\E\left[ \min\left\{ \|g\| : g \in \goldsteinfull F(x^R) \right\} 
    \right] \leq \varepsilon.\]
\end{theorem}
\begin{proof}
Define the smoothed objective $F_\mu(x) = f(x, y_\mu(x))$. Recall that, by \cref{lem:descent-Fsmooth}, $F_\mu(x)$ is $L_F$-smooth, with $L_F = L_g(1 + L_y)^2+\frac{k_1 L_f L_y \sqrt{\dx}}{\mu}$-Lipschitz gradient. By descent lemma for $F_\mu$, and since $x_{t+1} = x_t - \alpha g_t$ by \cref{alg:partial_zero_order}:
\begin{equation*}
    \begin{split}
        F_\mu(x_{t+1}) &\leq F_\mu(x_t) + \inner{\grad F_\mu(x_t)}{x_{t+1} - x_t} + \frac{L_F}{2}\norm{x_{t+1} - x_t}^2 \\
        &= F_\mu(x_t) - \alpha \inner{\grad F_\mu(x_t)}{g_t} + \frac{\alpha^2 L_F}{2}\norm{g_t}^2.
    \end{split}
\end{equation*}
Taking the conditional expectation given $x_t$:
\begin{equation}\label{p3_eq:descentcondexp_full}
\E[F_\mu(x_{t+1}) \mid x_t] \leq F_\mu(x_t) - \alpha \langle \grad F_\mu(x_t), \E[g_t \mid x_t] \rangle + \frac{\alpha^2 L_F}{2}\E[\|g_t\|^2 \mid x_t].
\end{equation}
For the inner product term, we have
\begin{align*}
\langle \grad F_\mu(x_t), \E[g_t \mid x_t] \rangle &= \langle \grad F_\mu(x_t), \grad F_\mu(x_t) \rangle + \langle \grad F_\mu(x_t), \E[g_t \mid x_t] - \grad F_\mu(x_t) \rangle \\
&= \norm{\grad F_\mu(x_t)}^2 + \langle \grad F_\mu(x_t), \E[g_t \mid x_t] - \grad F_\mu(x_t) \rangle \\
&\geq \norm{\grad F_\mu(x_t)}^2 - \norm{\grad F_\mu(x_t)}\norm{\E[g_t \mid x_t] - \grad F_\mu(x_t)} \\
&\geq \norm{\grad F_\mu(x_t)}^2 - \norm{\grad F_\mu(x_t)}L_g(1+L_y)L_y\mu \\
&\geq \norm{\grad F_\mu(x_t)}^2 - \left(\frac{1}{2}(L_g(1+L_y)L_y\mu)^2 + \frac{1}{2}\norm{\grad F_\mu(x_t)}^2\right) \\
&= \frac{1}{2}\norm{\grad F_\mu(x_t)}^2 - \frac{1}{2}L_g^2(1+L_y)^2L_y^2\mu^2
\end{align*}
where the first inequality follows as by the Cauchy--Schwarz inequality $|\langle a,b\rangle|\leq \norm{a}\norm{b}$, hence $-\norm{a}\norm{b} \leq \langle a,b\rangle \leq \norm{a}\norm{b}$, the second inequality is due to $\norm{\E[g_t \mid x_t] - \grad F_\mu(x_t)} \leq L_g(1 + L_y)L_y\mu$ by \cref{lem:difference-estimate-gradfsmooth}, and the third inequality follows from Young's inequality $ab \leq \frac{a^2}{2} + \frac{b^2}{2}$ applied to $a = \norm{\grad F_\mu(x_t)}$ and $b = L_g(1 + L_y)L_y\mu$.

For the term $\frac{\alpha^2 L_F}{2}\E[\norm{g_t}^2 \mid x_t]$, we recall from \cref{lem:estimatorsquared} that 
\begin{equation*}
    \E[\norm{g_t}^2 \mid x_t] \leq k_2 \dx L_f^2L_y^2 + L_f^2(1+2L_y),
\end{equation*}
where we define $\sigma^2_{\dx} = k_2 \dx L_f^2L_y^2 + L_f^2(1+2L_y)$ as the short-hand notation. Substituting this bound into \eqref{p3_eq:descentcondexp_full}:
\begin{equation*}
    \E[F_\mu(x_{t+1}) \mid x_t] \leq F_\mu(x_t) - \frac{\alpha}{2}\norm{\grad F_\mu(x_t)}^2 + \frac{\alpha}{2}L_g^2(1 + L_y)^2L_y^2\mu^2 + \frac{\alpha^2 L_F}{2}\sigma^2_{\dx}.
\end{equation*}
Taking expectations (where $\E[\E[F_\mu(x_{t+1})|x_t]] = \E[F_\mu(x_{t+1})]$ by the law of total expectation), summing from $t=0$ to $T-1$, telescoping $\sum_{t=0}^{T-1}\E[F_\mu(x_t)-F_\mu(x_{t+1})]$ to $\E[F_\mu(x_0)] - \E[F_\mu(x_T)]$, dividing by $\alpha T/2$ and re-arranging for $\E[\|\grad F_\mu(x_t)\|^2]$ yields 
\begin{equation}\label{p3_eq:descent_full_bound}
    \frac{1}{T}\sum_{t=0}^{T-1} \E[\|\grad F_\mu(x_t)\|^2] \leq \frac{2(\E[F_\mu(x_0)] - \E[F_\mu(x_{T})])}{\alpha T} + C_p^2\mu^2 + \alpha L_F\sigma^2_{\dx}.
\end{equation}
We now bound the term $\E[F_\mu(x_0)] - \E[F_\mu(x_T)]$. Recall that $F_\mu(x) = f(x, y_\mu(x))$ and $F(x) = f(x, y^*(x))$. By \cref{pro:Fsmoothbound}, $|F_\mu(x) - F(x)| \leq L_f L_y \mu$. By \cref{ass:bounded}, $f(x_0, y^*(x_0)) - \inf_{x \in \R^{\dx}} f(x, y^*(x)) \leq \Delta$, so $F(x_0) - F^* \leq \Delta$, where $F^* = \inf_{x \in \R^{\dx}} F(x)$. Therefore:
\begin{align*}
    \E[F_\mu(x_0)] - \E[F_\mu(x_T)] &= \E[F_\mu(x_0) - F(x_0)] + \E[F(x_0) - F(x_T)] + \E[F(x_T) - F_\mu(x_T)] \\
    &\leq |\E[F_\mu(x_0) - F(x_0)]| + \E[F(x_0) - F(x_T)] + |\E[F(x_T) - F_\mu(x_T)]| \\
    &\leq L_f L_y \mu + \E[F(x_0) - F(x_T)] + L_f L_y \mu \\
    &= 2L_f L_y \mu + \E[F(x_0) - F(x_T)] \\
    &\leq 2L_f L_y \mu + F(x_0) - F^* \\
    &\leq 2L_f L_y \mu + \Delta,
\end{align*}
where we use that $\E[F(x_T)] \geq F^*$ since $F^*$ is the global minimum. Substituting into the bound for $\E[\|\grad F_\mu(x_t)\|^2]$:
\begin{equation}\label{p3_eq:avg-grad-bound-descent}
    \frac{1}{T}\sum_{t=0}^{T-1} \E[\|\grad F_\mu(x_t)\|^2] \leq \frac{2(2L_f L_y \mu + \Delta)}{\alpha T} + C_p^2\mu^2 + \alpha L_F\sigma^2_{\dx}.
\end{equation}

\noindent\emph{Choice of $\alpha$.} We choose the stepsize to balance the first and third terms of the bound in \eqref{p3_eq:avg-grad-bound-descent}. Setting the first and third term equal gives
\begin{equation*}
    \alpha = \sqrt{\frac{2(\Delta + 2 L_f L_y \mu)}{T L_F \sigma^2_{\dx}}},
\end{equation*}
and we have
\begin{equation}\label{p3_eq:descent_bound_constants}
   \frac{1}{T}\sum_{t=0}^{T-1} \E[\|\grad F_\mu(x_t)\|^2] \leq 2\sqrt{\frac{2(\Delta + 2 L_f L_y \mu) L_F \sigma^2_{\dx}}{T}} + C_p^2\mu^2.
\end{equation}

\noindent\emph{Distance of $\grad F_\mu(x_t)$ to Goldstein subdifferential.} By Lemma~\ref{lem:distancefullgoldstein}, for each $x_t$ there exists $\bar{q}_t \in \goldsteinfullmu F(x_t) \subseteq \goldsteinfull F(x_t)$ (since $\mu \leq \delta$) with $\norm{\grad F_\mu(x_t) - \bar{q}_t} \leq C_f\mu$. By the triangle inequality:
\begin{multline}\label{eq:dist_full}
    \min\!\big\{\norm{g} : g \in \goldsteinfull F(x_t)\big\} \leq \norm{\bar{q}_t} = \norm{\grad F_\mu(x_t) + \bar{q}_t - \grad F_\mu(x_t)} \\
    \leq \norm{\grad F_\mu(x_t)} + \norm{\bar{q}_t - \grad F_\mu(x_t)} \leq \norm{\grad F_\mu(x_t)} +C_f\mu.
\end{multline}

\noindent\emph{From the average bound to the output guarantee.}
Let $x^R$ be chosen uniformly at random from $\{x_0, \ldots, x_{T-1}\}$. Taking expectations in~\eqref{eq:dist_full}:
\begin{equation}\label{eq:full_exp}
    \E\!\left[\min\!\big\{\norm{g} : g \in \goldsteinfull F(x^R)\big\}\right] \leq \E[\norm{\grad F_\mu(x^R)}] + C_f\mu.
\end{equation}
It remains to bound $\E[\norm{\grad F_\mu(x^R)}]$. Since $x^R$ is uniform over $\{x_0,\ldots,x_{T-1}\}$:
\begin{equation}\label{eq:uniform_selection}
    \E[\norm{\grad F_\mu(x^R)}^2] = \frac{1}{T}\sum_{t=0}^{T-1}\E[\norm{\grad F_\mu(x_t)}^2].
\end{equation}

We observe
\begin{equation}\label{eq:jensen_full}
    \begin{split}
    \E[\norm{\grad F_\mu(x^R)}] &= \frac{1}{T}\sum_{t=0}^{T-1}\E[\norm{\grad F_\mu(x_t)}] = \frac{1}{T}\sum_{t=0}^{T-1}\E\!\left[\sqrt{\norm{\grad F_\mu(x_t)}^2}\right] \\
    &\leq \frac{1}{T}\sum_{t=0}^{T-1}\sqrt{\E[\norm{\grad F_\mu(x_t)}^2]} \\
    &\leq \sqrt{\frac{1}{T}\sum_{t=0}^{T-1}\E[\norm{\grad F_\mu(x_t)}^2]}, 
    \end{split}
\end{equation}
where the first inequality applies Jensen's inequality $\E[\sqrt{Z}] \leq \sqrt{\E[Z]}$ to the concave function $\sqrt{\cdot}$ and $Z = \norm{\grad F_\mu(x_t)}^2$ for each $t$, and the second inequality applies the QM-AM inequality $\frac{1}{T}\sum_{t} a_t \leq \sqrt{\frac{1}{T}\sum_{t} a_t^2}$ for non-negative $a_t$ with $a_t = \sqrt{\E[\norm{\grad F_\mu(x_t)}^2]}$.

Substituting the bound on $\frac{1}{T}\sum_{t=0}^{T-1}\E[\norm{\grad F_\mu(x_t)}^2]$ from~\eqref{p3_eq:descent_bound_constants} into~\eqref{eq:jensen_full}:
\begin{equation}\label{eq:jensen_substituted}
    \E[\norm{\grad F_\mu(x^R)}] \leq \left(2\sqrt{\frac{2(\Delta + 2L_fL_y\mu)\,L_F\,\sigma^2_{\dx}}{T}} + C_p^2\mu^2\right)^{\!1/2}.
\end{equation}
Combining~\eqref{eq:full_exp} and~\eqref{eq:jensen_substituted}:
\begin{equation}\label{eq:full_combined}
    \E\!\left[\min\!\big\{\norm{g} : g \in \goldsteinfull F(x^R)\big\}\right] \leq \left(2\sqrt{\frac{2(\Delta + 2L_fL_y\mu)\,L_F\,\sigma^2_{\dx}}{T}} + C_p^2\mu^2\right)^{\!1/2} + C_f\mu.
\end{equation}

To qualify as a $(\delta,\varepsilon)$-stationary point, we require the right-hand side of~\eqref{eq:full_combined} to be at most $\varepsilon$:
\[
    \left(2\sqrt{\frac{2\Delta_\mu\,L_F\,\sigma^2}{T}} + C_p^2\mu^2\right)^{\!1/2} + C_f\mu \leq \varepsilon.
\]
This involves two contributions: $C_f\mu$ from the triangle inequality~\eqref{eq:full_exp}, and the the square root term from $\grad F_\mu$. We allocate $\varepsilon/2$ to each.

\noindent\emph{Choice of $\mu$.} The constraint $C_f\mu \leq \varepsilon/2$ requires $\mu \leq \varepsilon/(2C_f)$. We also need $\mu \leq \delta$ for the subdifferential containment $\goldsteinfullmu F(x) \subseteq \goldsteinfull F(x)$. Therefore we set
\begin{equation}\label{eq:mu_full}
    \mu = \min\!\left(\delta,\; \frac{\varepsilon}{2C_f}\right).
\end{equation}
With this choice, the term $C_p^2\mu^2$ satisfies $C_p^2\mu^2 \leq C_p^2\varepsilon^2/(4C_f^2) < \varepsilon^2/16$, as
\begin{equation*}
    \frac{C_f}{C_p} = \frac{(1+L_y)L_g(1+2L_y)}{L_g(1+L_y)L_y} = \frac{1+2L_y}{L_y} > 2,
\end{equation*}
so $C_p < C_f/2$ strictly. If we can ensure $2\sqrt{\frac{2(\Delta + 2 L_f L_y \mu)\,L_F\,\sigma^2_{\dx}}{T}} \leq \frac{\varepsilon^2}{8}$, it follows that the two terms in the square root contribute at most $\sqrt{\varepsilon^2/16 + \varepsilon^2/8} < \varepsilon/2$.

\noindent\emph{Choice of $T$.} It remains to ensure
\[
    2\sqrt{\frac{2(\Delta + 2 L_f L_y \mu)\,L_F\,\sigma^2_{\dx}}{T}} \leq \frac{\varepsilon^2}{8}, \qquad \text{i.e.,} \qquad T \geq \frac{512\,(\Delta + 2 L_f L_y \mu)\,L_F\,\sigma^2_{\dx}}{\varepsilon^4}.
\]
We simplify $\Delta_\mu$, $L_F$, and $\sigma^2_{\dx}$ in terms of $\mu$ and $\dx$: Since $\dx\ge 1$, $\sigma^2_{\dx} = k_2 \dx L_f^2 L_y^2 + L_f^2(1+2L_y) = \mc{O}(\dx)$, hence
\begin{equation*}
    \frac{512\,\Delta_\mu\,L_F\,\sigma^2_{\dx}}{\varepsilon^4} = \mc{O}\!\left(\frac{(1+\mu)(1+\sqrt{\dx}/\mu)\,\dx}{\varepsilon^4}\right) = O\!\left(\frac{\dx + \dx^{3/2}/\mu + \mu\dx + \dx^{3/2}}{\varepsilon^4}\right).
\end{equation*}
Under the assumption $\mu \leq 1$, all four terms are absorbed by $\dx^{3/2}/\mu$:
\begin{equation*}
    T = \mc{O}\!\left(\frac{\dx^{3/2}}{\mu\,\varepsilon^4}\right).
\end{equation*}
This yields the desired result.
\end{proof}

\end{document}